\documentclass{article}
\pdfoutput=1
\usepackage[utf8]{inputenc}
\usepackage{amsmath,amsthm, amssymb, amsfonts}
\usepackage{bbm}
\usepackage{enumerate}
\usepackage{kantlipsum}
\usepackage{wrapfig}
\usepackage{paralist}
\usepackage{wasysym}
\usepackage{verbatim}
\usepackage{lscape}
\usepackage{fullpage}
\usepackage{hyperref}
\usepackage{graphicx}
\usepackage{subcaption}
\usepackage{wrapfig}
\usepackage{mathtools}
\usepackage{mathrsfs}
\usepackage{dsfont}
\usepackage{units}
\usepackage[]{algorithm2e}
\usepackage{ltablex}
\usepackage{varwidth}
\usepackage{changepage}
\usepackage{xcolor}

\allowdisplaybreaks[1]
\usepackage{sistyle} 
\newcommand\bSI[1]{{\small[\SI{}{#1}]}}
\makeatletter
\newlength\unitwdth
\newlength\numwdth
\newlength\tdima
\newcommand\SIdescr[2]{
    \setlength\tdima{\linewidth}
    \addtolength\tdima{\@totalleftmargin}
    \addtolength\tdima{-\dimen\@curtab}
    \addtolength\tdima{-\unitwdth}
    \addtolength\tdima{-\numwdth}
    \parbox[t]{\tdima}{
        #1
        \leaders\hbox{$\m@th\mkern \@dotsep mu\hbox{\tiny.}\mkern \@dotsep mu$}
        \hfill
        \ifhmode\strut\fi
        \makebox[0pt][l]{
            \makebox[\unitwdth][l]{}
            \makebox[\numwdth][r]{#2}}}}
\makeatother
\newcommand{\pou}{\Psi^{(j,\tau,N,s)}}

\newcommand{\Bcal}{\mathcal{B}}

\newcommand{\NN}{\mathcal{NN}}
\newcommand{\Realization}{{R}_{\varrho}}
\DeclarePairedDelimiterX{\bweights}[1]{\lVert}{\rVert_{\mathrm{max}}}{#1}

\let\emptyset\varnothing

\newcommand{\CalD}{\mathcal{D}}
\newcommand{\CalC}{\mathcal{C}}

\newcommand{\R}{\mathbb{R}}
\newcommand{\N}{\mathbb{N}}

\newcommand{\Z}{\mathbb{Z}}

\usepackage{makecell}

\newcommand{\bmat}[2]{\left[ \begin{array}{#1} #2 \end{array} \right]}

\let\emptyset\varnothing

\DeclareDocumentCommand{\Wkp}{ O{k} O{p} O{\Omega}}{{W^{#1,#2}(#3)}}
\DeclareDocumentCommand{\tWkp}{ O{k} O{p} O{\Omega}}{{\wtilde W^{#1,#2}(#3)}}
\DeclareDocumentCommand{\Wkpm}{ O{k} O{p} O{\Omega} O{m}}{{W^{#1,#2}(#3;\,\R^{#4})}}
\DeclareDocumentCommand{\Wkpd}{ O{k} O{p} O{\Omega} m}{W^{#1,#2}(#3;d#4)}
\DeclareDocumentCommand{\Lp}{ O{p} O{\Omega}}{L^{#1}(#2)}
\DeclareDocumentCommand{\Lpd}{ O{p} O{\Omega}}{L^{#1}(#2;dx)}

\DeclareDocumentCommand{\Lip}{ O{\Omega}}{\Cns[0][1][#1]}

\newcommand{\Linf}{{L^{\infty}(\Omega)}}
\newcommand{\Linfc}{{L^{\infty}(\cube^d)}}

\newcommand{\cube}{\intervalo{0}{1}}

\newcommand{\As}{A_{\text{sum}}}
\newcommand{\Phis}{\Phi_{\text{sum}}}

\newcommand{\PhiPs}{\Phi_{P,\eps}}
\newcommand{\epsin}{\intervalo{0}{\nicefrac{1}{2}}}

\renewcommand{\epsilon}{\varepsilon}
\newcommand{\eps}{\varepsilon}
\newcommand{\alplus}{\hspace{5mm}}

\renewcommand{\rho}{\varrho}

\newcommand{\MNd}{\{0,\ldots,N\}^d}

\newcommand{\expl}[1]{\text{\scriptsize{(#1)}}}

\newcommand\numberthis{\addtocounter{equation}{1}\tag{\theequation}}

\newcommand{\wtilde}{\widetilde}

\newcommand{\intervalo}[2]{\left(#1,#2\right)}

\newcommand{\Fndp}{\mathcal{F}_{n,d,p}}

\newcommand{\apmult}{\widetilde\times}

\newcommand{\act}[1]{R_\rho(#1)}
\newcommand{\actbig}[1]{R_\rho\big(#1\big)}

\DeclareDocumentCommand{\icouple}{O{B_0} O{B_1} O{\theta} O{p}}{\left(#1,#2\right)_{#3,#4}}

\definecolor{green}{rgb}{0.2,0.6,0.15}

\DeclareMathOperator{\ran}{Range}

\DeclarePairedDelimiter{\ceil}{\lceil}{\rceil}
\DeclarePairedDelimiter{\floor}{\lfloor}{\rfloor}
\DeclarePairedDelimiterX{\norm}[1]{\lVert}{\rVert}{#1}
\DeclarePairedDelimiterX{\pabs}[1]{\lvert}{\rvert}{#1}

\theoremstyle{definition}
\newtheorem{definition}{Definition}[section]

\theoremstyle{plain}
\newtheorem{remark}[definition]{Remark}
\newtheorem*{remark*}{Remark}
\newtheorem{theorem}[definition]{Theorem}

\newtheorem{lemma}[definition]{Lemma}

\newtheorem{corollary}[definition]{Corollary}
\newtheorem{proposition}[definition]{Proposition}
\newtheorem{observation}[definition]{Observation}
\newtheorem*{overview*}{Overview of our proof strategy}

\expandafter\let\expandafter\oldproof\csname\string\proof\endcsname
\let\oldendproof\endproof
\renewenvironment{proof}[1][\proofname]{%
	\oldproof[{\bf #1 }]%
}{\oldendproof}

\renewcommand{\thefootnote}{\alph{footnote}}

\newcommand{\pp}{\mathrm{poly}_{m}}

\newcommand{\conc}{{\raisebox{2pt}{\tiny\newmoon} \,}}

\numberwithin{equation}{section}

\definecolor{darkcandyapplered}{rgb}{0.64, 0.0, 0.0}

\newcommand{\drate}{\mu\max\{0, k-\tau\}}
\newcommand{\dratej}[1]{\mu\max\{0, #1-\tau\}}

\newcommand{\mesh}{[-\eps^{-\theta},\eps^{-\theta}]\cap\eps^{\nu}\mathbb{Z}}

\title{Approximation Rates for Neural Networks \\ with Encodable Weights in Smoothness Spaces}
\author{Ingo Gührin$\text{g}^{*\dagger}$ \and Mones Rasla$\text{n}^{*\dagger}$ }

\DeclareDocumentCommand{\Micod}{ O{\mathcal{B}} O{\rho} O{C_0} O{\mathcal{C}} O{\mathcal{D}} O{\eps}}{M^{#1,#2,#3}_{#6}(#4,#5)}
\makeatletter
\def\blfootnote{\xdef\@thefnmark{}\@footnotetext}
\makeatother

\begin{document}
\maketitle
\blfootnote{\hspace{-1em}$~^*$Both authors contributed equally.}
\blfootnote{\hspace{-1.4em}$~^\dagger$Institute of Mathematics, Technical University of Berlin,  Stra\ss{}e des 17.~Juni 136, 10623 Berlin, Germany; \linebreak E-Mail: {$\{$\texttt{guehring, raslan$\}$@math.tu-berlin.de}}}

\renewcommand{\thefootnote}{\arabic{footnote}}
\begin{abstract}
We examine the necessary and sufficient complexity of neural networks to approximate functions from different smoothness spaces under the restriction of encodable network weights. Based on an entropy argument, we start by proving lower bounds for the number of nonzero encodable weights for neural network approximation in Besov spaces, Sobolev spaces and more. These results are valid for all sufficiently smooth activation functions. Afterwards, we provide a unifying framework for the construction of approximate partition of unities by neural networks with fairly general activation functions. This allows us to approximate localized Taylor polynomials by neural networks and make use of the Bramble-Hilbert Lemma. Based on our framework, we derive almost optimal upper bounds in higher-order Sobolev norms. This work advances the theory of approximating solutions of partial differential equations by neural networks.
\end{abstract}


\textbf{MSC classification (2010)}:  35A35, 41A25, 41A46, 46E35, 68T05

\section{Introduction}
Deep learning algorithms have lately shown promising results for dealing with classical mathematical problems, such as the solution of \emph{partial differential equations (PDEs)}, see for instance \cite{lagaris1998artificial,weinan2018deep,han2018solving,han2020derivativefree,sirignano2018dgm,weinan2017deep,DeepXDE,MachLearningSPDEJentzen,elbrachter2018dnn, beck2018solving,JentzenKolmogorov,schwab2018deep,grohs2018proof,raslan2019parametric,GeiPM2020,PetersenTransport}. 
In this work, we investigate the necessary and sufficient number of non-zero, encodable\footnote{i.e., representable by a bit-string of moderate length} weights for a vanilla feedforward neural network to approximate functions that are particularly relevant for the solution of PDEs. Notable works in this direction for neural networks with the ReLU (rectified linear unit) activation function
are \cite{guhring2019error,OPS19_811}. Due to the limited regularity of the ReLU, one is only able to derive approximation rates with respect to first-order Sobolev norms. However, in order to appropriately approximate solutions of PDEs of higher-order (i.e., $\geq 3$), approximation rates with respect to higher-order Sobolev norms are required. As an example, consider the \emph{Dirichlet problem for the biharmonic operator $\Delta^2$} (see e.g.~\cite{Ciarlet}) on some domain $\Omega\subset \R^d$, a typical fourth-order problem, which is given by
\begin{equation} \label{eq:FourthOrder}
   - \Delta^2 u = f,\quad\text{on }\Omega\quad+\text{boundary conditions}.
\end{equation}
In its weak formulation, this operator equation is uniquely solvable in some subspace $V$ (incorporating the boundary conditions) of the Sobolev space $\Wkp[2][2][\Omega]$. Additionally (see \cite[Section 6]{Ciarlet}), typical solutions $u$ of \eqref{eq:FourthOrder} are even in the Sobolev space $\Wkp[n][2][\Omega]$ for some $n\geq 3$. This motivates studying approximations of Sobolev-regular functions $f\in\Wkp[n][p]$ by neural networks in higher-order Sobolev norms.
In this paper, we make the following two contributions:

\subsubsection*{I. General Lower Bounds based on Entropy Arguments}

Let $\mathcal{C}\subset \mathcal{D}$ be two function spaces. 
We will lower bound the necessary number for nonzero, encodable weights of neural network approximations of functions from $\mathcal{C}$ with respect to the norm in $\mathcal{D}$. Our notion of a lower bound for the number of nonzero, encodable weights can be summarized as follows:

\emph{For some $\gamma>0$ (depending on $\mathcal{C}$ and $\mathcal{D}$) we have: If for every $\eps>0$ there exists some $M_\eps\in \N$ such that every $f\in \mathcal{C}$ can be $\eps$-approximated by a neural network $\Phi_{\eps,f}$ (i.e.,\ $\|f- \Phi_{\eps,f}\|_{\mathcal{D}}\leq \eps$) with $M_\eps$ nonzero, encodable weights, then (up to a logarithmic factor and for some constant $C$) $M_{\eps}\geq C\eps^{-\gamma}$.}

In~\cite{petersen2017optimal}, the concept of the \emph{$\epsilon$-entropy $H_\eps(\mathcal{C},\mathcal{D})$} was used to derive lower bounds for $M_\eps$ for specific choices of $\mathcal{C}$ and $\mathcal{D}$. In~Theorem~\ref{thm:lower_bound_enc} we generalize that approach to a wide range of function spaces. In detail, we show that every lower bound on the \emph{$\epsilon$-entropy $H_\eps(\mathcal{C},\mathcal{D})$} of the unit ball of $\mathcal{C}$ with respect to $\| \cdot \|_{\mathcal{D}}$ can directly be transferred to a lower bound on the number of nonzero, encodable weights of an approximating neural network. Concretely, if  $H_\eps(\mathcal{C},\mathcal{D})\geq C\eps^{-\gamma},$ then $M_\eps \geq C\eps^{-\gamma}/\log_2(1/\eps)$. 
Since the activation function $\rho$ determines the smoothness of $\Phi_{\eps,f}$ we only have the natural requirement that $\rho$ is smooth enough such that $\Phi_{\eps,f}\in \mathcal{D}$.
        
Since lower bounds on the $\eps$-entropy are well-studied for a variety of classical function spaces\footnote{see for instance \cite{triebel1978interpolation,edmunds_triebel_1996}}, we give a nonexhaustive list of concrete lower complexity bounds in Corollary~\ref{cor:sobolev_low_bound_enc} for Sobolev and Besov spaces. Appositely to the upper bounds that we present below, we state the following special instance of these results: 
For $\mathcal{C}=W^{n,p}(\Omega)$ and $\mathcal{D}=W^{k,p}(\Omega)$ with $n,k\in\N, n>k$ and $1\leq p\leq\infty$ we have $M_\eps\geq C\eps^{-d/(n-k)}/\log_2(1/\eps)$.

\subsubsection*{II. Almost Optimal Upper Bounds in Sobolev Spaces For a Wide Class of Activation Functions} 

We build an abstract, unifying framework which allows to approximate localized Taylor polynomials by neural networks with a wide class of activation functions. This proof strategy was originally used in~\cite{yarotsky2017error} for ReLU neural networks in $L^p$-norms and generalized to first order Sobolev norms in~\cite{guhring2019error}. Those works  heavily rely on the ReLU activation function which allows for the construction of an exact partition of unity (PU). However, constructing localized bump functions that together form a PU by neural networks with general activation function is highly-nontrivial and can, in general, only be done approximately. This means that the localizing bump functions are not compactly supported anymore and their sum only approximates one. We formulate conditions on the asymptotic behavior of the activation function under which such a construction becomes possible in higher-order Sobolev spaces. For this we derive three distinct categories of PUs splitting the domain $(0,1)^d$ into $(N+1)^d$ patches with diameter $1/N$.
\begin{itemize}
    \item \emph{Exact PU}: The $(N+1)^d$ localizing bump functions are compactly supported on the corresponding patch and the sum of the bumps equals one.
    \item \emph{Exponential PU}: For $N\to\infty$, the bumps converge exponentially fast in $N$ towards an exact PU.  
    \item \emph{Polynomial PU}: For $N\to\infty$, the bumps converge with polynomial speed in $N$ towards an exact PU. 
\end{itemize}
In other words, with increasing refinement of the partition the approximate PUs converge towards an exact PU and are categorized by their convergence speed. 

Based on the above categorization, we consider $\eps$-approximations of functions from the unit ball in $\mathcal{C}=\Wkp[n][p][\cube^d]$ where the distance is measured in $\mathcal{D}=\Wkp[k][p][\cube^d]$ norms ($n\in \N_{\geq k+1},$ $k\in \N$ and $1\leq p\leq \infty$) and derive for each case different approximation rates. We demonstrate this for three representative examples.
\begin{itemize}
    \item The \emph{rectified power unit (RePU)} of order $j\in\N_{\geq 2},$ given by $\mathrm{ReLU}^j,$ allows for the construction of exact PUs. In this case, for every $k\in\{0,\dots,j\},$ we need at most $C\eps^{-d/(n-k)}$ non-zero weights. 
    \item The \emph{softplus function}, given by $\ln(1+e^x),$ allows for the construction of exponential PUs. In this case, for  $k\in\N_0,$ and arbitrary $\mu>0,$ we need at most \begin{align*}
       \begin{cases} C\eps^{-d/(n-k)}, \quad &\text{if } k\leq 1, \\ C\eps^{-d/(n-k-\mu)}, \quad &\text{if } k\geq 2,\end{cases}
    \end{align*} non-zero weights.
    \item The \emph{inverse square root linear unit}, given by $\mathbbm{1}_{[0,\infty)} x+ \mathbbm{1}_{(-\infty,0)} \frac{x}{\sqrt{1+x^2}}$, allows for the construction of polynomial PUs. In this case, for $k\in\{0,1\},$ we need at most $C\eps^{-d/(n-k)}$ non-zero weights.
    
    Generally speaking, in the case of polynomial PUs, we are only able to show approximation rates in smoothness norms of a restricted order, depending on the asymptotic behavior of the underlying activation function. We describe the reasons for this issue in more detail in Section~\ref{subsec:MainRes}. 
\end{itemize}
 In all cases the depth of the constructed networks is constant (i.e.\ accuracy-independent) and greater than two.
Afterwards, we additionally show that the weights of $\Phi_{\eps,f}$ can be encoded by $C\log_2(1/\eps)$ bits which guarantees that the approximation complexity is not hidden in weights carrying arbitrarily complex information.

\vspace{1em}
\noindent As already outlined in~\cite[Section 1.4]{guhring2019error}, we observe in both, lower and upper bounds, a \emph{trade-off} between the complexity of the approximating neural networks and the order of the approximation norm: A higher order of $k$ requires neural networks with asymptotically more nonzero weights. Additionally, up to a log-factor (and in some cases up to $\mu>0$), our upper bounds are \emph{tight} if we only allow encodable weights.

\subsubsection*{Related Work}\label{subsec:RelWork}
The universal approximation theorem \cite{cybenko1989approximation,hornik1991approximation} is often regarded as the starting point of approximation theory for neural networks. It shows that every continuous function defined on a compact domain can be uniformly approximated by shallow neural networks under some assumptions on the activation function. Extensions of this theorem (see \cite[Section 4]{pinkus1999approximation} and the references therein) also take derivatives into account. In more detail, it has been established that shallow neural networks with sufficiently regular activation function and unrestricted width are dense in the space $C^m,$ where $m\in \N$. The existence of an activation function such that restricted width and depth networks are universal is shown in~\cite{Maiorov1999LowerBounds} and an explicit activation function based on the countability of the rational numbers with that property is constructed in~\cite{guliyev2018approximation}. For ReLU networks with restricted width and unbounded depth universality is established in~\cite{Kidger2019UniversalAW}.

The necessary and sufficient complexity of (higher-order) sigmoidal neural network approximations for (piecewise) smooth functions has been studied in~\cite{Barron1994,Mhaskar:1996:NNO:1362203.1362213,  bolcskei2017optimal,ohnELU}. The results in~\cite{Mhaskar:1996:NNO:1362203.1362213} are derived by approximating global (not localized) polynomials with degree increasing concurrently with the approximation accuracy. Our results include these approximation rates as a special case based on an alternative proof strategy. The ansatz in~\cite{Mhaskar:1996:NNO:1362203.1362213} can be used for $C^\infty$ activation functions with non vanishing derivatives at some point to obtain network approximations with constant depth and increasing width. Vanishing derivatives of the activation function need to be compensated by increasing depth in order to construct polynomials of increasing degree. This approach is utilized in~\cite{RePURates,ChebNet}, where approximations of weighted $L^2$-spaces by RePU-neural networks are derived\footnote{which are able to represent polynomials with zero error}. The function spaces considered therein can be efficiently described by non-localized (Jacobi or Chebychev) polynomials. Complexity bounds for ReLU neural networks based on localized polynomial approximation can be found in~\cite{yarotsky2017error,schmidt2017nonparametric,petersen2017optimal,suzuki2018adaptivity,ohnELU}. The upper bounds in~\cite[Thm.\ 1]{yarotsky2017error} are covered by our framework as a special case. In~\cite{petersen2017optimal}, localization is achieved by approximating characteristic functions. Our notion of PUs is general enough to include this approach but we focus on different function classes. Localization by means of wavelet approximations on manifolds is utilized in~\cite{ShaCC2015provableAppDNN} and by means of general affine systems in~\cite{bolcskei2017optimal}. The approximation error in all of these papers is measured with respect to $L^p$-norms. Only the papers~\cite{bolcskei2017optimal,petersen2017optimal} consider the restriction of encodable weights.

In this paper we are primarily interested in the approximation of functions with respect to Sobolev norms. In this direction, we mention two works, which examine the approximation capabilities of ReLU-neural networks with respect to $W^{1,p}$ norms. The paper \cite{guhring2019error} derives lower complexity bounds based on a VC dimension argument for unrestricted neural network weights (similar to the one presented in \cite{yarotsky2017error}) and upper bounds based on the emulation of localized polynomials for continuous, piecewise linear activation functions. These upper bounds are included in our results as a special case. In~\cite{OPS19_811} approximation rates were derived by re-approximating finite elements. None of these papers examine neural networks with encodable weights. 

We conclude this section by giving an overview of further works that introduce different types of PUs. An approach which is similar to ours for functions of sigmoidal type has been used in~\cite{COSTARELL1, COSTARELLI2, COSTARELLI3}. There, approximate bumps are constructed from differences of scaled and shifted sigmoidals. 
The key difference is that for a fixed patch the contributions of the neighboring approximate bump functions do not decrease with the number of patches $N$ going to infinity which is an important factor in our construction. In~\cite{PartitionPaper}, characteristic functions $\chi_\mathrm{p}$ for each patch are $L^\infty$-approximated in order to achieve localization. However, in this work, the Heaviside function is used as an activation function in the first layer (followed by a different activation function in the next layer), which is not transferable to our work, since it prevents higher order Sobolev approximations.


\subsubsection*{Outline}

After having introduced the necessary terminology for neural networks in Section~\ref{sec:Terminology}, we start by proving general lower complexity bounds in Section~\ref{sec:Lower}. In~Section~\ref{sec:main}, we derive almost optimal upper approximation rates for neural networks with fairly general activation functions. We describe the necessary ingredients for these results in Section \ref{subsec:IngredientI} and \ref{subsec:IngredientII} before outlining the main results as well as the underlying proof strategy in Section \ref{subsec:MainRes}. The proofs of the two main results in this section, Proposition~\ref{prop:main} and Theorem~\ref{thm:main}, can be found in Appendix~\ref{app:ProofMainProp} and Appendix~\ref{app:Encod}, respectively. To not interrupt the flow of reading, the notation section, basic facts about Sobolev spaces and basic operations one can perform with neural networks have been deferred to Appendices \ref{app:Not}-\ref{app:NNCalc}, respectively. An analysis of the PU-properties of many practically used activation functions can be found in Appendix \ref{sec:activation_admissibility}.

\section{Neural Networks with Encodable Weights: Terminology} \label{sec:Terminology}

We start by formally introducing neural networks closely sticking to the notions introduced in \cite{petersen2017optimal}. In the following, we will distinguish between a \emph{neural network}
as a structured set of weights and the associated function implemented by the network, called its \emph{realization}.
Towards this goal, let us fix numbers $L, d=N_0, N_1, \dots, N_{L} \in \N$.
\begin{itemize}
    \item  A family $\Phi = \big( (A_\ell,b_\ell) \big)_{\ell = 1}^L$ of matrix-vector tuples of the form $A_\ell \in  \R^{N_{\ell}, N_{\ell-1}}$ and $b_\ell \in \R^{N_\ell}$ is called \emph{neural network}.
    \item We refer to the entries of $A_\ell,b_\ell$ as the weights of $\Phi$ and call $M(\Phi)\coloneqq \sum_{\ell=1}^L \left(\|A_\ell\|_0+\|b_\ell\|_0 \right)$ its  \emph{number of nonzero weights}, 
 $L = L(\Phi)$ its \emph{number of layers} and we call $N_\ell$ the \emph{number of neurons in layer $\ell$}.
 \item We denote by $d\coloneqq N_0$ the \emph{input dimension} of $\Phi$ and by $N_L$ the \emph{output dimension}. 
 \item Moreover, we set  $$\|\Phi\|_{\max}\coloneqq \max_{\ell=1,\dots,L} \max_{\substack{i=1,\dots, N_\ell \\ j=1,\dots,N_{\ell-1}}} \max\{|(A_\ell)_{i,j}|,|(b_\ell)_i|\},$$ 
 which is the \emph{maximum absolute value of all weights}.
 \item For defining the realization of a network $\Phi = \big( (A_\ell,b_\ell) \big)_{\ell=1}^L,$ we additionally fix an \emph{activation function} $\varrho:\R\to \R$ and a set $\Omega\subset \R^d$.
The \emph{realization of the network} $\Phi = \big( (A_\ell,b_\ell) \big)_{\ell=1}^L$
is the function
\begin{align*}
  R_\varrho \left( \Phi \right) :
  \Omega \to \R^{N_L} , \ \
  x \mapsto x_L \, ,
\end{align*}
where $x_L$ results from the following scheme:
\begin{equation*}
  \begin{split}
    x_0 &\coloneqq x, \\
    x_{\ell} &\coloneqq \varrho(A_{\ell} \, x_{\ell-1} + b_\ell),
    \quad \text{ for } \ell = 1, \dots, L-1,\\
    x_L &\coloneqq A_{L} \, x_{L-1} + b_{L},
  \end{split}
\end{equation*}
and where $\varrho$ acts componentwise. 
\item We denote by $\mathcal{NN}_{\varrho}^{d}$ the \emph{set of all $\varrho$-realizations of neural networks with input dimension $d$ and output dimension $1$}.\footnote{In the following we will  denote by \emph{($\varrho$-)neural networks} both neural networks and their corresponding realizations as long it is clear from the context what is meant.}
\end{itemize}

\subsubsection*{Encodability}
In the following, we study neural networks with \emph{encodable} weights. This information-theoretic viewpoint has already been examined in~\cite{bolcskei2017optimal,petersen2017optimal} and is motivated by the observation that on a computer only weights of limited complexity (w.r.t.\ their bit-length) can be stored. In this paper, we consider weights that can be encoded by bit-strings with length logarithmically growing in $1/\eps$, where $\eps$ is the approximation accuracy.

To make the notion of encodability more precise, we first introduce coding schemes (see~\cite{petersen2017optimal}):
A \emph{coding scheme (for real numbers)} is a sequence $\mathcal{B}=(B_\ell)_{\ell\in\N}$ of maps $B_\ell:\{0,1\}^\ell\to \R$. 
Now we define sets of neural networks with weights encodable by a coding scheme. Given an arbitrary coding scheme $\mathcal{B}=(B_\ell)_{\ell\in\N},$ and $d\in\N, \eps,M>0$, we denote by 
\begin{equation}\label{eq:encod_notation}
    \mathcal{NN}^{\mathcal{B}}_{M,\ceil{C_0 \log_2(1/\eps)},d}
\end{equation} the set of all neural networks $\Phi$ with $d$-dimensional input, one-dimensional output and at most $M$ nonzero weights such that \emph{each nonzero weight of $\Phi$ is contained in $\mathrm{Range}(B_{\ceil{C_0 \log_2(1/\eps)}})$}.

\section{Lower Bounds For Neural Networks with Encodable Weights and General Activation Functions}\label{sec:Lower}

In this section, we derive lower bounds on the necessary number of nonzero, encodable weights of neural network approximations. The approximated function spaces include a wide variety of classical smoothness spaces and the accuracy is measured in rather general norms. Our result applies to every activation function that is sufficiently smooth to be considered in these norms. We note that the proof of our result is essentially an abstract version of the proof of \cite[Theorem 4.2]{petersen2017optimal}. After encouragement of one of the authors\footnote{We want to take the opportunity to thank Philipp Petersen for the fruitful suggestion.} of~\cite{petersen2017optimal} and after studying the paper more closely, we noticed that it is possible to consider the proof strategy of \cite[Theorem 4.2]{petersen2017optimal} in a very general setting which we will outline below. Throughout this section (unless stated otherwise) we fix some $d\in \N,$ some domain $\Omega\subset \R^d$ and two normed spaces $\mathcal{C},\mathcal{D}$ of (equivalence classes of) functions defined on $\Omega$ with values in $\R.$ Additionally, we assume that $\mathcal{C}\subset \mathcal{D}$.

First of all, we need the notion of the \emph{minimax code length $L_\eps(\mathcal{C},\mathcal{D})$ of $\CalC$ with respect to $\CalD$.} The minimax code length  describes the uniform description complexity of the set $\{f\in \CalC:\|f\|_{\CalC}\leq 1\}$ in terms of the number of nonzero bits necessary to encode every $f$ with distortion at most $\eps$ in $\CalD.$ It can be directly related to approximation capabilities of arbitrary computing schemes and is defined as follows (see also \cite[Definition B.2]{petersen2017optimal}):
\begin{definition}[Minimax Code Length]
    Let $\ell\in \N$. We denote by $\mathfrak{E}^{\ell} \coloneqq \left\{E:\CalC\to \{0,1\}^\ell \right\}$ the set of binary encoders mapping elements of $\CalC$ to bit strings of length $\ell,$ and by $\mathfrak{D}^{\ell}\coloneqq \{D:\{0,1\}^\ell\to \CalD\}$ the set of binary decoders mapping bit-strings of length $\ell$ into $\CalD$. For $\eps>0,$ we define the \emph{minimax code length} by $$ L_\eps(\mathcal{C},\mathcal{D}) \coloneqq \min \left\{\ell\in \N: \exists (E^\ell,D^\ell)\in \mathfrak{E}^{\ell}\times \mathfrak{D}^{\ell}:\sup_{f\in \CalC:\|f\|_{\CalC}\leq 1} \|D^\ell(E^\ell(f))-f\|_{\CalD}\leq \eps\right\}.$$
\end{definition}

The next observation demonstrates in the context of neural networks how the minimax code length can be employed to derive lower bounds for approximations.
\begin{observation}\label{ob:lowerbound_architecture}
Let $\eps>0$ and $\varrho:\R\to\R$ such that $\mathcal{NN}^{d}_\varrho \subset \mathcal{D}$. If $\mathcal{A}$ is a neural network architecture with $M$ unspecified nonzero weights\footnote{Or any computation scheme that takes as input $M$ parameters.} (but fixed number of layers, neurons and position of nonzero weights) such that for each $f\in \CalC$ there is a set of weights $w_1,\ldots, w_M$, where each weight can be encoded by at most $b\in\N$ bits and $\norm{\act{\mathcal{A}(w_1,\ldots,w_M)}-f}_\CalD\leq \eps$, then 
\[
M\geq L_\eps(\CalC, \CalD)/b.
\]
 Mapping $f\in\CalC$ to the bit representation of the $M$ weights can be viewed as an encoder, and mapping the encoded weights to $\act{\mathcal{A}(w_1,\ldots,w_M)}$ acts as a decoder with bit length $\ell=Mb$, which shows the claim. This in particular holds true, if $b\leq C\log_2(1/\eps)$ which is the focus of this paper. 
\end{observation}

In the following, we exploit this strategy to show that the same bound actually holds true, if we allow for the architecture to depend on the function to be approximated. That means, for each $f\in\CalC$ the number of layers, neurons and position of $M$ nonzero encodable weights (and the weights themselves) may change but need to be encoded. The next lemma (shown in \cite[Lemma~B.4]{petersen2017optimal} under the additional restriction that $\varrho(0)=0$\footnote{The Lemma is proven by first noting that a network with arbitrary number of neurons and layers, but $M$ non-zero weights, can be replaced by a network with the same number of non-zero weights, but number of neurons and layers bounded by $M+1$. This can be done by removing neurons that do not contribute to the next layer. This strategy (see also~\cite[Proposition~3.6]{bolcskei2017optimal}) allows us to drop the assumption that $\rho(0)=0$ from~\cite[Lemma~B.4]{petersen2017optimal}.}) shows the number of bits needed to encode this information.
\begin{lemma}\label{lem:nn_encoder}
Let $M,K\in\N$, and let $\mathcal{B}$ be an encoding scheme for real numbers and $\rho:\R\to\R$ an activation function. There is a constant $C=C(d)$, such that there is an injective map $\Gamma:\{\act{\Phi}: \Phi\in\mathcal{NN}^{\mathcal{B}}_{M,K,d}\}\to\{0,1\}^{CM(K+\lceil \log_2 M\rceil)}$.
\end{lemma}
 
To make the main statement of this section mathematically more precise, we introduce some further notation.
\begin{definition}\label{def:WeightBasic}
     Let $C_0>0$ be fixed. Additionally, let $f\in \CalC$, and for some function $\rho:\R\to \R$  assume that $\NN_\rho^{d}\subset \CalD$. Finally, let $\epsilon>0$ and fix some coding scheme  $\mathcal{B}$. Then, for $C_0>0,$ we define the quantities\footnote{we use the convention that $ \min \emptyset  = \infty$.}
    \[
    M^{\mathcal{B}}_{\eps}(f)\coloneqq M^{\mathcal{B},\varrho,C_0,\CalC,\CalD}_\eps(f)\coloneqq \min\left\{M\in \N: \exists \Phi \in \NN^{\Bcal}_{M, \ceil{C_0\cdot\log_2{\frac{1}{\eps}}},d}:\norm{f-\act{\Phi}}_{\mathcal{D}}\leq \eps\right\},
    \]
    and 
    \[
    M^{\mathcal{B}}_\eps(\CalC,\CalD)\coloneqq\Micod\coloneqq\sup_{f\in \CalC,~\|f\|_{\CalC}\leq 1} M_\eps^{\mathcal{B},\varrho,C_0,\CalC,\CalD}(f).
    \]
\end{definition}
In other words, the quantity $M^{\mathcal{B}}_\eps(f)$ denotes the required number of nonzero weights of a neural network~$\Phi$ to $\epsilon$-approximate $f$ with weights that can be encoded with $\lceil C_0 \log_2(1/\epsilon) \rceil$ bits using the coding scheme $\mathcal{B}$. $ M^{\mathcal{B}}_\eps(\CalC,\CalD)$ gives a uniform bound of this quantity over the unit ball in $\CalC$. 

Theorem \ref{thm:lower_bound_enc} now states that if we can lower bound the minimax code length, then we are also able to lower bound $M^{\mathcal{B}}_\eps(\mathcal{C},\mathcal{D})$.
Lower bounds on the minimax code length (and hence for the quantity $M^{\mathcal{B}}_\eps(\mathcal{C},\mathcal{D})$) for specific, frequently used function spaces fulfilling the assumptions of the theorem will be given in Corollary~\ref{cor:sobolev_low_bound_enc}. 
\begin{theorem}\label{thm:lower_bound_enc}
Let $\varrho:\R\to\R$ such that $\mathcal{NN}^{d}_\varrho \subset \mathcal{D}$. Additionally, assume that $L_\eps(\mathcal{C},\mathcal{D})\geq C_1 \eps^{-\gamma}$  for some $\gamma=\gamma(\CalC,\CalD), C_1=C_1(\CalC, \CalD)>0$ and all $\eps>0$. Then, for each $C_0>0$ there exists a constant $C = C(\gamma, \CalC, \CalD, C_0)>0$, such that for each coding scheme of real numbers $\Bcal$, and for all $\eps\in \epsin$ we have
\begin{align*}
 M^{\mathcal{B},\varrho,C_0}_\eps(\mathcal{C},\mathcal{D})\geq C\cdot\eps^{-\gamma} \Big/ \log_2\left(\frac{1}{\eps}\right).
\end{align*}
\end{theorem}
The idea for the proof of this theorem is the same as for Observation~\ref{ob:lowerbound_architecture}. Here, the encoder is $E:\CalC\to\{0,1\}^{\ell}, f\mapsto\Gamma(\act{\Phi_{\eps,f}})$, where $\Phi_{\eps,f}$ is the neural network $\eps$-approximating $f$, $\Gamma$ the network encoder from~Lemma~\ref{lem:nn_encoder} and $\ell=CM(\log_2(1/\eps) + \log_2(M))$. The decoder is given by $D:\{0,1\}^\ell\to\CalC, b\mapsto\Gamma^{-1}(b)$. The bound now follows from $CM(\log_2(1/\eps) + \log_2(M))\geq C_1\eps^{-\gamma}$.

\begin{remark}[Activation Functions]
We only require sufficient smoothness of the activation function for the spaces under consideration. 
Hence, we are in a position to conclude suitable lower bounds for \emph{all} practically used activation functions. 
\end{remark}
\begin{remark}[Bounds With Non-Encodable Weights]\label{rem:lower_bounds_nonenc}
     If one drops the restriction of encodable weights and considers the more general setting of arbitrary weights, a lesser number of weights is required in general. For this setting, we mention two examples.
     
     \begin{itemize}
         \item  The results from~\cite{yarotsky18a,guhring2019error} combined state:
    
\begin{adjustwidth}{}{3em}
    \emph{For $\CalC=\Wkp[n][\infty][(0,1)^d]$ and $\CalD=\Wkp[k][\infty][(0,1)^d]$ with $k=0,1$, it holds for the necessary number of nonzero weights $M_\eps$ to achieve an $\eps$-approximation in $W^{k,\infty}$ norm that
\begin{equation*}
M_\eps\geq C\eps^{-d/(2n-k)}.
\end{equation*}}
\end{adjustwidth}
For $k=0$, in~\cite{yarotsky18a} neural networks are constructed that achieve this approximation rate.
In comparison, our entropy bounds show that under the assumption of encodable weights $M_\eps\geq C \eps^{-d/(n-k)}$ (suppressing the $\log_2(\nicefrac{1}{\eps})$ factor for simplicity of exposition).
\item In~\cite{guliyev2016single} it is shown that there exists an activation function such that a neural network with three parameters is able to uniformly approximate each function in $\mathcal{C}=C([0,1])$ arbitrary well. Observation~\ref{ob:lowerbound_architecture} now shows that there is no finite encoding bit length for the weights necessary to approximate all functions in the unit ball of $C([0,1])$, since in this case $L_\eps(\CalC, \CalC)=\infty$ for $0<\eps<1$.\footnote{$L_\eps(\CalC, \CalC)=\infty$ for $0<\eps<1$ follows from the fact that the unit ball in $\CalC=C([0,1])$ is not compact. The same argument can also be used to directly deduce from the construction of the weights in~\cite{guliyev2016single} that their encoding bit length is not finite.}
     \end{itemize}
\end{remark}
We proceed by listing a variety of lower bounds for a selection of specific examples for frequently used function spaces. One can deduce similar lower bounds for other choices of $\mathcal{C},\mathcal{D}$. Notable examples that are not covered below include H\"older spaces, Triebel-Lizorkin, or Zygmund spaces (see for instance \cite{triebel1978interpolation,edmunds_triebel_1996} and the references therein for further examples).
\begin{corollary} \label{cor:sobolev_low_bound_enc}
Assume that $\Omega$ fulfills some regularity conditions.\footnote{Many results estimating the $\eps$-entropy are only formulated and proven for $C^\infty$-domains for simplicity of exposition. However, as has been described in \cite[Section 4.10.3]{triebel1978interpolation} and \cite[Section 3.5]{edmunds_triebel_1996}, these results remain valid for function spaces on more general domains including cubes.} Let $\varrho:\R\to \R $ be chosen such that $\mathcal{NN}_\varrho^{d}\subset \CalD$ (where $\CalD$ is a function space on $\Omega$ specified below).  Moreover, let $\mathcal{B}$ be an arbitrary coding scheme. Then, the following statements hold:
\begin{enumerate}[(i)]
    \item \textbf{Besov spaces:}  
    Let $s,t\in\R$ with $s<t$ as well as $p_1,p_2,q_1,q_2\in (0,\infty]$ such that $$t-s-d\max\left\{\left(\frac{1}{p_1}-\frac{1}{p_2} \right),0\right\}>0.$$ Moreover, let $\CalC= B^{t}_{p_1,q_1}(\Omega),$ and $\CalD=B^{s}_{p_2,q_2}(\Omega).$ Then, for some $C>0,$ we have
    \begin{align*}
        M^{\mathcal{B}}_\eps(\mathcal{C},\mathcal{D}) \geq C\eps^{-\frac{d}{t-s}} \Big/ \log_2\left(\frac{1}{\eps}\right), \qquad \text{ for all } \eps\in \epsin.
    \end{align*}

    \item \textbf{Sobolev Spaces:} \label{item:lowerbounds_sobolev} Let $s,t\in \N$ with $t>s$ and let $p\in (0,\infty]$.  
    Then, for $\CalC= W^{t,p}(\Omega)$ and for $\CalD=W^{s,p}(\Omega)$ there exists some $C>0$ with
    \begin{align*}
        M^{\mathcal{B}}_\eps(\mathcal{C},\mathcal{D}) \geq C\eps^{-\frac{d}{t-s}}\Big/ \log_2\left(\frac{1}{\epsilon} \right), \qquad \text{ for all } \eps\in \epsin.
    \end{align*}

\end{enumerate}
\end{corollary}
\begin{proof}
(i) follows immediately from Theorem~\ref{thm:lower_bound_enc} in combination with Theorem \cite[Section 3.5]{edmunds_triebel_1996}.

(ii) 
 follows from Theorem~\ref{thm:lower_bound_enc} together with \cite[Section 1.3]{Hardy}, where we use the estimate on the approximation number $a_k(id)$ (cf. page 9) combined with the relation of $a_k(id)$ and the entropy.
\end{proof}

\section{Upper Bounds For General Activation Functions in Sobolev Spaces} \label{sec:main}
In this section, we show that for an arbitrary accuracy $\eps>0$, every function from the unit ball of the Sobolev space $W^{n,p}$ 
\[
\Fndp\coloneqq\{f\in\Wkp[n][p][\cube^d]: \norm{f}_{\Wkp[n][p][\cube^d]}\leq 1\}
\]
can be $\eps$-approximated in weaker Sobolev norms $W^{k,p}$ (with $n> k$) by neural networks with fairly general activation function. For this, we explicitly construct approximating neural networks with constant depth (i.e., independent\ of $\eps$) and give upper bounds for the number of nonzero, encodable weights (depending on~$\eps$), which in the light of the results of Section~\ref{sec:Lower} are almost optimal. The main idea is based on the common strategy (see e.g.~\cite{yarotsky2017error,schmidt2017nonparametric,guhring2019error,ohnELU}) of approximating $f$ by localized polynomials which in turn are approximated by neural networks.
Our work differs from these other works in three major aspects: 
\begin{enumerate}[\indent(a)]
    \item Our approximations include $W^{k,p}$ for $k\leq j$ (instead of maximally $W^{1,p}$) depending on the smoothness $j$ of the activation function.
    \item Constructing a PU by neural networks with general activation function is tricky (contrary to $\mathrm{ReLU}$ networks) and can, in general, only be done approximately (see Section~\ref{subsec:IngredientI} and Figure~\ref{fig:pou_approx_vs_relu}).
    \item Our polynomial approximations and approximate PUs have depth independent of $\eps$, which results in constant-depth approximations of $f$.
\end{enumerate}
We construct localizing bump functions that form an \emph{(approximate) partition of unity} in Section~\ref{subsec:IngredientI} and efficiently \emph{approximate polynomials} by neural networks in Section~\ref{subsec:IngredientII}. 
Afterwards, the statements of the main results as well as a detailed overview of their overall proof strategies are given in Section~\ref{subsec:MainRes}.

\subsection{Ingredient I: (Approximate) Partition of Unity }\label{subsec:IngredientI}

In~\cite{yarotsky2017error,guhring2019error} the ReLU activation function is used to construct continuous, piecewise linear bump functions with compact support that form a PU. However, this approach heavily relies on properties of the ReLU and is only suitable for approximations in Sobolev norms up to order one. For general activation functions, there is, to the best of our knowledge, no canonical way to build a PU by neural networks. As a remedy we introduce approximate partitions of unity which are compatible with all practically used activation functions. In detail, for a gridsize $1/N$ (with $N\in\N$), we divide the domain $(0,1)^d$ into $(N+1)^d$ equally large patches and construct, for each patch $\Omega_m$ for $m\in\MNd$, a bump function $\phi_{m}\in W^{j,\infty}$. Deviating from usually used bump functions, $\phi_{m}$ is in general not compactly supported on the corresponding patch and their sum only approximates $\mathbbm{1}_{\cube^d}$, i.e.,\ $\sum_{m} \phi_{m} \approx \mathbbm{1}_{\cube^d}$. Additionally, we introduce a scaling factor $s\geq1$, which regulates the closeness of the approximate PU to an exact PU. For $s\to \infty$, we have that $\norm{\phi^s_m}_{\Omega_m^{\mathrm{c}}}\to 0$ and $\sum_{m} \phi_{m}^s \to \mathbbm{1}_{\cube^d}$. The overall approximation rates in our main result now also depend on properties of the approximate PU. It will later turn out that the speed of the convergence is the decisive factor for which rates can be shown. We distinguish between exponential and polynomial speed.
Besides the smoothness $j$ and the convergence speed there is one more defining quantity $\tau$ which we call the \emph{order of the PU}. The order $\tau$ specifies at which derivative the scaling factor starts to show. In other words, all derivatives up to order $\tau$ absorb the effect of the scaling. In~Definition~\ref{def:partition_of_unity} we formally introduce the notion of an approximate PU. Additionally to approximate PUs with exponential and polynomial convergence properties we also include exact PUs in this definition since these include (leaky) ReLUs and powers thereof. 
 \begin{definition}\label{def:partition_of_unity}
     Let $d\in\N,j,\tau \in \N_0$. We say that the collection of families of functions $(\Psi^{(j,\tau,N,s)})_{N\in\N,s\in\R_{\geq 1}}$, where each $\pou\coloneqq \{\phi_m^s: m\in\{0,\ldots, N\}^d\}$ consists of $(N+1)^d$ functions $\phi_m^s:\R^d\to \R$, is an \emph{exponential} (respectively \emph{polynomial}, \emph{exact}) \emph{partition of unity of order $\tau$ and smoothness $j$}, or short \emph{exponential} (\emph{polynomial}, \emph{exact}) $(j,\tau)$\emph{-PU}, if the following conditions are met:
     
There exists some $D>0,~C=C(k,d)>0$ and $S>0$ such that for all $N\in\N, s\geq S, k\in\{0,\ldots,j\}$ the following properties hold:
 \begin{enumerate}[(i)]
		\item \label{item:pou_derivativeGeneralSigmoid}  $\norm{\phi_m^s}_{\Wkp[k][\infty][\R^d]}\leq C N^k\cdot s^{\max\{0,k-\tau\}}$ for every $\phi_m^s\in\pou$;
			\item \label{item:pou_suppGeneralSigmoid} for $\Omega^{c}_m = \big\{x\in \R^d:~\|x-\frac{m}{N}\|_{\infty}\geq \frac{1}{N}\big\}$, we have 
		\[
		\norm{\phi_m^s}_{\Wkp[k][\infty][\Omega_m^{c}]}\leq 
		\begin{cases} CN^k s^{\max\{0,k-\tau\}}e^{-Ds,}&\text{if \emph{exponential PU,}}\\
        CN^k s^{\max\{0,k-\tau\}}s^{-D,}&\text{if \emph{polynomial PU,}}\\
		0, &\text{if \emph{exact PU}},
		\end{cases}
		\]
		for every $\phi_m^s \in\pou$. 
		\item  \label{item:pou_sumGeneralSigmoid} We have \[
		\norm*{\mathbbm{1}_{\cube^d}-\sum_{m\in\MNd}\phi_m^s}_{\Wkp[k][\infty][\cube^d]}\leq \begin{cases} CN^k s^{\max\{0,k-\tau\}}e^{-Ds},&\text{if \emph{exponential PU,}}\\
        CN^k s^{\max\{0,k-\tau\}}s^{-D},&\text{if \emph{polynomial PU,}}\\
		0, &\text{if \emph{exact PU}}.
		\end{cases}
		\]
		
		\item \label{item:pou_networkGeneralSigmoid}There exists a function $\varrho:\R\to\R$ such that for each $\phi_m^s\in\Psi$ there is a neural network $\Phi_m^s$ with $d$-dimensional input and $d$-dimensional output, with two layers and $C$ nonzero weights, that satisfies
			\[
	\prod_{l=1}^d [\act{\Phi_m^s}]_l=\phi_m^s, \quad
				\]
			and $\norm{\act{\Phi_m^s}}_{\Wkp[k][\infty][\cube^d]}\leq CN^k\cdot s^{\max\{0,k-\tau\}}$. Furthermore, for the weights of $\Phi_m^s$ it holds that $\bweights{\Phi_m^s}\leq CsN$.
	\end{enumerate}
 \end{definition}

In the next Definition, we state conditions for an activation function $\rho$ to admit (in the sense of Definition~\ref{def:partition_of_unity}~\eqref{item:pou_networkGeneralSigmoid}) an exponential (polynomial, exact) PU of order $\tau$ with smoothness $j$ for $\tau\in\{0,1\}$ and afterwards we explicitly construct the corresponding PUs. For $\tau=0$ the activation functions are approximately piecewise constant outside of a neighborhood of zero (e.g.,\ sigmoidal) and for $\tau=1$ approximately piecewise affine-linear outside of a neighborhood of zero (e.g.,\ ELU). The speed they approach their asymptotes with (see (d) in the next definition) defines the convergence speed of the resulting PU. Furthermore, we require $\rho$ to be $j$-smooth. 
 \begin{definition}\label{def:Admissible}
     Let $j\in \N_0,\tau \in \{0,1\}.$ We say that a function $\varrho:\R\to\R$ is \emph{exponential (polynomial, exact) $(j,\tau)$-PU-admissible}, if 
     \begin{itemize}
        \item[(a)]\label{item:BoundedLipschitz} 
        $\varrho \text{ is }\begin{cases} 
                                    \text{bounded},\hfill &\text{ if } \tau=0, \\
                                    \text{Lipschitz~ continuous} , \hfill \quad &\text{ if } \tau=1;
                            \end{cases}$
        \item[(b)]\label{item:SmoothnessRho} There exists $R>0$ such that $\varrho\in C^{j}(\R\setminus [-R,R])$;
        \item[(c)]\label{item:SmoothnessK} $\varrho'\in W^{j-1,\infty}(\R)$, if $j\geq 1$ ;
         \item[(d)]\label{item:Decay} There exist $A=A(\varrho),B=B(\varrho)\in\R$ with $A<B,$ some $C=C(\varrho,j)>0$ and some $D=D(\varrho,j)>0$ such that 
                \begin{itemize}
                    \item[(d.1)]\label{item:DecayPos} $\left|B-\varrho^{(\tau)}(x) \right|\leq  Ce^{-Dx}$ ($C x^{-D}$ if \emph{polynomial}, $0$ if \emph{exact}) for all $x>R$;
                    \item[(d.2)]\label{item:DecayNeg} $\left|A-\varrho^{(\tau)}(x) \right|\leq Ce^{Dx}$ ($C' \pabs{x}^{-D}$ if \emph{polynomial}, $0$ if \emph{exact}) for all $x<-R$;
                    \item[(d.3)]\label{item:DecayHigh} $\left|\varrho^{(k)}(x) \right|\leq Ce^{-D|x|}$ ($C \pabs{x}^{-D}$ if \emph{polynomial}, $0$ if \emph{exact}) for all $x\in \R\setminus[-R,R]$ and all $k=\tau +1,\dots,j$.
                \end{itemize}
     \end{itemize}
 \end{definition}
 
\begin{remark}
To give the reader a better intuition for the above definition we mention the similarity to \emph{$\tau$-degree sigmoidal functions} (see~\cite{MHASKAR1992350}) defined as $\rho:\R\to\R$ with
\[
\lim_{x\to-\infty} \frac{\rho(x)}{x^\tau} = 0, \quad \lim_{x\to\infty} \frac{\rho(x)}{x^\tau} = 1.
\]
 Roughly speaking we require the same asymptotic behavior (with the exception that the asymptotes do not need to be $0,1$) and additionally that the asymptotes are approached with a certain speed.
\end{remark} 
 
In Table~\ref{tab:ActFunctions}, we list a large variety of commonly used activation functions and their corresponding PU properties. The proofs of these properties can be found in Appendix~\ref{sec:activation_admissibility}. 

In the next Definition, we give (depending on $\tau$) a recipe for the construction of a one-dimensional approximate bump from which multi-dimensional bumps are derived via a tensor approach. 
 To give the reader a better impression of the definition below and the role of the scaling factor, we present exponential, polynomial and exact bumps and resulting PUs for different activation functions and scaling $s$ in~Figure~\ref{fig:pou_approx_vs_relu}.
 
 \begin{definition}\label{def:BumpsGeneral}
 Let $j\in\N_0,\tau \in \{0,1\}$. Assume that $\varrho:\R\to\R$ is exponential, polynomial or exact $(j,\tau)$-PU-admissible. We define, for a scaling factor $s\geq 1,$ the one-dimensional bump functions  
 \begin{align*}\displaystyle
     \psi^s:\R\to\R, \quad\psi^s(x)\coloneqq \begin{cases} \frac{1}{B-A}\left(\rho(s(x+3/2))-\rho(s(x-3/2))\right), & \text{if } \tau=0, \\
     \frac{1}{s(B-A)}\left(\rho(s(x+2))-\rho(s(x+1))-\rho(s(x-1)+\rho(s(x-2))\right), & \text{if } \tau=1.\end{cases} 
 \end{align*}
 
 \noindent For  $N\in\N,~d\in\N$ and $m\in \{0,\dots,N\}^d$ we define multi-dimensional bumps $\phi_m^s:\R^d\to\R$ as a tensor product of scaled and shifted versions of $\psi^s$. Concretely, we set
\begin{equation*}
	\phi_m^s(x)\coloneqq \prod_{l=1}^d\psi^s\left(3N\left(x_l-\frac{m_l}{N}\right)\right).
\end{equation*}
Finally, for $N\in\N,s\geq 1$, the collection of bump functions is denoted by $\Psi^{(j,\tau,N,s)}(\rho)\coloneqq \{\phi_m^s: m\in\{0,\ldots,N\}^d\}$.
 \end{definition}
 In the next Lemma we show that the conditions from~Definition~\ref{def:Admissible} together with the construction in~Definition~\ref{def:BumpsGeneral} are indeed sufficient to generate an (approximate) PU. 
 \begin{lemma}\label{prop:partition_of_unity}
     Let $j\in \N_0,\tau \in \{0,1\}$ and a function $\varrho:\R\to\R$ be \emph{exponential (polynomial, exact) $(j,\tau)$-PU-admissible}. Then the collection of families of functions $(\Psi^{(j,\tau,N,s)}(\rho))_{N\in\N,s\in\R_{\geq 1}}$ defined in~Definition~\ref{def:BumpsGeneral} is an exponential (polynomial, exact) PU of order $\tau$ and smoothness $j$. 
 \end{lemma}
 \begin{proof}
  The proof of this statement is the subject of Appendix \ref{app:Bumps}. We only give the proof for exponential PUs. The statement for the other two cases follows analogously.
   \end{proof} 
   
   We demonstrate in Appendix \ref{sec:activation_admissibility} the admissibility for many practically-used activation functions. In Table \ref{tab:ActFunctions} below we have included the types of PUs these activation functions induce. 

\begin{remark}
Definition~\ref{def:Admissible} can be generalized to higher $\tau\geq 2$, resulting in an increasing amount of terms in the definition of a bump. Since most  activation functions used in practice are of order $\tau\in\{0,1\}$ we did not introduce this concept for simplicity of exposition. An example of  $(\tau\geq 2)$-functions are \emph{$\tau$-order RePUs} (short for \emph{Rectified Power Unit}, see, e.g., \cite{RePURates}), given by $ \mathrm{ReLU}^\tau$. Due to its obvious connections to \emph{B-splines} of order $\tau+1$ (see for instance \cite[Chapter IX]{DeBoor}), and their ability to form an exact PU  (\cite[p.~96]{DeBoor}) as well as their smoothness properties, it is clear that the resulting system $\left(\Psi^{(\tau,\tau,N,s)}(\mathrm{ReLU}^\tau)\right)_{N\in\N,s\geq 1}$ forms an exact $(\tau,\tau)$-PU.
\end{remark}

\begin{figure}
\centering
\subcaptionbox{exponential PU implemented by sigmoid-neural networks with scaling $s=1$. }{\scalebox{0.8}{\includegraphics[width=0.4\linewidth]{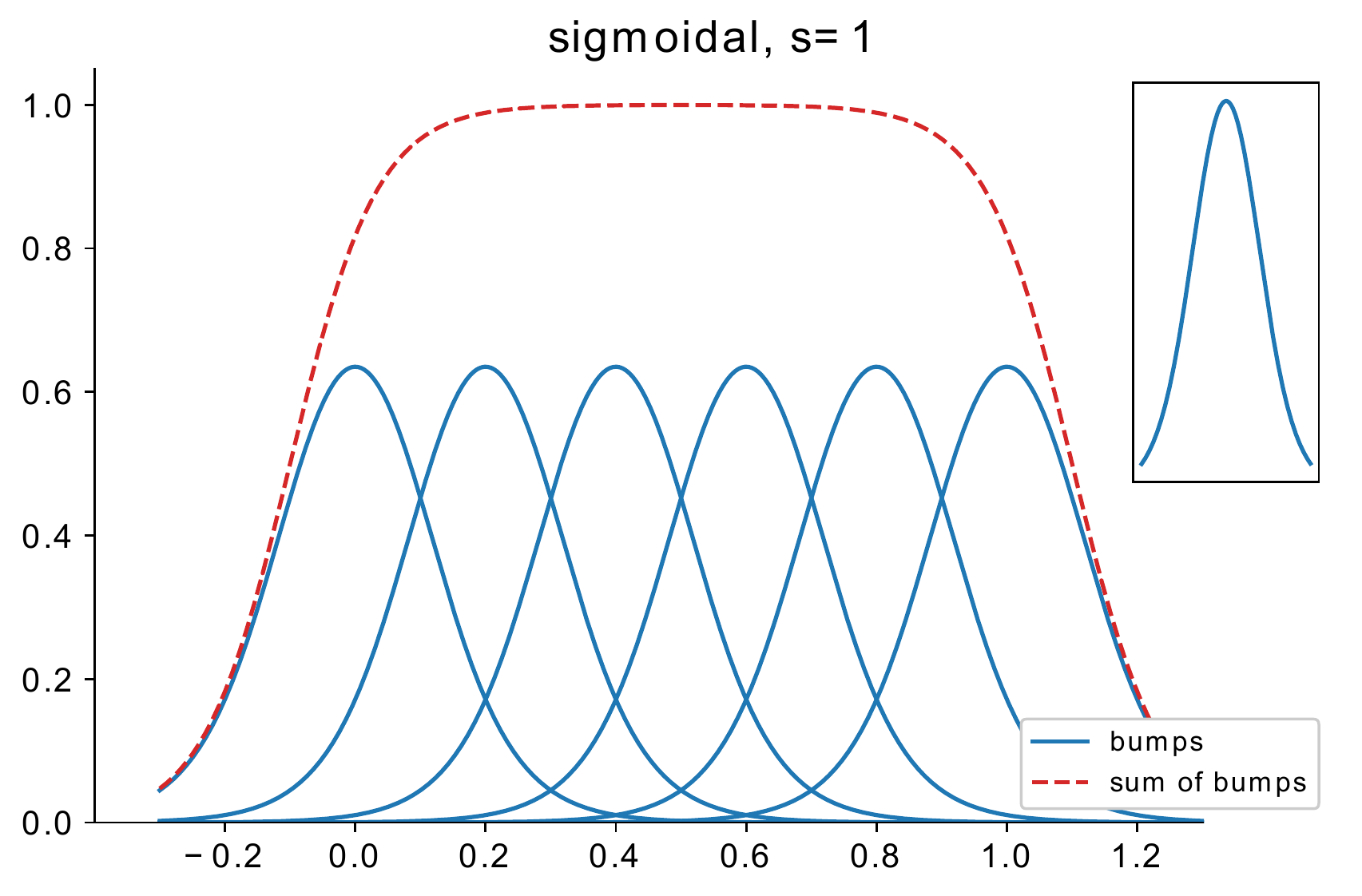}}} \hspace{3em}
\subcaptionbox{exponential PU implemented by sigmoid-neural networks with scaling $s=4$.}{\scalebox{0.8}{\includegraphics[width=0.4\linewidth]{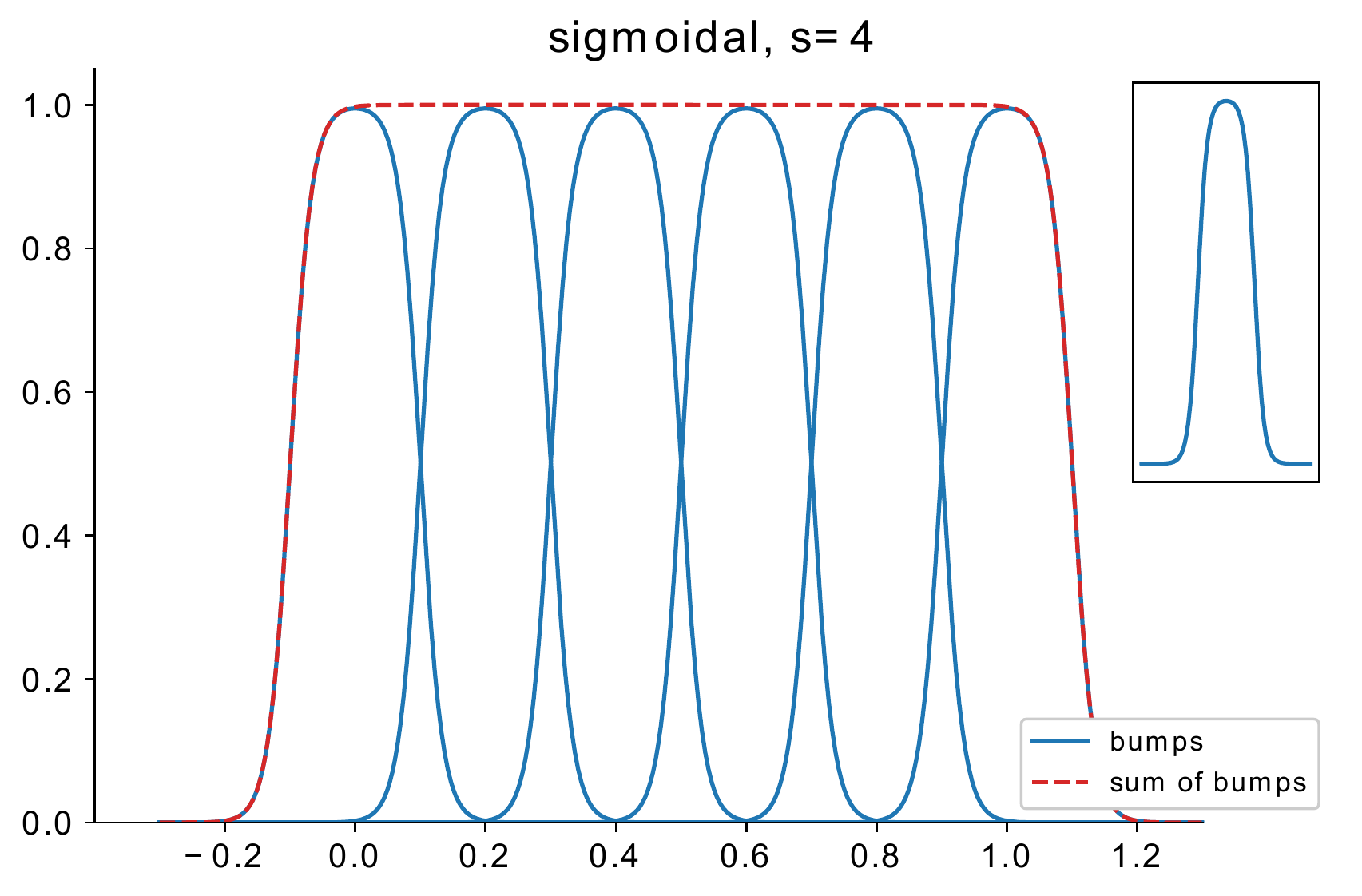}}}

\vspace{1.5em}

\subcaptionbox{exponential PU implemented by ELU-neural networks with scaling $s=1$. }{\scalebox{0.8}{\includegraphics[width=0.4\linewidth]{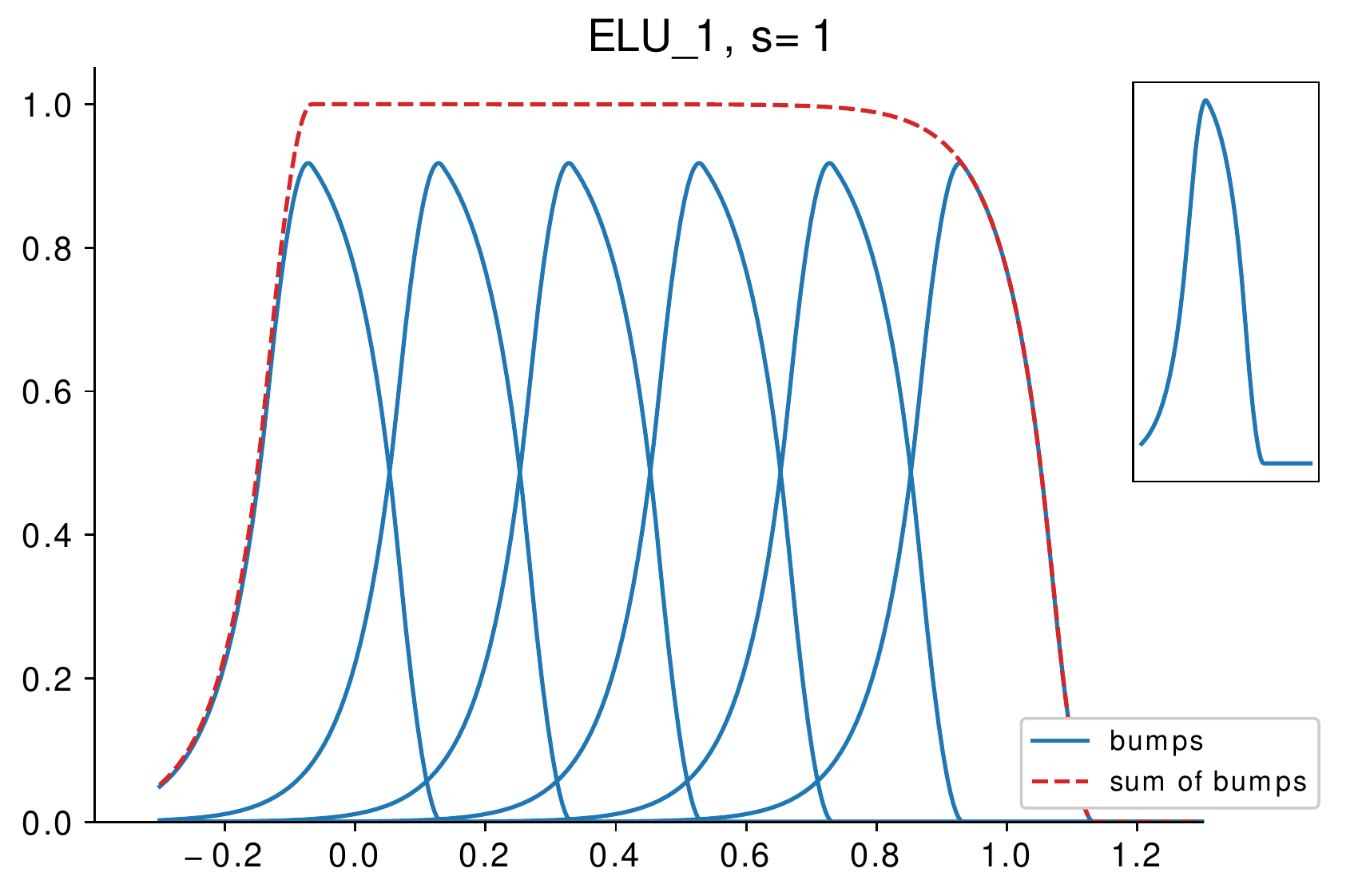}}}
\hspace{3em}
\subcaptionbox{exponential PU implemented by ELU-neural networks with scaling $s=4$. }{\scalebox{0.8}{\includegraphics[width=0.4\linewidth]{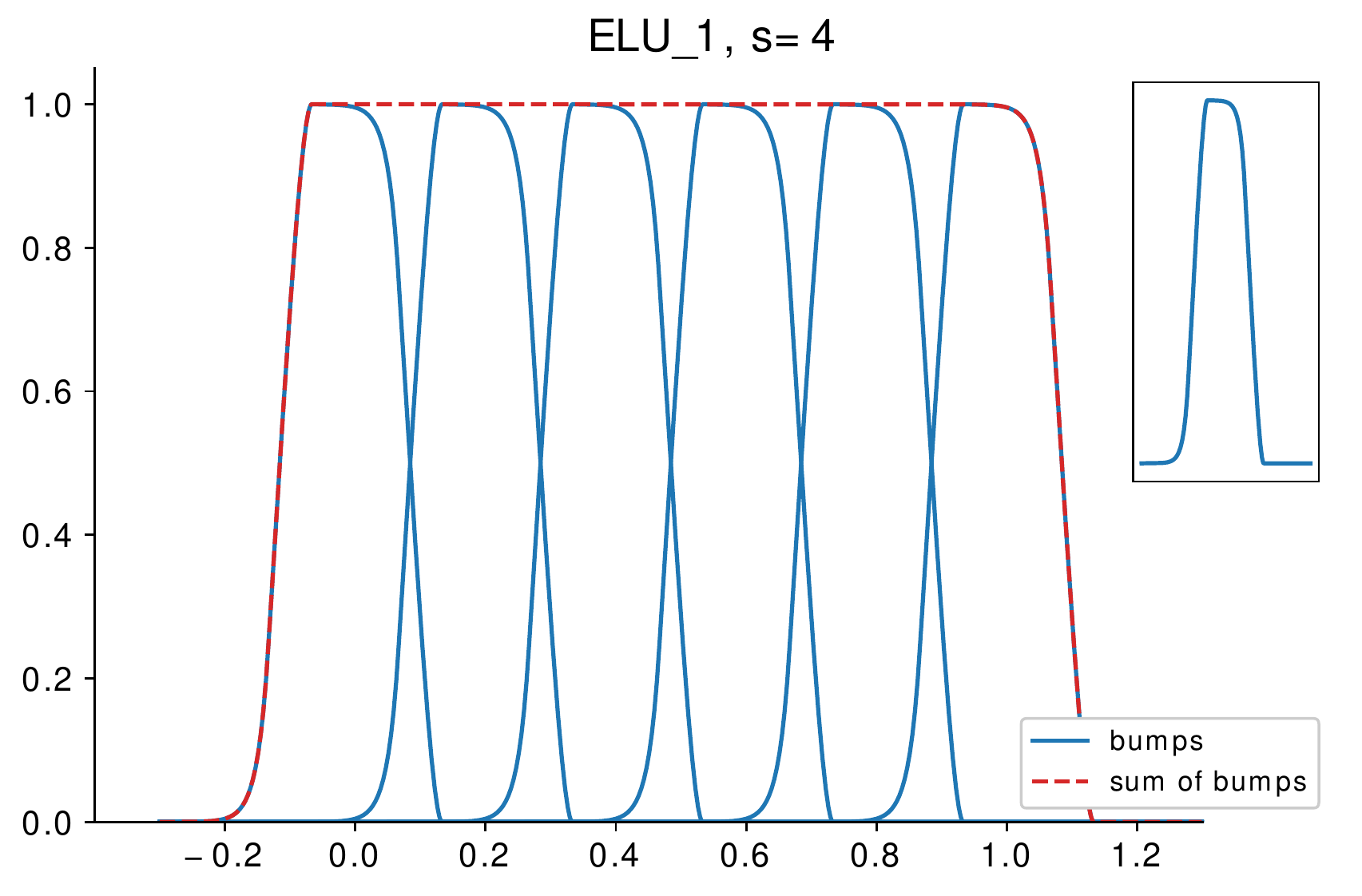}}}

\vspace{1.5em}

\subcaptionbox{polynomial PU implemented by softsign-neural networks with scaling $s=1$. }{\scalebox{0.8}{\includegraphics[width=0.4\linewidth]{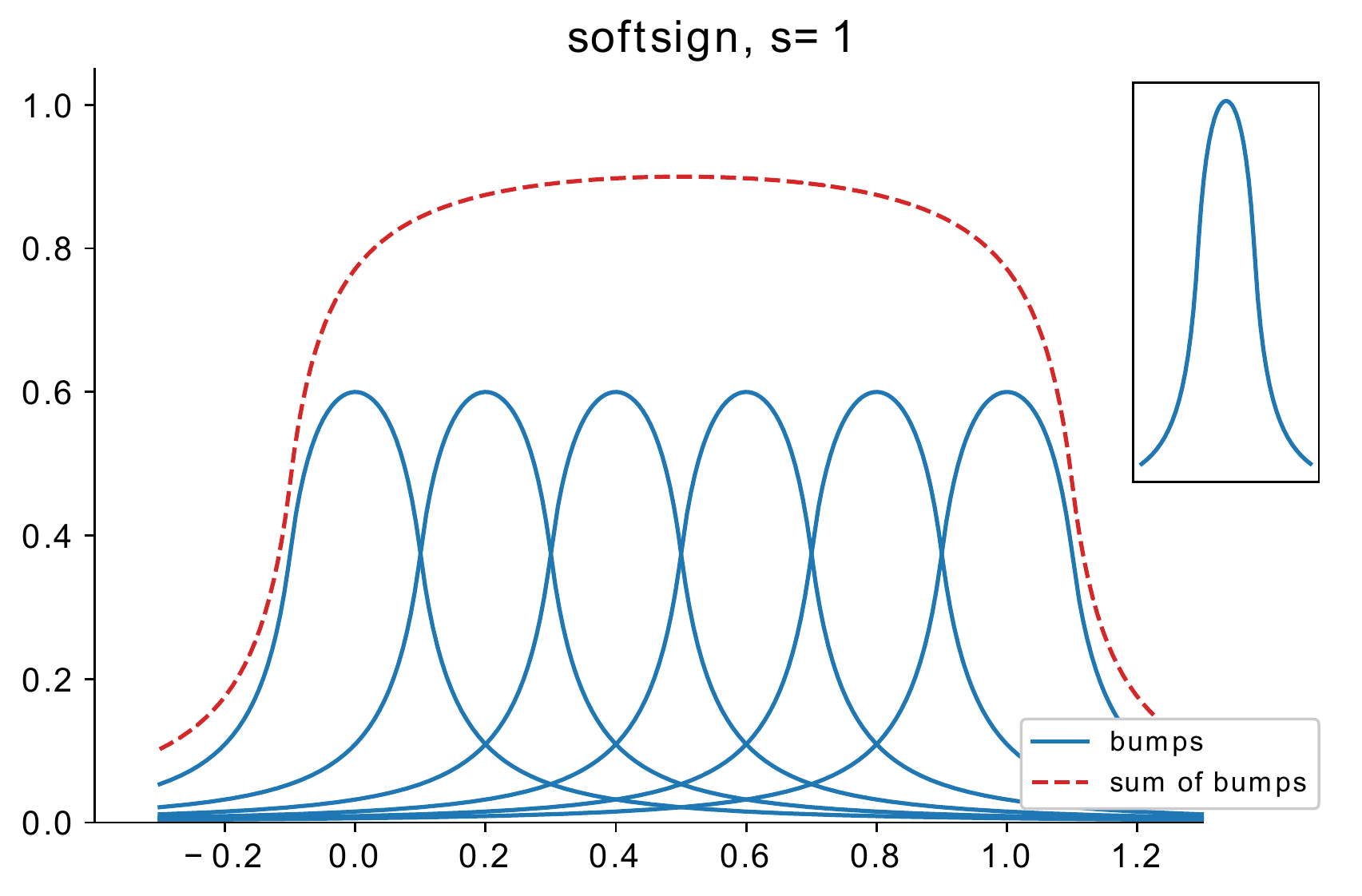}}}
\hspace{3em}
\subcaptionbox{polynomial PU implemented by softsign-neural networks with scaling $s=4$. }{\scalebox{0.8}{\includegraphics[width=0.4\linewidth]{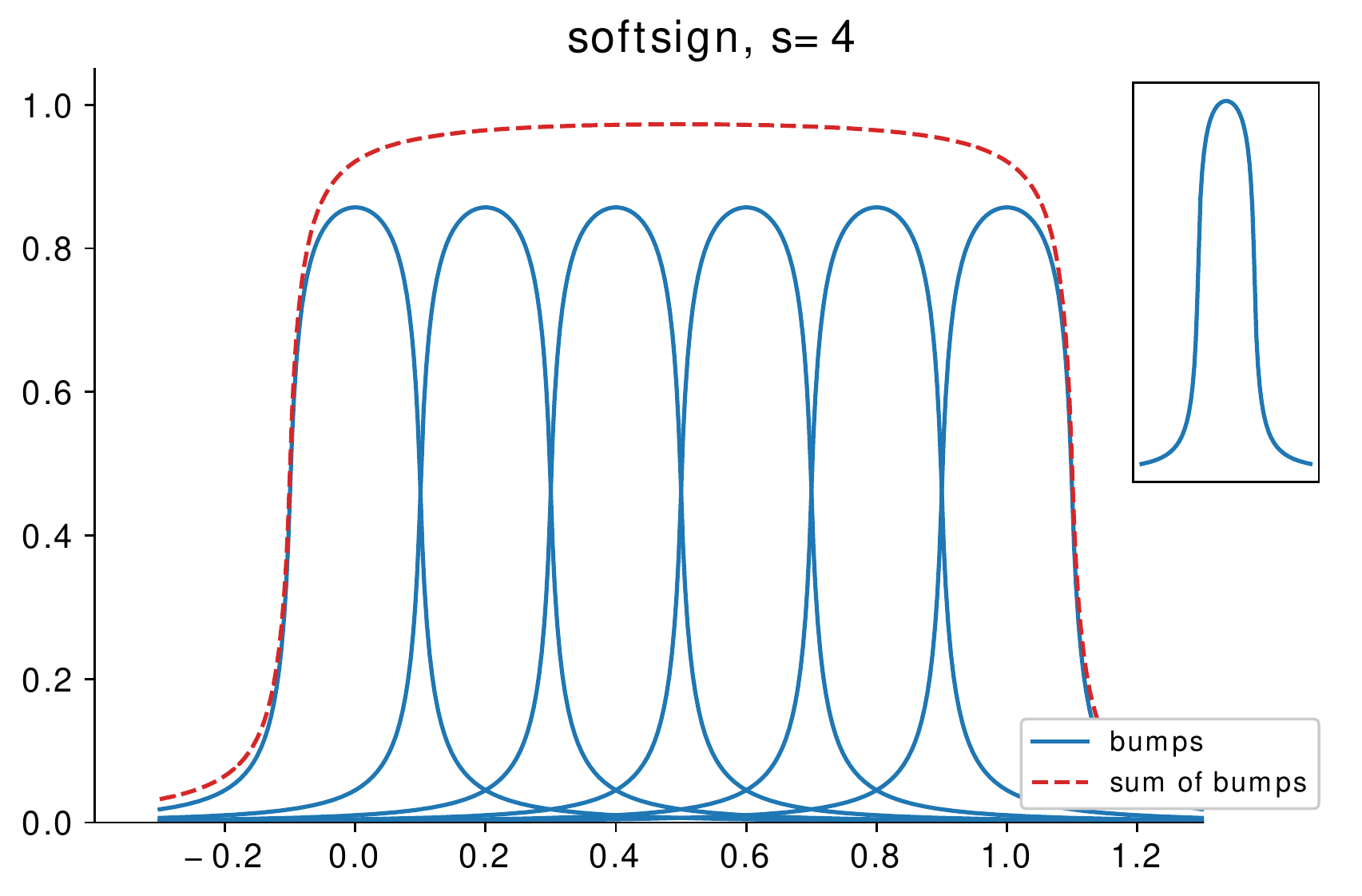}}}

\vspace{1.5em}

\subcaptionbox{Exact PU implemented by ReLU-neural networks used in~\cite{yarotsky2017error,guhring2019error}.}{\scalebox{0.8}{\includegraphics[width=0.4\linewidth]{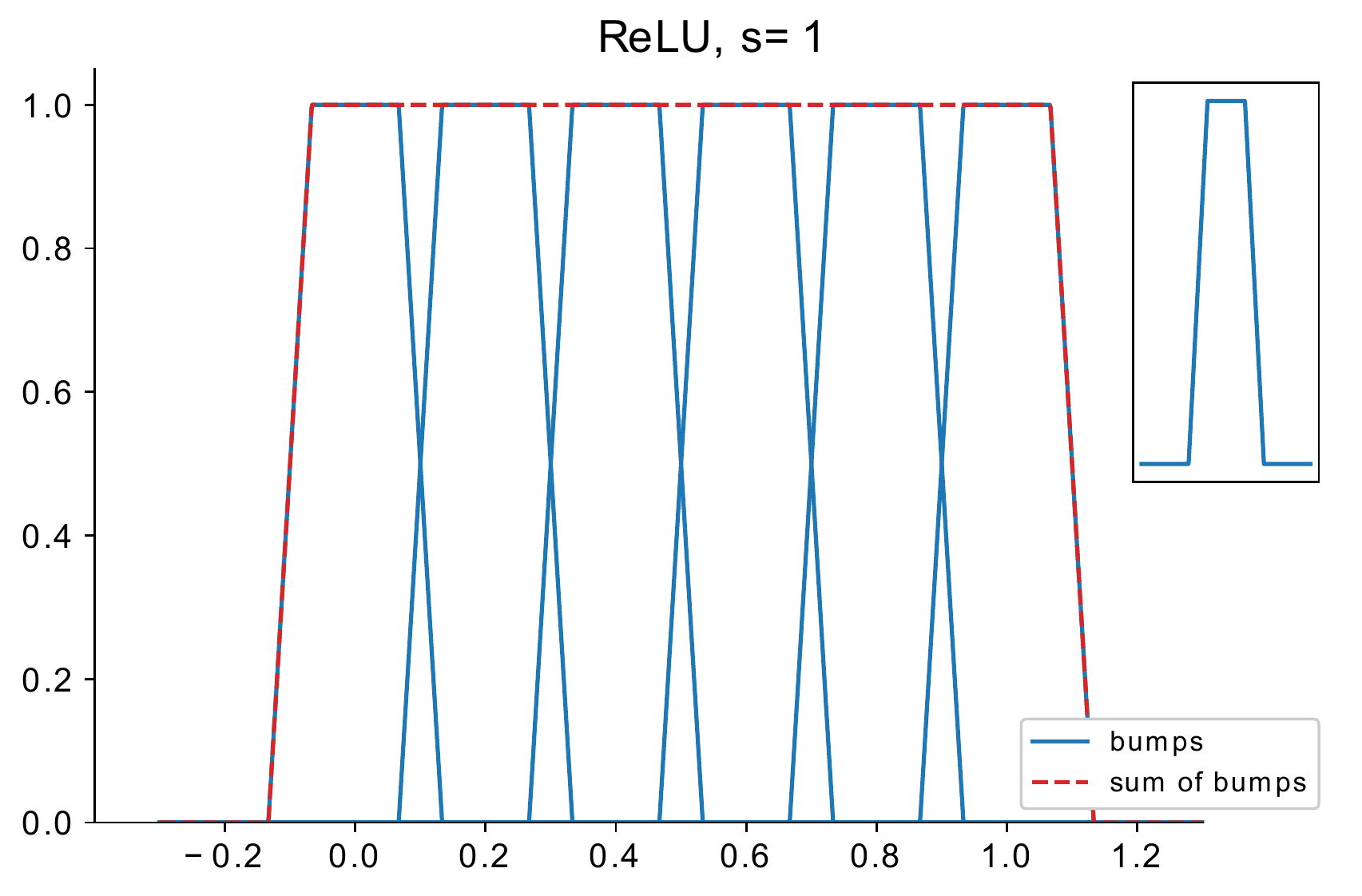}}}
\hspace{3em}
\subcaptionbox{Exact PU implemented by quadratic RePU-neural networks ($\tau=2$).}{\scalebox{0.8}{\includegraphics[width=0.4\linewidth]{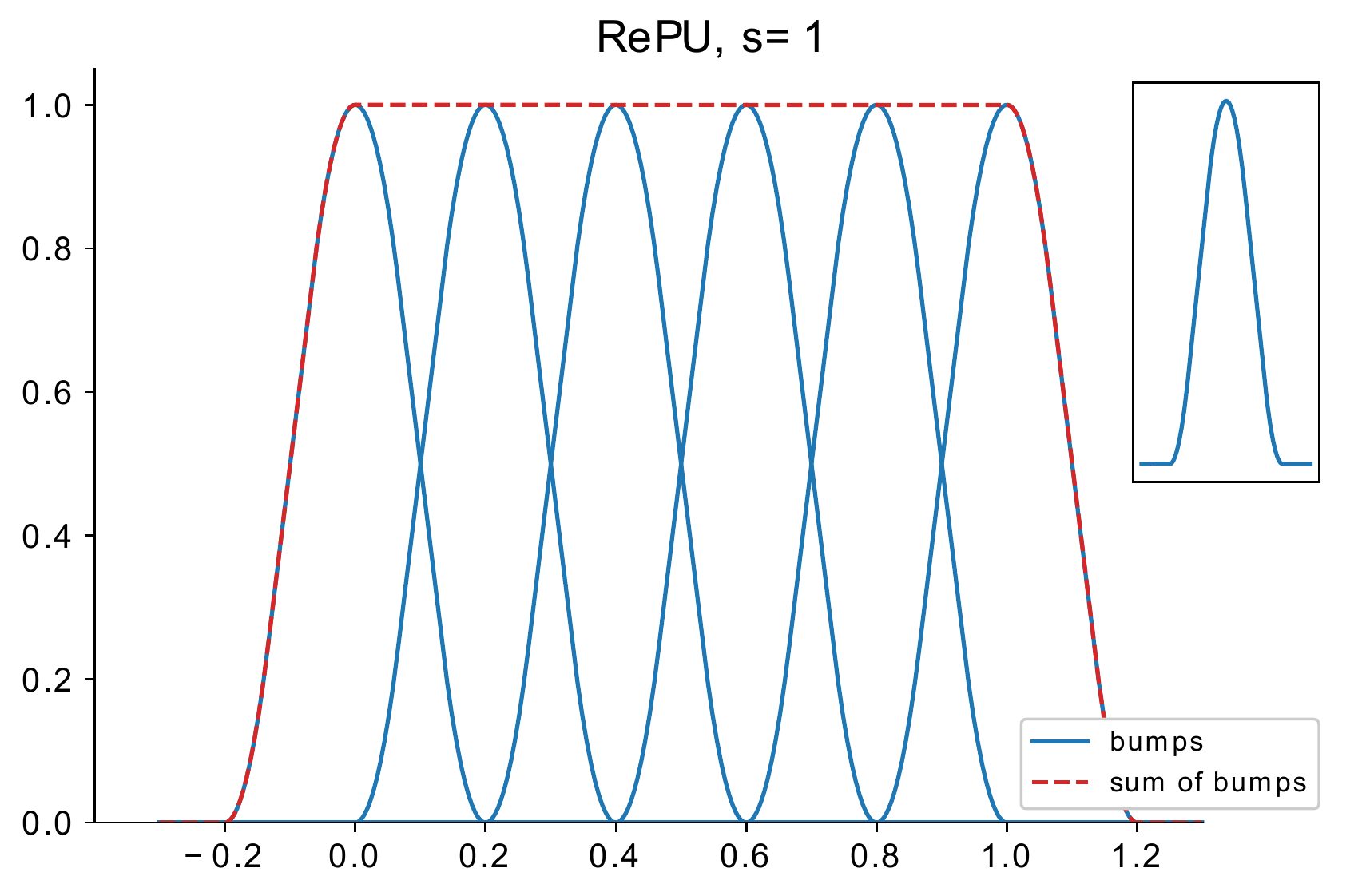}}}
\end{figure}
\addtocounter{figure}{-1}
\begin{figure} [t!]
\caption{(Previous page.) All displayed partitions of unity have $6$ bumps ($N=5$). The red curve shows the sum of the bump functions. A single bump function can be seen in the small window in the upper right part of each plot. The first two rows depict an exponential PU for $\tau=0$ (first row) and $\tau=1$ (second row). A polynomial PU of order $\tau=0$ can be seen in the third row. The impact of increasing the scaling factor $s$ can be seen in the second column. In the last row two exact PUs are shown. Here, the sum is constant $1$ on $(0,1)$ and scaling has no impact. }
\label{fig:pou_approx_vs_relu}
\end{figure}
 
 \subsection{Ingredient II: Approximation of Polynomials}\label{subsec:IngredientII}

Later on, we approximate our target function by localized polynomials $\sum \phi_{m}\cdot \mathrm{poly}_{m}$, where the $\phi_{m}$ are the localizing functions from Section \ref{subsec:IngredientI}\footnote{see Appendix \ref{app:ApproxLocal} for the precise statement and its proof}. Afterwards, we emulate these localized polynomials by neural networks\footnote{see Lemma~\ref{lemma:network_polynomial_approximation} in Appendix \ref{app:ApproxLocPol} for the final statement and its proof}.  For this, we need to approximate polynomials in an efficient way.
We start with approximating monomials $x\mapsto x^r$ on $\R$ by two-layered neural networks with activation functions that have a non-vanishing Taylor coefficient of order $r\in \N.$ The construction is mainly based on a generalization of a standard approach for approximating the function $x\mapsto x^2$ by using finite differences. This has been studied in~\cite{PowerOfDepth} and variations thereof have been considered, e.g., in \cite{ohnELU,schwab2018deep}.
\begin{proposition}\label{prop:MonApp}
    Let $\varrho:\R\to\R$ be a function. Assume, that for some
    $n\in\N$ there exists $x_0\in \R$ such that $\varrho$ is $n+1$ times continuously differentiable in some open neighborhood $U$ around $x_0$ and $\varrho^{(r)}(x_0)\neq 0$ for some $r\in  \{1,\dots,n\}.$   
    Then, for  every $\epsilon\in (0,1),$ and every $B>0$ there exists a constant $C=C(B,\varrho,r,n)>0$ as well as a neural network $\Phi^r_\eps$ with $\Realization(\Phi^r_\epsilon)|_{[-B,B]}\in C^{n+1}([-B,B])$ and the following properties:
    \begin{enumerate}
        \item[(i)] $\left\|\Realization(\Phi^r_\epsilon)-x^r\right\|_{C^{k}([-B,B])} \leq \epsilon$ for all $k=0,\dots,n$;
        \item[(ii)] $\pabs{\Realization(\Phi^r_\epsilon)}_{\Wkp[k][\infty][[-B,B]]}\leq C\frac{r!}{(r-k)!}B^{r-k}
        $ for $k=0,\dots,r$ and $\pabs{\Realization(\Phi^r_\epsilon)}_{\Wkp[k][\infty][[-B,B]]}\leq \eps
        $ for $k=r+1,\dots,n$;
        \item[(iii)] $L\left(\Phi_\eps^r\right)=2,$ as well as $M\left( \Phi_\eps^r\right) \leq 3(r+1);$
        \item[(iv)]  $\bweights*{\Phi^r_\eps}\leq C\epsilon^{-r}.$
    \end{enumerate}
\end{proposition}
\begin{proof}
 The proof of this result can be found in Appendix \ref{app:CalcMon}.
\end{proof}
 
Proposition~\ref{prop:MonApp} comes handy for two other usages besides monomial approximation:
 \begin{itemize}
     \item  We construct neural networks which implement an \emph{approximate multiplication} (see Corollary~\ref{cor:approximate_multiplication}) via the polarization identity
    \begin{align*}
        xy= \frac{1}{4} \left((x+y)^2-(x-y)^2 \right)\quad\text{for }x,y\in\R.
    \end{align*}    
This can by now be considered a standard approach in neural network approximation theory (originally used in~\cite{yarotsky2017error}). For this, the assumptions from Corollary~\ref{cor:approximate_multiplication} need to be fulfilled for $n=2$ and $r=2$, which holds true for all activation functions listed in~Table~\ref{tab:ActFunctions} \emph{except for the ReLU and the leaky ReLU}.
We use the approximate multiplication to obtain approximations of the multi-dimensional bumps $\phi_m$ from one-dimensional bumps which are in turn by construction neural networks. Furthermore, we can now deal with the multiplication of and $\phi_{m}$ with $\mathrm{poly}_{m}$ (see Corollary~\ref{cor:approximate_multiplication} in Appendix \ref{app:CalcMon} and~Lemma~\ref{lemma:network_multiplikation} in~Appendix~\ref{app:ApproxLocPol}).
     \item It is often useful to pass output from a layer to a non neighboring layer deeper in the network. Previous works have solved this issue for the ReLU activation function by constructing an identity network (e.g., \cite{petersen2017optimal,guhring2019error}). For general activation functions this is not possible. With help of Proposition~\ref{prop:MonApp} (for $n=1$ and $r=1$) an \emph{approximate identity neural network} can be constructed (see Proposition~\ref{cor:ApproxIdent}). It is clear that all activation functions listed in Table~\ref{tab:ActFunctions} fulfill the requirements.
 \end{itemize}

\subsection{Main Results Based on Ingredients I \& II}\label{subsec:MainRes}
The proof of the main statement of this section can be roughly divided into two steps:
In Proposition~\ref{prop:main}, the approximating neural networks are constructed with weights whose absolute values are bounded polynomially in $\eps^{-1}$. In Theorem \ref{thm:main}, the encodability of the weights is enforced. Before we state the actual results we give an overview of the proof of~Proposition~\ref{prop:main}, in which we explain the different approximation rates that can be obtained from different PUs. We hope that this excurse will make it easier for the reader to keep track of the different approximation rates presented in the results of this section.
\begin{overview*}
Let $\eps>0$. The proof of Proposition \ref{prop:main} is based on approximating a sum of $N=N(\eps)$ localized Taylor polynomials (which are close to $f$) by a neural network $\Phi_{P,\eps}$, such that we get
\begin{equation*}
  \norm[\bigg]{f-\act{\Phi_{P,\eps}}}_{\Wkp[k][\infty][\cube^d]}\leq\underbrace{\norm[\bigg]{f-\sum_m \phi_m^s \pp}_{\Wkp[k][\infty][\cube^d]}}_{\text{Step 1}}+\underbrace{\norm[\bigg]{\sum_m \phi_m^s \pp- \act{\Phi_{P,\eps}}}_{\Wkp[k][\infty][\cube^d]}}_{\text{Step 2}}.
\end{equation*}

\noindent \textbf{Step 1}: We start by depicting how our PUs are used together with localized Taylor polynomials. In the process the interplay between the convergence speed of the PUs and the approximation rates that can be obtained becomes clear. When approximating a function $f$ by localized Taylor polynomials $\pp$, where the localization is realized by a PU from Section~\ref{subsec:IngredientI}, we estimate the error on a fixed patch $\Omega_{\wtilde m}$ by 
\begin{align*}
\norm[\bigg]{f-\sum_m \phi_m^s \pp}_{\Wkp[k][\infty][\Omega_{\wtilde m}]}&\leq \norm[\bigg]{\Big(\mathbbm{1}_{\cube^d}-\sum_m \phi_m^s\Big)f}_{{\Wkp[k][\infty][\Omega_{\wtilde m}]}}+ \norm[\bigg]{\sum_{m}\phi_{m}^s(f-\pp)}_{{\Wkp[k][\infty][\Omega_{\wtilde m}]}}.
\end{align*}
The first term can be handled by Definition~\ref{def:partition_of_unity}~\eqref{item:pou_sumGeneralSigmoid} of the PU. Here, we only focus in detail on the second term. We have
\begin{align*}
\norm[\bigg]{\sum_{m}\phi_{m}^s(f-\pp)}_{\Wkp[k][\infty][\Omega_{\wtilde m}]}\lesssim C\underbrace{\sum_{\norm{m-\wtilde m}_\infty> 1}\norm{\phi_{m}^s}_{\Wkp[k][\infty][\Omega_{\wtilde m}]}}_{\mathrm{(a)}}+\underbrace{\sum_{\norm{m-\wtilde m}_\infty\leq 1}\norm{\phi_m^s(f-\mathrm{poly}_{\wtilde m})}_{\Wkp[k][\infty][\Omega_{\wtilde m}]}}_{\mathrm{(b)}}.
\end{align*}
 In the cases of exponential/polynomial PUs, we will make use of the decay property of Definition~\ref{def:partition_of_unity}~\eqref{item:pou_suppGeneralSigmoid}. In general we get
\begin{align*}
    \sum_{\norm{m-\wtilde m}_\infty> 1}\norm{\phi_{m}^s}_{\Wkp[k][\infty][\Omega_{\wtilde m}]}\lesssim N^d\cdot\begin{cases} CN^k s^{\max\{0,k-\tau\}}e^{-Ds},&\text{if \emph{exponential PU,}}\\
        CN^k s^{\max\{0,k-\tau\}}s^{-D},&\text{if \emph{polynomial PU,}}\\
		0, &\text{if \emph{exact PU}}.
		\end{cases}
\end{align*}
The closeness of the approximate bump to an exact bump is determined by the scaling factor $s$ which we now couple with $N$.
\begin{itemize}
    \item For the \emph{exponential case} we set $s\coloneqq N^\mu$ for arbitrarily small $\mu>0$ and can now use that the exponential term decays faster than any polynomial in $N$ grows. In particular, we have
\[
N^dN^k s^{\max\{0,k-\tau\}}e^{-Ds}=N^dN^k N^{\drate}e^{-DN^\mu}\leq N^{-(n-k)}
\] for $N$ large enough.
\item In the \emph{polynomial case} an arbitrarily small exponent is not sufficient to get rid of $N^d$, instead we must set $s\coloneqq N^{\frac{d +k+(n-k)}{D}}$ and get
\[
N^dN^k s^{\max\{0,k-\tau\}}s^{-D}=N^dN^k N^{-d-k-(n-k)}= N^{-(n-k)}\quad\text{for}\quad k\leq \tau.
\]
Here, we can only compensate for $N^d$ for $k\leq \tau$, since only the derivatives up to order $\tau$ absorb the effect of the scaling.
\item Finally, in case of an exact PU, term (a) is zero.
\end{itemize}

For term (b) we only consider $m=\wtilde m$. For $k\geq \tau+1$, we now pay the price for the scaling in the exponential case, since there is no exponential decay for the derivative of $\phi_{\wtilde m}^{s}$ on the patch $\Omega_{\wtilde m}.$ From Definition~\ref{def:partition_of_unity}~\eqref{item:pou_derivativeGeneralSigmoid} together with the Bramble-Hilbert Lemma~\ref{lemma:bramble_hilbert} we get the estimate
\[
\norm{\phi_m^s(f-\mathrm{poly}_{\wtilde m})}_{\Wkp[k][\infty][\Omega_{\wtilde m}]}\lesssim \begin{cases} N^{-(n-k-\drate)},&\text{if \emph{exponential PU,}}\\
        N^{-(n-k)},\quad & \text{for } k\leq \tau,\text{ if \emph{polynomial PU,}}\\
		N^{-(n-k)}, &\text{if \emph{exact PU}}.
		\end{cases}
\]
Combining the computations for (a) and (b) we get the total estimate in Step~1
\begin{align*}
\norm[\bigg]{\sum_{m}f-\phi_{m}^s\pp}_{\Wkp[k][\infty][\Omega_{\wtilde m}]}\lesssim\begin{cases} N^{-(n-k-\drate)},&\text{if \emph{exponential PU,}}\\
        N^{-(n-k)},\quad &\text{for } k\leq \tau, \text{ if \emph{polynomial PU,}}\\
		N^{-(n-k)}, &\text{if \emph{exact PU}}.
		\end{cases}
\end{align*}
By choosing $N\coloneqq \ceil{\eps^{-1/(n-k-\drate)}}$ in the exponential case and $N\coloneqq \ceil{\eps^{-1/(n-k)}}$ in the other two cases, we get that the term from Step~1 can be bounded by $\eps$.

\noindent \textbf{Step 2:} To construct the neural network we use the results from~Section~\ref{subsec:IngredientII} to 
\begin{enumerate}[(i)]
\item  approximate Taylor polynomials by neural networks; 
\item approximate the multi-dimensional PU since only the $d$ factors of its tensor structure can be exactly represented by a neural network (see Definition~\ref{def:partition_of_unity}~\eqref{item:pou_networkGeneralSigmoid}), their multiplication must be approximated;  
\item  approximate the multiplication of (i) with (ii) by neural networks $\Phi_{m,\widetilde \eps}$ with accuracy $\widetilde \eps$ (chosen below); 
\item  build the sum of all approximations of localized Taylor polynomials by neural networks.
\end{enumerate} The network $\Phi_{P,\eps}$ thus consists of the subnetworks from step (c). We get the estimate
\begin{align*}
\norm[\bigg]{\sum_m \phi_m^s \pp- \act{\Phi_{P,\eps}}}_{\Wkp[k][\infty][\cube^d]}\leq \sum_{m}\norm{\phi_m^s \pp- \Phi_{m,\widetilde \eps}}_{\Wkp[k][\infty][\cube^d]}\lesssim N^d \widetilde \eps.
\end{align*}
Consequently, we need to chose $\widetilde \eps\coloneqq \eps N^{-d}\approx\eps^{-d/(n-k-\drate)+1}$ (some terms are suppressed here for simplicity of exposition). We can only do this, since neither the number of weights of $\Phi_{m,\widetilde\eps}$ nor its number of layers depends on $\widetilde \eps$ (only the values of the weights do). In other words, each $\Phi_{m,\widetilde \eps}$ has a constant number of weights and layers. Combining $\sim N^d$ of such networks to get $\Phi_{P,\eps}$ yields a network with about $N^d=\eps^{-d/(n-k-\drate)}$ weights and constant number of layers for the exponential case (with obvious adaptations for the other two cases).

\textbf{Conclusion:} For activation functions $\rho$ with an exponential PU, we obtain optimal rates for Sobolev norms $k\leq\tau$ and almost optimal rates for $k\geq \tau+1$;  in the polynomial case, we get optimal approximation rates only in $W^{k,p}$-norms if $k\leq \tau$; in the case of an exact PU, we get optimal approximation rates for Sobolev norms up to order $j$ (smoothness of $\rho$).
\end{overview*}

We now give the statement of Proposition~\ref{prop:main}, which can be proven by using the ideas and concepts presented so far in this section. The detailed proofs are executed in Appendices \ref{app:Bumps}-\ref{app:PutTogether}, mostly for the case of \emph{exponential} $(j,\tau)$-PUs. The statements for the other two cases can be proven in an analogous way. 
\begin{proposition}\label{prop:main}
We make the following assumptions:
\begin{itemize}
    \item Let $d\in \N$, $j,\tau\in\N_0, k\in\{0,\ldots,j\}, n\in\N_{\geq k+1}$, $1\leq p\leq \infty$,  and $\mu>0$;
    \item let $\rho:\R\to\R$ such that $\left(\Psi^{(j,\tau,N,s)}(\varrho)\right)_{N\in\N,s\geq 1}$ is an exponential (polynomial, exact) $(j,\tau)$-PU;
    \item there exists $x_0\in\R$ such that $\rho$ is three times continuously differentiable in a neighborhood of $x_0$ and $\rho''(x_0)\neq 0$. 
\end{itemize}
	  Then, there exist constants $L, C,\theta,\widetilde\eps$ depending on $d,n,p,k, \mu$ with the following properties:
	
	For every $\epsilon \in (0,\widetilde \eps)$ and every $f\in\Fndp$, there is a neural network $\Phi_{\eps,f}$ with $d$-dimensional input and one-dimensional output, at most $L$ layers and at most \begin{align*}\begin{cases} C\eps^{-d/(n-k-\drate)}, \quad &\text{ if exponential PU }, \\
	C\eps^{-d/(n-k)}, \quad &\text{ for } k\leq \tau, \text{ if polynomial PU }, \\
	C\eps^{-d/(n-k)}, \quad &\text{ if exact PU }, 
	\end{cases}\end{align*}
	nonzero weights bounded in absolute value by $C\eps^{-\theta}$ such that 
	\[
		\norm{\act{\Phi_{\eps,f}} - f}_{\Wkp[k][p][\cube^d]}\leq \eps.
	\]
\end{proposition}

The main theorem now states that Proposition~\ref{prop:main} also holds with encodable weights, i.e.\ for each $\eps>0$,  every element of the set of weights $W_\eps = \bigcup_{f}{W_{\eps,f}}$ (where $W_{\eps,f}$ denotes the weights of $\Phi_{\eps,f}$) can be uniquely encoded by $\ceil{C\log_2(1/\eps)}$ bits. To state this in a formal way, we use the notation introduced in Equation~\eqref{eq:encod_notation}. 
\begin{theorem}\label{thm:main}
We make the following assumptions:
\begin{itemize}
    \item Let $d\in \N$, $j,\tau\in\N_0, k\in\{0,\ldots,j\}, n\in\N_{\geq k+1}$, $1\leq p\leq \infty$,  and $\mu>0$;
    \item let $\rho:\R\to\R$ such that $\left(\Psi^{(j,\tau,N,s)}(\varrho)\right)_{N\in\N,s\geq 1}$ is an exponential (polynomial, exact) $(j,\tau)$-PU;
    \item there exists $x_0\in\R$ such that $\rho$ is three times continuously differentiable in a neighborhood of $x_0$ and $\rho''(x_0)\neq 0$. 
\end{itemize}
 Then, there exist constants $L,C$ and $\widetilde\eps$, and a coding scheme $\mathcal{B}=(B_\ell)_{\ell\in\N}$ depending on $d,n,p,k, \mu$ with the following properties: 
	
	For every $\epsilon \in (0,\widetilde \eps)$ and every $f\in\Fndp$, there is a neural network $\Phi_{\eps,f}\in \mathcal{NN}^{\mathcal{B}}_{M_\eps,\ceil{C\log_2(1/\eps)},d}$ with $d$-dimensional input, one-dimensional output, at most $L$ layers and at most 
 \begin{align*}	M_\eps=\begin{cases} C\eps^{-d/(n-k-\drate)}, \quad &\text{ if exponential PU }, \\
	C\eps^{-d/(n-k)}, \quad &\text{ for } k\leq \tau, \text{ if polynomial PU }, \\
	C\eps^{-d/(n-k)}, \quad &\text{ if exact PU }, 
	\end{cases}\end{align*}
nonzero weights, such that 
	\[
		\norm{\act{\Phi_{\eps,f}} - f}_{\Wkp[k][p][\cube^d]}\leq \eps.
	\]
\end{theorem}
\begin{proof}
We give a short outline of the proof here, the details can be found in Appendix~\ref{app:Encod}.
Let $\Phi_{\eps,f}=((A_1,b_1),\dots,(A_{L-1},b_{L-1}),(A_L,b_L))$ be the network from~Proposition~\ref{prop:main} (where the main work has already been done). From the proof of the proposition (see~Equation~\eqref{eq:final_network}) it follows that that $A_L= A_f\widetilde{A}_L$ and $b_L= A_f \widetilde{b}_L$ where the entries of the block diagonal matrix $A_f$ depend on $f$ and the entries of $ A_1,b_1,\dots,A_{L-1},b_{L-1},\widetilde{A}_L, \widetilde{b}_L$ are independent from $f$ (i.e.,\ they only depend on $\eps,n,d,p,k,\mu$).  We denote the collection of nonzero entries of $A_1,b_1,\dots,A_{L-1},b_{L-1},\widetilde{A}_{L},\widetilde{b}_{L}$ by ${W}_\eps$. 
\begin{itemize}
    \item The number of independent weights $\pabs{W_\eps}$ is bounded by $C\cdot\eps^{-d/(n-k-\drate)}$ since the total number of nonzero weights is bounded by this quantity.
    \item We round the entries of $A_f,b_f$ with a suitable precision $\nu$ to the mesh $[-\eps^{-\theta},\eps^{-\theta}]\cap\eps^\nu\Z$, where we also use the fact that the weights of $\Phi_{\eps,f}$ are bounded in absolute value by $C\eps^{-\theta}$.
         \item The nonzero entries of $A_L$ in the last layer of $ \Phi_{\eps,f}$ are in the set $G_{\mathrm{mult}}\coloneqq\{x_1x_2: x_1\in W_\eps, x_2\in \mesh\}$ with cardinality bounded by $\eps^{-\tilde s}$ (similar for $b_L$).
\end{itemize}
Hence, the weights of the approximating neural networks can be chosen from a set $\widetilde{W}_\eps$ with less than $\eps^{-s}$ real numbers, where $s>0$ only depends on $d,n,p,k,\mu$ and not on $f$. Consequently, there exists a surjective mapping $B_\eps:\{0,1\}^{\ceil{s\log_2(1/\eps)}}\to W_\eps$. The collection of these maps constitutes the coding scheme.
\end{proof}

\begin{small}
\begin{centering}
\renewcommand{\arraystretch}{2}
\begin{tabularx}{\textwidth}{|m{3.4cm}|m{2.7cm}|m{2.9cm}|m{2.1cm}|m{3.2cm}|}
\hline
  \textbf{Name}
& \textbf{Given by}
& \textbf{Smoothness} \linebreak \textbf{Boundedness}
& \textbf{PU-Decay \linebreak $(j,\tau)$} & \textbf{Approximation Rates} ($k\leq j$)
\\ \hline 
(leaky) ReLU, $a\in[0,1)$
& $\max\{ax,x\}$
& ${C(\R)\cap W^{1,\infty}_{\mathrm{loc}}(\R)}$ \linebreak Unbounded
& exact \linebreak $(1,1)$ & $\eps^{-d/(n-k)}\log(1/\eps)  $
\\ \hline
  exponential linear unit ($\mathrm{ELU}_a$), $a>0,~a\neq 1$
& $\begin{aligned}
    \begin{cases}x, \quad &{x\geq 0} \\ 
        a(e^x-1),\quad  & x<0 \end{cases} 
\end{aligned}  $
& ${C(\R)\cap W^{1,\infty}_{\mathrm{loc}}(\R)}$ \linebreak Unbounded
& exponential \linebreak $(1,1)$ & $\eps^{-d/(n-k)} $
\\ \hline
  exponential linear unit ($\mathrm{ELU}_1$)
& $\begin{aligned} \begin{cases}x, \quad &{x\geq 0} \\ 
        e^x-1,\quad  & x<0 \end{cases} 
\end{aligned}$ & ${C^1(\R)\cap W^{2,\infty}_{\mathrm{loc}}(\R)}$ \linebreak Unbounded & exponential \linebreak (2,1) & $\begin{alignedat}{2}&\eps^{-d/(n-k)} &&\text{ for } k\leq 1,\\ &\eps^{-d/(n-2-\mu)} &&\text{ for } k=2\end{alignedat}$ 
\\ \hline 
  softsign
& $\displaystyle\frac{x}{1+|x|}$
& ${C^1(\R)\cap W^{2,\infty}(\R)}$ \linebreak Bounded
& polynomial \linebreak $(2,0)$ & $ \eps^{-d/(n-k)} $\linebreak for $k=0$
\\ \hline
  inverse square root linear unit, $a>0$
&$\begin{alignedat}{1}
    \begin{cases}x, \quad &{x\geq 0} \\ 
        \frac{x}{\sqrt{1 + a x^2}},\quad  & x<0 \end{cases} 
\end{alignedat}   $
& ${C^2(\R)\cap W^{3,\infty}_{\mathrm{loc}}(\R)}$ \linebreak Unbounded
&  polynomial \linebreak $(3,1)$ & $\eps^{-d/(n-k)} $ \linebreak for $k\leq 1$
\\ \hline
  inverse square root unit, $a>0$
& $\displaystyle\frac{x}{\sqrt{1 + a x^2}}$ 
& Analytic \linebreak Bounded
& polynomial \linebreak $(j,0)$ $\forall j\in\N_0$ &$ \eps^{-d/(n-k)} $ \linebreak for $k=0$
\\ \hline
  sigmoid / logistic
& $\displaystyle\frac{1}{1+e^{-x}}$
& Analytic \linebreak Bounded
& exponential \linebreak $(j,0)$ $\forall j\in\N_0$ & $\begin{alignedat}{2}&\eps^{-d/n} &&\text{ for } k=0,\\ &\eps^{-d/(n-k-\mu)} &&\text{ for } k\geq 1\end{alignedat}$  
\\ \hline
  tanh
& $\displaystyle \frac{e^x-e^{-x}}{e^x+e^{-x}}$
& Analytic \linebreak Bounded
& exponential \linebreak $(j,0)$ $\forall j\in\N_0$ & $\begin{alignedat}{2}&\eps^{-d/n} &&\text{ for } k=0,\\ &\eps^{-d/(n-k-\mu)} &&\text{ for } k\geq 1\end{alignedat}$ 
\\ \hline
  arctan
& $\arctan(x)$
& Analytic \linebreak Bounded
& polynomial \linebreak $(j,0)$ $\forall j\in\N_0$ &$ \eps^{-d/(n-k)} $ \linebreak for $k=0$
\\ \hline
  softplus
& $\ln(1+e^x)$ 
& Analytic \linebreak Unbounded
& exponential \linebreak $(j,1)$ $\forall j\in\N_0$ & $\begin{alignedat}{2}&\eps^{-d/n} &&\text{ for } k\leq 1,\\ &\eps^{-d/(n-k-\mu)} &&\text{ for } k\geq 2\end{alignedat}$ 
\\ \hline
 swish
& $\displaystyle \frac{x}{1+e^{-x}}$
& Analytic \linebreak Unbounded
& exponential \linebreak $(j,1)$ $\forall j\in\N_0$ & $\begin{alignedat}{2}&\eps^{-d/n} &&\text{ for } k\leq 1,\\ &\eps^{-d/(n-k-\mu)} &&\text{ for } k\geq 2\end{alignedat}$ 
\\ \hline
 rectified power unit (RePU), $a\in \N_{\geq 2}$
& $\max\{0,x\}^a$
& $C^{a-1}(\R)\cap W^{a,\infty}_{\mathrm{loc}}(\R)$ \linebreak Unbounded
& exact \linebreak $(a,a)$ & $\eps^{-d/(n-k)}  $
\\ \hline
 \caption{Commonly-used activation functions, the type of PU they admit and the approximation rates in $W^{k,p}$ in terms of the number of nonzero weights. The rates are provided by Theorem~\ref{thm:main} and, for the (leaky) ReLU case, in combination with Remark~\ref{rem:Plug}. The results for the (leaky) ReLU are consistent with those rates derived in \cite{yarotsky2017error,guhring2019error}. $\mu>0$ is arbitrary and, unless specified otherwise, $k\in \{0,\dots,j\}$ and $n\geq k+1$.} 
 \label{tab:ActFunctions}
\end{tabularx}
\end{centering}
\end{small}

\begin{remark}[Plug \& Play] \label{rem:Plug}
Some well-known activation functions, e.g.,\ the (leaky) ReLU, do not fulfill all assumptions stated in~Proposition~\ref{prop:main} and Theorem~\ref{thm:main} ($\rho$ should be three times continuously differentiable in a neighbourhood of some $x_0\in\R$ with $\rho''(x_0)\neq 0$). However, we note that our proof strategy only requires the approximation of monomials and an approximate multiplication. In case of the (leaky) ReLU this can be done with $\mathcal{O}(\log(1/\eps))$ weights and layers (see~\cite[Proposition~2 and~3]{yarotsky2017error}). Generally speaking: As long as an activation allows for 
\begin{itemize}
    \item the construction of an (approximate) PU along the lines of Definition~\ref{def:partition_of_unity},
    \item  an efficient approximation of polynomials and the identity function,
\end{itemize} our proof strategy can be employed to yield efficient convergence rates. As such, our framework is very general and unifies several previous approaches (e.g.\ \cite{yarotsky2017error, guhring2019error}) as well as extends the previously known rates to a very general class of activation functions and rather general smoothness norms.
\end{remark}

\begin{remark}[Tightness of the Bounds]
     From Corollary~\ref{cor:sobolev_low_bound_enc} (\ref{item:lowerbounds_sobolev}) it follows that our bounds for encodable neural network weights are tight up to a log factor for $k\leq \tau$. For $k\geq \tau +1$, they are (up to a log factor) tight in the case of exact PUs and we get arbitrarily close to the optimal bound (again up to a log factor) in case of exponential PUs. If we allow for \emph{arbitrary} weights, then this upper bound might be drastically improved (see~Remark~\ref{rem:lower_bounds_nonenc}).
\end{remark}

\section*{Acknowledgments}
The authors would like to thank Philipp Petersen for fruitful discussions on the topic. Moreover, they would like thank the anonymous reviewers for suggestions to improve the manuscript. 
I.\ Gühring acknowledges support from the Research Training Group ``Differential Equation- and Data-driven Models in Life Sciences and Fluid Dynamics: An Interdisciplinary
Research Training Group (DAEDALUS)'' (GRK 2433) funded by the German Research Foundation (DFG).


\appendix

\section{Notation and Auxiliary Results}\label{app:Not}
In this subsection, we depict the (mostly standard) notation used throughout this paper. We set $\N\coloneqq\{1,2,\ldots\}$ and $\N_0\coloneqq \N\cup \{0\}$. For $k\in\N_0$ we define $\N_{\geq k}\coloneqq\{k,k+1,\ldots\}$. For a set $A$ we denote its cardinality by $\lvert A\rvert\in\N\cup\{\infty\}$ and by $\mathbbm{1}_A$ its indicator function of $A$. If $x\in\R$, then we write $\ceil{x}\coloneqq\min\{k\in\Z:k\geq x\}$ where $\Z$ is the set of integers and $\floor{x}\coloneqq \max\{k\in\Z:k\leq x\}$. 
        
If $d\in\N$ and $\norm{\cdot}$ is a norm on $\R^d$, then we denote for $x\in\R^d$ and $r>0$ by $B_{r,\norm{\cdot}}(x)$ the open ball around $x$ in $\R^d$ with radius $r$, where the distance is measured in $\norm{\cdot}$. By $\pabs{x}$ we denote the euclidean norm of $x$ and by $\norm{x}_{\infty}$ the maximum norm. We endow $\R^d$ with the standard topology and for $A\subset \R^d$ we denote by $\overline{A}$ the closure of $A.$

For $d_1,d_2\in\N$ and a matrix $A\in\R^{d_1,d_2}$ the number of nonzero entries of $A$ is counted by $\norm{\cdot}_{0}$, i.e.\
	\[
		\norm{A}_{0}\coloneqq\pabs*{\{(i,j): A_{i,j}\neq 0\}}.
	\]
 If $f:X\to Y$ and $g:Y\to Z$ are two functions, then we write $g\circ f:X\to Z$ for their composition. If additionally $U\subset X,$ then $f|_U:U\to Y$ denotes the restriction of $f$ onto $U$.
We use the usual \emph{multiindex} notation, i.e.\ for $\alpha\in\N_0^d$ we write $\pabs{\alpha}\coloneqq\alpha_1+\ldots + \alpha_d$ and $\alpha!\coloneqq\alpha_1!\cdot\ldots\cdot \alpha_d!$. Moreover, if $x\in\R^d$, then we have
	\[
		x^\alpha\coloneqq\prod_{i=1}^d x_i^{\alpha_i}.
		\]
Let from now on $\Omega\subset\R^d$ be open. For a function $f:\Omega\to\R,$ we denote by
\[
	D^\alpha f\coloneqq\frac{\partial^{\pabs{\alpha}}f}{\partial x_1^{\alpha_1} \partial x_2^{\alpha_2}\cdots \partial x_d^{\alpha_d}}.
\]
its (weak or classical) derivative of order $\alpha$. For $n\in\N_0 \cup \{\infty\}$, we denote by $C^n(\Omega)$ the set of $n$ times continuously differentiable functions on $\Omega$. Additionally, if $\overline{\Omega}$ is compact, we set, for $f\in C^n(\Omega)$ 
\begin{align*}
    \|f\|_{C^n(\overline{\Omega})} \coloneqq \max_{0\leq|\alpha|\leq n}~\sup_{x\in \Omega} |D^{\alpha}f(x)|.
\end{align*}
We denote by $L^p(\Omega),~1\leq p \leq \infty$ the standard Lebesgue spaces. 

In the following, we will also make use of the following well-known fact stating that the exponential function decays faster than any polynomial.

\begin{proposition}\label{prop:ExpPolEst}
Let $\alpha,\beta,c,c'>0.$ Then
\begin{align*}
    \lim_{x\to \infty} \frac{c'x^\alpha}{e^{c\cdot x^\beta}} = 0.
\end{align*}
This implies that for all $\gamma>0$ there exists some constant $C=C(\alpha,\beta,\gamma)>0$ such that for all $x >0$ there holds
\begin{align*}
    \frac{c'x^\alpha}{e^{c\cdot x^\beta}} \leq Cx^{-\gamma}.
\end{align*}
\end{proposition}

\section{Sobolev Spaces}\label{app:Sobolev}
In this section, we introduce Sobolev spaces (see \cite{adams1975sobolev}) which constitute a crucial concept within the theory of PDEs (see e.g.~\cite{roubivcek2013nonlinear,evans1998partial}). 

\begin{definition}
    Given some domain $\Omega\subset\R^d$, $1\leq p< \infty$, and $n\in\N$, the Sobolev space $\Wkp[n]$ is defined as
		\[
			\Wkp[n]:=\left\{f:\Omega\to\R :  \norm{D^\alpha f}_{L^p(\Omega)}^p  < \infty, \text{ for all } \alpha\in\N_0^d\text{ with }\pabs{\alpha}\leq n\right\},
		\]
		and is equipped with the norm
			\[
				\norm{f}_{\Wkp[n]}:=\left(\sum_{0\leq \pabs{\alpha}\leq n}\norm{D^\alpha f}_{L^p(\Omega)}^p\right)^{1/p}.
				\]
Additionally, we set
\[
			\Wkp[n][\infty]:=\left\{f:\Omega\to\R : \norm{D^\alpha f}_{L^\infty(\Omega)} < \infty \text{ for all } \alpha\in\N_0^d\text{ with }\pabs{\alpha}\leq n\right\},
		\]
		and we equip this space with the norm
		$\norm{f}_{\Wkp[n][\infty]}:=\max_{|\alpha|\leq n}\norm{D^\alpha f}_{L^\infty(\Omega)}$.
Moreover, for $0\leq k\leq n,$ on $\Wkp[n]$ we introduce the family of \emph{semi-norms} 

\[|f|_{W^{k,p}(\Omega)} \coloneqq \left( \sum_{|\alpha|=k} \norm{D^\alpha f}_{L^p(\Omega)}^p \right)^{1/p},\qquad 	\pabs{f}_{\Wkp[k][\infty]}:=\max_{|\alpha|=k}\norm{D^\alpha f}_{L^\infty(\Omega)}, \]
respectively. Finally, let $W^{n,p}_{\mathrm{loc}}(\Omega)\coloneqq \{f:\Omega\to\R:~f|_{\widetilde{\Omega}}\in W^{n,p}(\widetilde{\Omega})~\text{ for all compact } \widetilde{\Omega}\subset \Omega\}.$
\end{definition}
\begin{remark}
    If $\Omega$ is bounded and fulfills a local Lipschitz condition, arguments from \cite{adams1975sobolev} show that $W^{2,\infty}(\Omega)$ can be continuously embedded into $C^1(\overline{\Omega})$. This can be seen as follows: \cite[Theorem 4.12]{adams1975sobolev} shows that $W^{2,p}(\Omega)$ can be continuously embedded into $C^1(\overline{\Omega})$ for $p>d.$ Since also $W^{2,\infty}(\Omega)$ can be continuously embedded into $W^{2,p}(\Omega),$ the claim follows.
\end{remark}

\begin{remark}\label{remark:extension_operator}
For purely technical reasons we sometimes make use of an extension operator. For this, let $E:\Wkp[n][p][\cube^d]\to\Wkp[n][p][\R^d]$ be the extension operator from \cite[Theorem~VI.3.1.5]{stein2016singular} and set $\widetilde f\coloneqq Ef$. Note that for arbitrary $\Omega\subset\R^d$ and $0\leq k\leq n$ it holds
	\begin{equation}\label{eq:extension_bound}
		\pabs[\big]{\widetilde f}_{\Wkp[k][p][\Omega]} \leq \norm[\big]{\widetilde f}_{\Wkp[n][p][\R^d]} \leq C \norm{f}_{\Wkp[n][p][\cube^d]},
	\end{equation}
	where $C=C(n,p,d)$ is the norm of the extension operator. 
\end{remark}

The following lemma which will be crucial for the proofs of our results can be stated in much more generality (see~\cite[Chapter 4.1]{brenner2007mathematical}) and relies on the use of averaged Taylor polynomials. We only state a version tailored to our specific needs and will not give a proof since the details of this specific version have been worked out in~\cite[Section B.3 and Lemma C.4]{guhring2019error}.
\begin{lemma}[Bramble-Hilbert]\label{lemma:bramble_hilbert}
	Let $d,n\in\N$ and $1\leq p\leq\infty$. Furthermore, let $N\in\N$ and set for $m\in\MNd$
	\[\Omega_{m,N} \coloneqq B_{\frac{1}{N},\norm{\cdot}_{\infty}}\Big(\frac{m}{N}\Big).
	\]
	 Then there exists a constant $C=C(n,d)>0$ such that for all $f\in\Wkp[n][p][\R^d]$ and $m\in\MNd$ there is a polynomial $p_{m}(x)=\sum_{\pabs{\alpha}\leq n-1} c_{\alpha}x^\alpha$ such that
	\begin{equation*}
		\norm[\big]{f-p_m}_{\Wkp[k][p][\Omega_{ m,N}]}\leq C\left(\frac{1}{N}\right)^{n-k}\norm{f}_{\Wkp[n][p][\Omega_{m,N}]},\quad\text{for } k=0,1,\ldots,n
	\end{equation*}
	and the coefficients $c_\alpha$ are bounded by $\pabs{c_\alpha}\leq C N^{d/p}\norm{f}_{\Wkp[n][p][\Omega_{m,N}]}$ for all $\alpha$ with $\pabs{\alpha}\leq n-1$.
\end{lemma}

Now we turn our attention to a version of a product rule tailored to our needs.

\begin{lemma}\label{lemma:product_rule_bound}
	Let $k\in\N$, and assume that $f\in \Wkp[k][\infty][\Omega]$ and $g\in \Wkp[k][p][\Omega]$ with $1\leq p\leq \infty.$ If $k\geq 3,$ additionally assume that $f\in C^k(\Omega)$ or $g\in C^k(\Omega).$ Then $fg\in\Wkp[k][p][\Omega]$ and there exists a constant $C=C(d,p,k)>0$ such that
	\[
		\norm{fg}_{\Wkp[k][p]}\leq C \sum_{i=0}^k\norm{f}_{\Wkp[i][\infty]} \norm{g}_{\Wkp[k-i][p]},
		\]
	and, consequently
	\[
		\norm{fg}_{\Wkp[k][p]}\leq C \norm{f}_{\Wkp[k][\infty]} \norm{g}_{\Wkp[k][p]}.
		\]
\end{lemma}
\begin{proof}
For $k=0$ the statement is obvious. 

For $k=1$ we get from~\cite[Lemma~B.6]{guhring2019error} that there exists a constant $C=C(d,p)>0$ such that
	\[
		\pabs{fg}_{\Wkp[1][p]}\leq C \left(\norm{f}_{\Wkp[1][\infty]} \norm{g}_{\Lp} + \norm{f}_{\Linf} \norm{g}_{\Wkp[1][p]}\right),
		\]
		 from which the statement can easily be deduced.
		 
For $k=2$ it follows from \cite[Chap.~7.3]{gilbarg1998elliptic} that the usual product rule also holds for the second order derivatives such that we have
\begin{align*}
		\pabs{fg}_{\Wkp[2][p]}&\leq C \sum_{i,j=1,\ldots,d}\left\|\frac{\partial^2 }{\partial x_i\partial x_j} fg\right\|_{\Lp[p]} +\left\|\frac{\partial }{\partial x_i}f \frac{\partial }{\partial x_j}g\right\|_{\Lp[p]}+\left\|\frac{\partial }{\partial x_j}f \frac{\partial }{\partial x_i}g \right\|_{\Lp[p]}+\left\|f \frac{\partial^2}{\partial x_i\partial x_j}g\right\|_{\Lp[p]} \\
		&\leq C \left(\norm{f}_{\Wkp[2][\infty]} \norm{g}_{\Lp} + \norm{f}_{\Wkp[1][\infty]} \norm{g}_{\Wkp[1][p]} +\norm{f}_{\Linf} \norm{g}_{\Wkp[2][p]}\right).
\end{align*}
Again the overall statement follows easily.
The statement for $k\in\N_{\geq 3}$ can directly be concluded from the Leibniz formula (see \cite[Lemma 8.18]{LectureProduct}) 
\begin{align*}
    D^\alpha(fg)=\sum_{|\beta|\leq |\alpha|} \binom{\alpha}{\beta} D^\beta f D^{\alpha-\beta} g.
\end{align*}

\end{proof}

The following corollary establishes a chain rule estimate for $W^{k,\infty}$.
\begin{corollary}\label{cor:composition_norm}
	Let $d,m\in\N,$ $k\in\N_{\geq 2}$ and $\Omega_1\subset \R^d$, $\Omega_2\subset\R^m$ both be open, bounded, and convex. Then, there is a constant $C=C(d,m,k)>0$ with the following properties:
	\begin{enumerate}[(i)]
	    \item If $k=2$ and $f\in \Wkpm[2][\infty][\Omega_1][m]\cap C^1(\Omega_1;\R^m)$ and $g\in \Wkp[2][\infty][\Omega_2]\cap C^1(\Omega_2)$ such that $\mathrm{Range} (f)\subset \Omega_2$, then for the composition $g\circ f$ it holds that $g\circ f\in \Wkp[2][\infty][\Omega_1]\cap C^1(\Omega_1)$ and we have 
	\begin{align*}
		\pabs{g\circ f}_{\Wkp[1][\infty][\Omega_1]}\leq C \pabs{g}_{\Wkp[1][\infty][\Omega_2]}\pabs{f}_{\Wkpm[1][\infty][\Omega_1][m]},
	\end{align*}
		and
	\begin{align*}
		\pabs{g\circ f}_{\Wkp[2][\infty][\Omega_1]}\leq C \left(\pabs{g}_{\Wkp[2][\infty][\Omega_2]}\pabs{f}_{\Wkpm[1][\infty][\Omega_1][m]}^2+\pabs{g}_{\Wkp[1][\infty][\Omega_2]}\pabs{f}_{\Wkpm[2][\infty][\Omega_1][m]} \right).
	\end{align*}
	
	\item If $k\geq 3,$  $f\in C^k(\overline{\Omega}_1;\R^m)$ and $g\in C^k(\overline{\Omega}_2)$ such that $\mathrm{Range} (f)\subset \Omega_2,$ then for the composition $g\circ f$ it holds that $g\circ f\in C^k(\Omega_1)$ and 
	\begin{enumerate}[(a)]
	    \item if $\pabs{f}_{\Wkpm[l][\infty][\Omega_1]}\leq C N^l$ for all $l=1,\dots,k$, then
	\begin{equation}\label{eq:multi_dim_chainrule}
	    \pabs{g\circ f}_{\Wkp[k][\infty][\Omega_1]}\leq C\sum_{l=1}^k \pabs{g}_{\Wkp[l][\infty][\Omega_2]}N^k;
	\end{equation}
	\item 
	if $\tau\in\N_0$ and $\pabs{f}_{\Wkpm[l][\infty][\Omega_1]}\leq C N^{l+\mu\max\{0, l-\tau\}}$ for all $l=1,\dots,k$, then
	\begin{equation}\label{eq:mutli_dim_chainrule_mu}
	    \pabs{g\circ f}_{\Wkp[k][\infty][\Omega_1]}\leq C\sum_{l=1}^k \pabs{g}_{\Wkp[l][\infty][\Omega_2]}N^{k+\drate}.
	\end{equation}
	\end{enumerate}
	\end{enumerate}
\end{corollary}
\begin{proof}
(i) can be shown by basic computations using the classical first derivative and~\cite[Corollary~B.5, Lemma~B.6]{guhring2019error}.
For (ii), we make use of the multivariate Faa Di Bruno formula (see~\cite[Theorem~2.1]{ChainRule}) and get that
\begin{equation*}
	    \pabs{g\circ f}_{\Wkp[k][\infty][\Omega_1]}\leq C\max_{\pabs{\nu}=k} \sum_{l=1}^k \pabs{g}_{\Wkp[l][\infty][\Omega_2]}\sum_{|\lambda|=l}\sum_{p(\nu,\lambda)}\prod_{j=1}^k\pabs{f}_{\Wkpm[\pabs{l_j}][\infty][\Omega_1][m]}^{\pabs{r_j}},
	\end{equation*}
	where 
\begin{equation*}
	p(\nu,\lambda):=\left\{\begin{array}{c}
		(r_1,\ldots,r_k;l_1,\ldots,l_k): \text{for some }1\leq s\leq k, r_i=0 \text{ and }l_i=0 \text{ for }1\leq i\leq k-s;\\[0.3em]
		\pabs{r_i}>0\text{ for } k-s+1\leq i\leq k; \text{ and }0\leq l_{k-s+1}\leq \ldots\leq l_k \text{ are such that }\\[0.3em]
		\sum_{i=1}^k r_i=\lambda, \sum_{i=1}^k\pabs{r_i}l_i = \nu.
	\end{array}
		\right\}.   
	\end{equation*}
Equation~\eqref{eq:multi_dim_chainrule} now follows from $\prod_{j=1}^k\pabs{f}_{\Wkpm[\pabs{l_j}][\infty][\Omega_1][m]}^{\pabs{r_j}}\leq C \prod_{j=1}^k N^{\pabs{l_j}\pabs{r_j}}=C N^{\sum_{j=1}^k \pabs{l_j}\pabs{r_j}}=C N^{k}$. Equation~\eqref{eq:mutli_dim_chainrule_mu} for $\tau=0$ follows from (a) with $N=N^{1+\mu}$. 
For $\tau\geq  1$, we have 
\begin{align*}
\prod_{j=1}^k N^{\mu\max\{0, \pabs{l_j}-\tau\}\pabs{r_j}}= N^{\mu \sum_{j=1}^k \max\{0,\pabs{l_j}-\tau\}\pabs{r_j}}
\end{align*}
and
\[
\sum_{j=1}^k \max\{0,\pabs{l_j}-\tau\}\pabs{r_j}=\sum_{j:\pabs{l_j}\geq \tau }^k (\pabs{l_j}-\tau )\pabs{r_j}.
\]
If $|l_j|< \tau$ for all $j=1,\dots,k$, then $\sum_{j:\pabs{l_j}\geq \tau }^k (\pabs{l_j}-\tau )\pabs{r_j}=0\leq \mu \max\{0,k-\tau\}.$ If there exists some $j'$ with $|l_{j'}|\geq \tau$ and $|r_{j}|=0$ for all $j$ with $|l_{j}|\geq \tau,$ then also $\sum_{j:\pabs{l_j}\geq \tau }^k (\pabs{l_j}-\tau )\pabs{r_j}=0\leq \mu \max\{0,k-\tau\}.$ Otherwise, there exists some $j'$ with $|l_{j'}|\geq \tau$ and $|r_{'j}|\geq 1.$ We then have 
$$\sum_{j:\pabs{l_j}\geq \tau }^k (\pabs{l_j}-\tau )\pabs{r_j} \leq \sum_{j:\pabs{l_j}\geq 1 }^k \pabs{l_j}\pabs{r_j}- \tau \sum_{j:|l_j|\geq \tau} |r_j| = k- \tau \sum_{j:|l_j|\geq \tau} |r_j| \leq k-\tau |l_{j'}|\,|r_{j'}| \leq k-\tau$$
from which the statement in combination with (a) follows.
\end{proof}

\section{Neural Network Calculus}\label{app:NNCalc}

In this section, we introduce several operations one can perform with neural networks, namely 
the \emph{concatenation} and the \emph{parallelization} of neural networks. Moreover, Section \ref{app:CalcMon} is devoted to approximations of polynomials. We give the proof of Proposition \ref{prop:MonApp} (approximation of monomials by neural networks) and show how to derive approximations of the identity function as well as of approximate multiplications.

We first consider the \emph{concatenation} of two neural networks as given in~\cite{petersen2017optimal}.
\begin{definition}
    Let $\Phi^1= \left((A^1_1,b^1_1),\dots,(A^1_{L_1},b_{L_1}^1 \right)$ and $\Phi^2=\left((A^1_1,b^1_1),\dots,(A^1_{L_1},b_{L_1}^1 \right)$ be two neural networks such that the input dimension of $\Phi^1$ is equal to the output dimension of $\Phi^2.$ Then the \emph{concatenation} of $\Phi^1,\Phi^2$ is defined as the $L_1+L_2-1$-layer neural network
    \begin{align*}
        \Phi^1\conc\Phi^2\coloneqq \left((A_1^2,b_1^2),\dots,(A_{L_2-1}^2,b_{L_2-1}^2),(A_1^1A_{L_2}^2,A_1^1b_{L_2}^2+b_1^1),(A_2^1,b_2^1),\dots,(A_{L_1}^1,b_{L_1}^1) \right).
    \end{align*}
\end{definition}

It is easy to see that $\Realization(\Phi^1\conc\Phi^2)=  \Realization(\Phi^1)\circ \Realization(\Phi^2)$.

Now, we introduce the \emph{parallelization} of neural networks with the same number of layers, inspired by the construction in \cite{petersen2017optimal}. 
\begin{lemma}\label{lem:ParalSame}
    Let $\varrho:\R\to \R$. Additionally, let $\Phi^1,\dots \Phi^n$ be neural networks with $d$-dimensional input and $L\in \N $ layers, respectively. Then,
    there exists a neural network
    $\mathrm{P}(\Phi^1,\dots,\Phi^n)$ with $d$-dimensional input and
    \begin{enumerate}[(i)]
        \item \label{item:parallel_i}
       There holds $\Realization\left(\mathrm{P}(\Phi^1,\dots,\Phi^n)\right)(x)= \left(\Realization(\Phi^1)(x),\dots,\Realization(\Phi^n)(x) \right)$ for all $x\in \R^d.$ ;
       \item \label{item:parallel_ii} $L$ layers;
       \item $M\left(\mathrm{P}(\Phi^1,\dots,\Phi^n)\right)= \sum_{i=1}^n M(\Phi^i)$; 
       \item
        $\left\| \mathrm{P}\left(\Phi^1,\dots,\Phi^n \right) \right\|_{\max} =
         \max\left\{\left\|\Phi^1\right\|_{\max},\dots,\left\|\Phi^n\right\|_{\max} \right\}$.
    \end{enumerate}
    \end{lemma}

\begin{proof} The neural network  \begin{align*}
        \mathrm{P}(\Phi^1,\dots,\Phi^n)\coloneqq \left((\widetilde{A}_1,\widetilde{b}_1),\dots, (\widetilde{A}_L,\widetilde{b}_L) \right),
    \end{align*}
    with 
    \begin{align*}
        \widetilde{A}_1\coloneqq \begin{pmatrix} A_1^1 \\ \vdots \\ A_1^n \end{pmatrix}, \quad \widetilde{b}_1 \coloneqq \begin{pmatrix} b_1^1 \\ \vdots \\ b_1^n \end{pmatrix} \quad \text{ and } \widetilde{A}_\ell\coloneqq \left( \begin{array}{cccc}
A^1_\ell &  &  &    \\
 &A^2_\ell &  &    \\
 &  & \ddots &   \\
 &  &  & A^n_L 
\end{array} \right), \widetilde{b}_\ell \coloneqq \begin{pmatrix} b_\ell^1 \\ \vdots \\ b_\ell^n \end{pmatrix}, \quad \text{ for } 1<\ell\leq L, 
    \end{align*}
    fulfills all the desired properties.
\end{proof} 

\subsection{Approximate Monomials and Multiplication}\label{app:CalcMon}
We first give the proof of Proposition \ref{prop:MonApp}:
    \begin{proof}[Proof of Proposition \ref{prop:MonApp}] Choose $C_0>1$ so that $[x_0-\frac{nB}{C_0},x_0+\frac{nB}{C_0} ]\subset U.$ Moreover, let $\delta\geq C_0$ be arbitrary. 
Define the function
\begin{align*}
    \varrho_{\delta}^r:\R\to \R,~~ x\mapsto \frac{\delta^r}{\varrho^{(m)}(x_0)} \sum_{j=0}^r (-1)^j\binom{r}{j}\cdot \varrho\left(x_0- j\frac{x}{\delta}\right).
\end{align*}
Then $\varrho_{\delta}^r|_{[-B,B]}\in C^{n+1}([-B,B])$.
Using the Taylor expansion and the following identity from~\cite{katsuura2009summations}
\begin{equation}\label{eq:katsuura}
    \sum_{j=1}^r (-1)^{j}\binom{r}{j} j^k= \begin{cases} 
                                    0,\hfill &\text{ if } 1\leq k< r, \\
                                    (-1)^r r!, \hfill \quad &\text{ if } k=r,
                            \end{cases}
\end{equation}
it can easily be shown that $\varrho_{\delta}^r(x)\approx x^r$ for $\delta>0$ sufficiently large. In detail, we have by Taylor's Theorem  (where $\xi_j$ is between $x_0$ and $x_0-j\frac{x}{\delta}$ for $j=1,\dots,r$) that
\begin{align*}
    &\sum_{j=0}^r (-1)^j\binom{r}{j}\cdot \varrho\left(x_0- j\frac{x}{\delta}\right) \\ &= \varrho(x_0)+\sum_{j=1}^r (-1)^j\binom{r}{j}\cdot \left(\sum_{k=0}^r  \frac{\varrho^{(k)}(x_0)}{k!}\left(\frac{-jx}{\delta}\right)^{k} + \frac{\varrho^{(r+1)}(\xi_j)}{(r+1)!}\left(\frac{-(r+1)x}{\delta}\right)^{r+1}\right) \\
    &=\varrho(x_0)+\sum_{k=0}^r\left(\frac{-x}{\delta}\right)^k \frac{\varrho^{(k)}(x_0)}{k!}\sum_{j=1}^r (-1)^{j}\binom{r}{j} j^k +\underbrace{\sum_{j=1}^r (-1)^j \binom{r}{j}  \frac{\varrho^{(r+1)}(\xi_j)}{(r+1)!}\left(\frac{-(r+1)x}{\delta}\right)^{r+1}}_{\eqqcolon r_{\delta}^r(x)}\\
    &=\varrho(x_0)\underbrace{\sum_{j=0}^r (-1)^j\binom{r}{j} }_{=0}+\sum_{k=1}^r\left(\frac{-x}{\delta}\right)^k \frac{\varrho^{(k)}(x_0)}{k!}\underbrace{\sum_{j=1}^r (-1)^{j}\binom{r}{j} j^k}_{\text{use Eq.~\eqref{eq:katsuura}}} +r_{\delta}^r(x)\\
    &=\left(\frac{x}{\delta}\right)^r\rho^{(r)}(x_0) +r_{\delta}^r(x).
\end{align*}
 Hence, for every $k=0,\dots,n$ and every $x\in [-B,B],$ we have 
\begin{align*}
     \left|(\varrho_{\delta}^r)^{(k)}(x)  - (x^r)^{(k)} \right| &= \left| \frac{\delta^r}{\varrho^{(r)}(x_0)} (r_{\delta}^r)^{(k)}(x)\right| \\ &\leq \underbrace{\sum_{j=1}^r  \binom{r}{j} \cdot  \left| \frac{\varrho^{(r+1)}(\xi_j)}{(r+1)!} \right|}_{\leq 2^{n}\|\varrho\|_{C^{n+1}(U)}} \cdot \underbrace{\left|\frac{\delta^r}{\varrho^{(r)}(x_0)} \left(\frac{-(r+1)}{\delta}\right)^{r+1}\right|}_{\leq \frac{(n+1)^{n+1}}{\delta\min_{i=0,\dots,n}|\varrho^{(i)}(x_0)|}} \cdot \underbrace{\vphantom{\left|\frac{\delta^r}{\varrho^{(r)}(x_0)} \left(\frac{-(r+1)}{\delta}\right)^{r+1}\right|}\left| (x^{r+1})^{(k)}\right|}_{\leq n! \max\{B,1\}^{n+1}} \\
    &\leq  2^n\cdot (n+1)^{n+1}n!\cdot \frac{\|\varrho\|_{C^{n+1}(U)}}{\min_{i=0,\dots,n}|\varrho^{(i)}(x_0)|} \max\{B,1\}^{n+1} \cdot \frac{1}{\delta} \eqqcolon \frac{C'(B,n,\varrho)}{\delta}.
\end{align*}
This implies, that there exists some $C \geq \max\{C_0,{C'(B,n,\varrho)}\}$ such that for every $\epsilon\in (0,1)$ and the neural network $\Phi^r_\epsilon \coloneqq \left((A_1,b_1),(A_2,b_2) \right)$ with 
\begin{align*}
    A_1&\coloneqq \left(0,-\frac{\eps}{C},\dots,-\frac{r\eps}{C}\right)^T \in \R^{r+1,1}, \\
    b_1&\coloneqq  (x_0,\dots,x_0)^T\in \R^{r+1}, \\
    A_2 &\coloneqq \frac{C^r}{\eps^r \varrho^{(r)}(x_0)}\left((-1)^0 \binom{r}{0},(-1)^1 \binom{r}{1},\dots,(-1)^r \binom{r}{r}   \right)\in \R^{1,r+1}, \\
    b_2 &\coloneqq 0\in\R,
\end{align*}
fulfills 
\begin{align*}
    \left\|\Realization(\Phi_\eps^r) -x^r\right\|_{C^n([-B,B]}\leq \epsilon.
\end{align*}
Moreover, $L\left(\Phi_\eps^r\right)=2$ and $M\left( \Phi_\eps^r\right) \leq 3(r+1)$.

Additionally, for every $k=0,\dots,r$ and for every $x\in [-B,B]$ we have 
\begin{align*}
    \left|\left(\Realization(\Phi_\eps^r)\right)^{(k)}(x) \right| \leq \left\|\left(\Realization(\Phi_\eps^r)\right)^{(k)} -(x^r)^{(k)}\right\|_{C^n([-B,B])} + \left|(x^r)^{(k)}\right| \leq \epsilon + \frac{n!}{(n-k)!} |\max\{1,B\}|^{r-k}. 
\end{align*}
Finally, for all $k=r+1,\dots,n$ we have that 
\begin{align*}
    \left|\left(\Realization(\Phi_\eps^r)\right)^{(k)}(x) \right| \leq \left\|\left(\Realization(\Phi_\eps^r)\right)^{(k)} -(x^r)^{(k)}\right\|_{C^n([-B,B])} + \left|(x^r)^{(k)}\right| \leq \epsilon + 0=\epsilon. 
\end{align*}
This completes the proof.
\end{proof}

Based on Proposition \ref{prop:MonApp}, we are now in a position to introduce neural networks that approximate the map which multiplies two real inputs.
\begin{corollary}\label{cor:approximate_multiplication}
	Let $\varrho\in W^{j,\infty}_{\mathrm{loc}}(\R)$ for some $j\in\N_0$ and $x_0\in \R$ such that $\varrho$ is three times continuously differentiable in a neighborhood of some $x_0\in\R$ and $\varrho''(x_0)\neq 0.$  	
	Let $B> 0$, then there exists a constant $C=C(B,\varrho)>0$ such that for every $\eps \in \epsin$, there is a neural network $\apmult$ with two-dimensional input and one-dimensional output that satisfies the following properties:
	\begin{enumerate}[(i)]
		\item \label{item:network_approximation}$\norm{R_\rho(\apmult)(x,y) - xy}_{\Wkp[j][\infty][(-B,B)^2;dxdy]}\leq \epsilon$;
		\item \label{item:network_derivative}$\norm{R_\rho(\apmult_\eps)}_{\Wkp[j][\infty][(-B,B)^2]}\leq C$;
		\item \label{item:network_complexity_apmult} $L(\apmult)=2$ and $M(\apmult)\leq C$;
		\item\label{item:mult_bounded_weights} $\bweights{\apmult}\leq C\eps^{-2}$.
	\end{enumerate}
\end{corollary}
\begin{proof}
    Let $C$ be the constant from Corollary~\ref{cor:composition_norm} and set $\widetilde \eps\coloneqq \nicefrac{\epsilon}{2C}$. Proposition~\ref{prop:MonApp} yields that there exists a neural network $\Phi^2_{\widetilde\eps}$ with 2 layers and at most $9$ nonzero weights such that  for all $k\in\{0,\dots,j\}$ we have
    $$
        \left|\Realization(\Phi^2_{\widetilde\eps}) -x^2\right|_{W^{k,\infty}([-2B,2B];dx)} \leq \widetilde\eps.
    $$
   
    As in~\cite{yarotsky2017error}, we make use of the polarization identity
    \begin{align*}
        xy= \frac{1}{4} \left((x+y)^2-(x-y)^2 \right)\quad\text{for }x,y\in\R.
    \end{align*}    
    In detail, we define the neural network 
    $$\apmult_{\eps}\coloneqq \left(\left(\frac{1}{4},\frac{-1}{4}\right),0 \right) \conc \mathrm{P}\left(\Phi^{2}_{\widetilde\eps},\Phi^{2}_{\widetilde\eps} \right) \conc \left(\begin{pmatrix} 1 & 1 \\ 1 & -1 \end{pmatrix},0 \right),$$
    which fulfills for all $(x,y)\in \R^2$ that 
    \begin{align*}
        \Realization(\apmult_{\eps})(x,y) = \frac{1}{4}\left(\Realization\left(\Phi^{2}_{\widetilde \eps}\right)(x+y)- \Realization\left(\Phi^{2}_{\widetilde \eps}\right)(x-y)\right).
    \end{align*}
    Now, setting $f:[-2B,2B]\to \R,x\mapsto x^2$ as well as 
    $$
        u:[-B,B]^2\to[-2B,2B],~(x,y)\mapsto x+y\quad \text{and}\quad
          v:[-B,B]^2\to[-2B,2B],~(x,y)\mapsto x-y, 
    $$
    we see that for all $(x,y)\in [-B,B]^2$ there holds
    $
        xy=\nicefrac{1}{4} \left( f\circ u(x,y) - f\circ  v(x,y)\right).
    $
    We estimate 
    \begin{align*}
        &\left\| \Realization(\apmult_\eps)(x,y) - xy\right\|_{W^{k,\infty}([-B,B]^2;dxdy)} \\ &= \frac{1}{4}\left\| \Realization\left(\Phi^{2}_{\widetilde \eps}\right)\circ u- \Realization\left(\Phi^{2}_{\widetilde \eps}\right)\circ v-\left( f\circ u - f\circ  v\right) \right\|_{W^{k,\infty}([-B,B]^2)} \\
        &\leq \frac{1}{4} \left\| \Realization\left(\Phi^{2}_{\widetilde \eps}\right)\circ u- f\circ u  \right\|_{W^{k,\infty}([-B,B]^2)}
        +  \frac{1}{4}\left\| \Realization\left(\Phi^{2}_{\widetilde \eps}\right)\circ v-  f\circ v  \right\|_{W^{k,\infty}([-B,B]^2)},
    \end{align*}
    and directly see for $k=0$ that
    \begin{equation*}
        \left|\Realization(\apmult_\eps)(x,y) - xy\right|_{W^{0,\infty}([-B,B]^2;dxdy)} \leq \frac{2}{4}\norm{\Realization\left(\Phi^{2}_{\widetilde \eps}\right)-x^2}_{\Lpd[\infty][[-B,B]^2]}\leq\frac{1}{2}\widetilde \eps \leq \epsilon.
    \end{equation*}
    
    Now, we proceed with the case $k\in\{1,\dots,j\}$. We first note that 
    \begin{align*}
        |u|_{W^{0,\infty}([-B,B]^2)} &= |v|_{W^{0,\infty}([-B,B]^2)} = 2B, \\
        |u|_{W^{1,\infty}([-B,B]^2)} &=|v|_{W^{1,\infty}([-B,B]^2)} = 1, \\
         |u|_{W^{k,\infty}([-B,B]^2)} &=|v|_{W^{k,\infty}([-B,B]^2)} = 0,~\text{for all } k\geq 2.
    \end{align*}
    The composition rule from Corollary~\ref{cor:composition_norm} then yields that 
      \begin{align*}
        \left|\Realization(\apmult_\eps)(x,y) - xy\right|_{W^{k,\infty}([-B,B]^2;dxdy)} \leq 2C \sum_{i=1}^k \left|\Realization\left(\Phi^{2}_{\widetilde \eps}\right)-x^2 \right|_{W^{i,\infty}([-2B,2B];dx)} \left|u \right|_{W^{1,\infty}([-B,B]^2)}^i \leq 2C\widetilde \eps =\epsilon.
    \end{align*}
    and, thus, claim (i) is shown.
    Finally, we have for $k\in\{0,\dots,j\}$
    \begin{align*}
        \left|\Realization(\apmult_\eps) \right|_{W^{k,\infty}([-B,B]^2)} \leq \left|\Realization(\apmult_\eps)-xy \right|_{W^{k,\infty}([-B,B]^2;dxdy)}  + \left|xy \right|_{W^{k,\infty}([-B,B]^2;dxdy)}  \leq C_1,
    \end{align*}
    for a constant $C_1=C_1(B)>0$, yielding (ii). Claim (iii),(iv) immediately follow from the construction of $\apmult$ in combination with Proposition~\ref{prop:MonApp} and Lemma \ref{lemma:one_layer_concat}.(i).
    \end{proof}

Another statement that can be deduced from Proposition \ref{prop:MonApp} is connected to the construction of neural networks which approximate the identity on $\R^d.$
\begin{corollary}\label{cor:ApproxIdent}
    Let $\varrho:\R\to\R$ be  such that $\varrho$ is twice times continuously differentiable in a neighborhood of some $x_0\in\R$ and $\varrho'(x_0)\neq 0.$ fulfill the assumptions of Proposition \ref{prop:MonApp} for some $n=2,$ for $r=1$ and  assume that for some $k\leq n$ we have that $\varrho\in W^{k,\infty}_{\mathrm{loc}}(\R)$. Then, for every $B>0,$ $d\in\N,$ for every $L\in\N_{\geq 2}$ and for every $\eps\in(0,1)$ there exists a constant $C=C(B,\varrho)>0$ and a neural network $\Phi_\eps^{L,B,d}$ with $d$-dimensional input, $d$-dimensional output and the following properties:
    \begin{enumerate}
        \item[(i)] $\left\|\Realization(\Phi^{L,B,d}_\epsilon)-x\right\|_{W^{k,\infty}([-B,B]^d;\R^d)} \leq \epsilon$;
        \item[(ii)] $\norm{\Realization(\Phi^{L,B,d}_\epsilon)}_{W^{k,\infty}([-B,B]^d;\R^d)}\leq C\max\{1,B\}
        $;
        \item[(iii)] $L\left(\Phi^{L,B,d}_\epsilon\right)=L,$ as well as $M\left( \Phi^{L,B,d}_\epsilon\right) \leq 4dL-3d;$
        \item[(iv)]  $\bweights*{\Phi^{L,B,d}_\epsilon}\leq CL\epsilon^{-1}.$
    \end{enumerate}
\end{corollary}
\begin{proof}
W.l.o.g., we assume that $d=1.$ The other cases follow  from a minor modification of the parallelization of neural networks with the same number of layers. Let $\Phi^1_{\eps/L}$ be the neural network from Proposition \ref{prop:MonApp} for $B= B+1.$ We define $\Phi_\eps^{L,B,d} \coloneqq \Phi^1_{\eps/L}\conc \dots\conc \Phi^1_{\eps/L},$
where we perform $L-2$ concatenations. It is easy to see that $\Phi_\eps^{L,B,d}=\left((A_1,b_1),(A_2,b_2),\dots,(A_L,b_L) \right),$ where 
\begin{align*}
    A_1&= \left(0,-\frac{\eps}{LC}\right)^T \in \R^{2,1}, \\
    b_1&=  (x_0,x_0)^T\in \R^{2}, \\
    A_\ell &= \begin{pmatrix} 0 & 0\\ -\frac{1}{\varrho'(x_0)} &\frac{1}{\varrho'(x_0)} \end{pmatrix} \in \R^{2,2},\quad \text{ for } \ell=2,\dots,L-1, \\
    b_\ell &= (x_0,x_0)^T\in \R^2,\quad \text{ for } \ell=2,\dots,L-1, \\ 
    A_L &= \frac{LC}{\eps \varrho'(x_0)}\left(1,-1    \right)\in \R^{1,2}, \\
    b_L &= 0\in\R,
\end{align*}
and where $C>0$ is a suitable constant provided by Proposition \ref{prop:MonApp}. By Proposition \ref{prop:MonApp} we also have that $\Realization(\Phi^1_{\eps/L})(x)\in[-B-\eps/L,B+\eps/L]$ for all $x\in[-B,B]$ as well as
\begin{align*}
    \left\|\Realization(\Phi^1_{\eps/L})-x \right\|_{W^{k,\infty}([-B,B])}\leq \frac{\eps}{L}. 
\end{align*}
 Iterating this argument shows that $ \Realization(\Phi_\eps^{L,B})(x)\in [-B-\eps,B+\eps]$ for all $x\in[-B,B]$ and that
\begin{align*}
   \left\| \Realization(\Phi_\eps^{L,B,d})-x\right\|_{W^{k,\infty}([-B,B])} \leq \eps. 
\end{align*}
The other properties follow immediately from (i) in combination with the definition of $\Phi_\eps^{L,B,d}.$
\end{proof}

Before we continue, let us have a closer look at the properties of the concatenation of two neural networks in the following special cases. 
\begin{lemma}\label{lemma:one_layer_concat}
Let $\Phi$ be a neural network with $m$-dimensional output. 
\begin{enumerate}[(i)]
    \item\label{item:ConcVector} If $a\in\R^{1\times m}$, then,
\[
M(((a,0))\conc\Phi)\leq M(\Phi)\quad \text{and}\quad \|((a,0))\conc\Phi)\|_{\max}\leq m\|\Phi\|_{\max}\max_{i=1,\ldots,m}a_i .
\]

\item\label{item:ConcId} Let $\Phi^{L,B,m}_\eps$ be the approximate identity network from Corollary~\ref{cor:ApproxIdent}. Then, for some constant $C=C(B,\varrho)$ there holds 
\[
M(\Phi^{L,B,m}_\eps\conc\Phi)\leq M(\Phi)+M(\Phi^{L,B,m}_\eps) \quad  \text{and}\quad \|(\Phi^{L,B,m}_\eps\conc\Phi)\|_{\max}\leq C \max\{\|\Phi\|_{\max},\eps^{-1}\}.
\]
    \item\label{item:ConcMult}  Let $\widetilde{\times}$ be the approximate multiplication network from Corollary~\ref{cor:approximate_multiplication}. If $m=2,$ then, for some constant $C=C(B,\varrho)$ there holds 
\[
M(\widetilde{\times}\conc\Phi)\leq CM(\Phi)  \quad \text{and}\quad \|\widetilde{\times}\conc\Phi\|_{\max}\leq C \max\{\|\Phi\|_{\max},\eps^{-2}\} .
\]
\end{enumerate}
\end{lemma}
\begin{proof}
 For the first part of the proof of (i), see \cite{raslan2019parametric}. The second part is clear.
 
 From now on, let $\Phi=((A_1,b_1),\dots,(A_{L(\Phi)},b_{L(\Phi)})).$

 For the proof of (ii), let $\Phi^{L,B,m}_\eps = ((A_1^{\mathrm{id}},b_1^{\mathrm{id}}),\dots,(A_L^{\mathrm{id}},b_L^{\mathrm{id}}))$ and recall that 
 
 $\Phi^{L,B,m}_\eps\conc\Phi =((A_1,b_1),\dots,(A_{L(\Phi)-1},b_{L(\Phi)-1}),(A_1^{\mathrm{id}}A_{L(\Phi)}, A_1^{\mathrm{id}}b_{L(\Phi)}+b_1^{\mathrm{id}}),(A_2^{\mathrm{id}},b_2^{\mathrm{id}}),\dots,(A_L^{\mathrm{id}},b_L^{\mathrm{id}})).$ Hence, in order to proof (ii), we only need to examine
 $(A_1^{\mathrm{id}}A_{L(\Phi)}, A_1^{\mathrm{id}}b_{L(\Phi)}+b_1^{\mathrm{id}})$. From the construction of $\Phi^{L,B,m}_\eps$ we have that $\|A_1^{\mathrm{id}}\|_0 = m$ and that $A_1^{\mathrm{id}}$ has block diagonal structure. Additionally, all entries of $A_1^{\mathrm{id}}$ are bounded in absolute value by $\frac{\eps}{L\widetilde{C}}\leq 1$ for some $\widetilde{C}\geq 1.$ From this, the claim follows. 
 
 The proof of (iii) can be done in a similar manner as the proof of (ii). 
\end{proof}

\section{Proof of Proposition \ref{prop:main}}\label{app:ProofMainProp}

In this section we provide the proofs of the statements of Section \ref{sec:main} as well as additional auxiliary statements which together lead to the proof of Proposition \ref{prop:main}. Appendix \ref{app:Bumps} is concerned with the proof of Lemma \ref{prop:partition_of_unity} which establishes the conditions of the PU. Appendix \ref{app:ApproxLocal}, which contains the proof of Lemma \ref{lemma:polynomial_approximation} which shows that we are in a position to efficiently approximate $f\in\mathcal{F}_{n,d,p}$ by sums of polynomials multiplied with the functions from the PU. Appendix \ref{app:ApproxLocPol} in turn shows that these sums of localized polynomials can be approximated by neural networks. Appendix \ref{app:PutTogether} concludes the proof of Proposition \ref{prop:main}. 

\subsection{Approximate Partition of Unity}
We start with the proof of Lemma \ref{prop:partition_of_unity} which establishes the properties of the exponential (respectively polynomial, exact) $(j,\tau)$-PU.

\begin{proof}[Proof of Lemma \ref{prop:partition_of_unity}] \label{app:Bumps} 
  For the proof of the properties (\ref{item:pou_derivativeGeneralSigmoid}) and (\ref{item:pou_suppGeneralSigmoid}), we will always assume w.l.o.g. that $m=0$ unless stated otherwise. Moreover, we only give the proof for the case of an exponential PU. The other cases follow in essentially the same way with some  simplifications.
 
  \textbf{ad (\ref{item:pou_derivativeGeneralSigmoid}):} First of all, assume that $d=1.$ For \underline{$\tau=0$} and $j=0$ this follows directly from the boundedness of $\rho$. For \underline{$\tau=1$} and $j=0$, we have that $\varrho$ is Lipschitz continuous, and, thus,
  \begin{align*}
      \left|\phi_0^s(x)\right| &\leq \frac{1}{s(B-A)}\left|\varrho(3sNx+2s)-\varrho(3sNx+s) \right|+ \frac{1}{s(B-A)}\left|\varrho(3sNx-s)-\varrho(3sNx-2s) \right| \\
      &\leq 2\frac{\mathrm{Lip}(\varrho)\cdot s}{s(B-A)}= 2\frac{\mathrm{Lip}(\varrho)}{(B-A)}.
  \end{align*} 
  For \underline{$\tau\in\{0,1\}$} and $j\geq 1$ this follows from the case $j=0$ together with $\rho'\in\Wkp[j-1][\infty][\R]$ and the chain rule.
  
  Now, let $d\in\N$ be arbitrary. Since we will need it in the proof of (\ref{item:pou_suppGeneralSigmoid}), we prove the following more general statement (Statement (\ref{item:pou_derivativeGeneralSigmoid}) follows by considering $I=\{1,\dots,d\}$). Moreover, we will prove this statement only for $k\leq\min\{j,2\},$ since the rest of the proof can be done in exactly the same way by exploiting the tensor structure of $\phi_m^s.$
\vspace{.75em}
\begin{adjustwidth}{3em}{3em}
\textit{Let $I\subset \{1,\dots,d\}$ be arbitrary. Moreover, for  $m\in\{0,\dots,N\}^{|I|}$ we define $\phi_{m,I}^s:\R^{|I|}\to \R,x\mapsto  \prod_{1\leq l\leq  |I|} \psi^s\left(3N\left(x_l-\frac{m_l}{N}\right)\right)$ as well as $\phi_m^s\coloneqq \phi^s_{m,I},$ if $I=\{1,\dots,d\}$. Then for $k\in \{0,\dots,j\}$ it holds that
        \begin{equation*}
            \left|\phi_{m,I}^s \right|_{W^{k,\infty}(\R^{|I|})} \leq C^{|I|}\cdot N^k\cdot s^{\max\{0,k-\tau\}}.
        \end{equation*}}
\end{adjustwidth}
\vspace{.75em}

\noindent It is clear that by the definition of $\phi^s_{m,I}$ and what we have shown for $d=1$ that for $k=0$ there holds 
\begin{align}\label{eq:WinftyIndex}
    \left|\phi^s_{m,I} \right|_{W^{0,\infty}(\R^{|I|})} \leq C^{|I|}.
\end{align}
Now, let $i\in I$ be arbitrary. Then, by using the tensor product structure of $\phi^s_{m,I}$ in combination with what we have shown before for $d=1$, for the case $k=1$ and \eqref{eq:WinftyIndex} for $I'\coloneqq I\setminus\{i\}$ we obtain for a.e. $x\in \R^{|I|}$ 
\begin{align*}
    \left|\frac{\partial}{\partial x_i} \phi^s_{m,I}(x) \right| &= \left| \phi^s_{m,I'}(x_1,\dots,x_{i-1},x_{i+1},\dots,x_{|I|}) \right|\cdot \left|\left(\psi^s\left(3N\left(\cdot-\nicefrac{m_i}{N}\right)\right) \right)'(x_i) \right| \\
    &\leq C^{|I|-1}\cdot CN =C^{|I|}N s^{\max\{0,k-\tau\}}
\end{align*}
which implies that $|\phi^s_{m,I} |_{W^{1,\infty}(\R^{|I|})} \leq C^{|I|}N$.

Finally, let additionally be $r\in I$ be arbitrary. If $i=r$ then we have that (by using \eqref{eq:WinftyIndex} in combination with what we have shown for $d=1$) that 
\begin{align*}
    \left|\frac{\partial^2}{\partial x_i^2} \phi^s_{m,I}(x) \right| &= \left| \phi^s_{m,I'}(x_1,\dots,x_{i-1},x_{i+1},\dots,x_{|I|}) \right|\cdot \left|\left(\psi^s\left(3N\left(\cdot-\nicefrac{m_i}{N}\right)\right) \right)''(x_i) \right| \\
    &\leq C^{|I|-1}\cdot CN^2s^{\max\{0,k-\tau\}} =C^{|I|}N^2s^{\max\{0,k-\tau\}}. 
\end{align*}
Moreover, if $i\neq r,$ then, if we set $I''\coloneqq I\setminus \{i,r\}$ we obtain with similar arguments as before that 
\begin{align*}
    &\left|\frac{\partial^2}{\partial x_i\partial x_r} \phi^s_{m,I}(x) \right| \\
    &=  \left| \phi^s_{m,I''}(x_1,\dots,x_{i-1},x_{i+1},\dots,x_{r-1},x_{r+1},\dots,x_{|I|}) \right| \\ &\alplus\cdot \left|\left(\psi^s\left(3N\left(\cdot-\nicefrac{m_i}{N}\right)\right) \right)'(x_i) \right|\cdot \left|\left(\psi^s\left(3N\left(\cdot-\nicefrac{m_r}{N}\right)\right) \right)'(x_r) \right| \\ &\leq C^{|I|-2}\cdot CN\cdot CNs^{\max\{0,k-\tau\}}=C^{|I|}N^2s^{\max\{0,k-\tau\}},
\end{align*}
where we assumed w.l.o.g.\ that $i<r$. This implies $|\phi^s_{m,I} |_{W^{2,\infty}(\R^{|I|})} \leq C^{|I|}N^2s^{\max\{0,k-\tau\}}$.
  
 \textbf{ad (\ref{item:pou_suppGeneralSigmoid}):}  First of all, assume that $d=1.$ Let \underline{$\tau=0$} and let $x\leq -1/N.$ Then, since $s>R,$ we have that $3Nsx+3/2s,3Nsx-3/2\leq -R.$ We then have by the triangle inequality and the assumption on $\varrho$ that 
 \begin{align*}
     \left|\phi_0^s(x) \right| &= \left|\frac{\varrho(3Nsx+3/2s)-\varrho(3Nsx-3/2s)}{B-A} \right| \leq \left|\frac{\varrho(3Nsx+3/2s)-A}{B-A}\right| +\left|\frac{\varrho(3Nsx-3/2s)-A}{B-A}\right| \\ &\leq \frac{C'e^{D(3Nsx+3/2s)} + C'e^{D(3Nsx-3/2s)}}{B-A} \leq \frac{C'e^{D(-3s+3/2s)} + C'e^{D(-3s-3/2s)}}{B-A} \leq 2C'\frac{e^{-Ds}}{B-A}.
 \end{align*}
 Now, let $k\in\{1,\dots,j\}.$ Then, by the assumption on $\varrho,$ we have  
 \begin{align*}
     &\left|(\phi_0^s)^{(k)}(x) \right| \\ &=(3Ns)^k \left|\frac{\varrho^{(k)}(3Nsx+3/2s)-\varrho^{(k)}(3Nsx-3/2s)}{B-A} \right| 
     \leq (3Ns)^k\left|\frac{\varrho^{(k)}(3Nsx+3/2s)}{B-A}\right| +\left|\frac{\varrho^{(k)}(3Nsx-3/2s)}{B-A} \right| \\
     &\leq \frac{C'(3Ns)^k\left(e^{D(3Nsx+3/2s)}+e^{D(3Nsx-3/2s)}\right)}{B-A} 
     \leq \frac{C'(3Ns)^k\left(e^{D(-3s+3/2s)}+e^{D(-3s-3/2s)}\right)}{B-A} \\
     &\leq \frac{2C'}{B-A}(3Ns)^ke^{-Ds}.
 \end{align*}
 The case $x\geq 1/N$ can be proven in the same way.
 
 Now let \underline{$\tau=1$} and let again $x\leq -1/N$. Then $3Nsx+2s,3Nsx+s,3Nsx-2s,3Nsx-s\leq -s<-R.$ By the mean value theorem there exist $\xi_1\in(3Nsx+s,3Nsx+2s)$ and $\xi_2\in (3Nsx-2s,3Nsx-s)$ such that
  \begin{align*}
      \phi_0^s(x)= \frac{1}{s(B-A)}\left( \varrho'(\xi_1^x) - \varrho'(\xi_2^x)\right).
  \end{align*}  
  The remainder of the proof follows in exactly the same way as the proof of the analogous statement for $\tau=0$.
  The statement for $x\geq 1/N$ can be done in exactly the same manner. 
  Now, let $d\in\N$ and let $x\in \Omega^{c}_m$. Then there exists some $l\in\{1,\dots d\}$ with $|x_l- \frac{m_l}{N}|\geq 1/N.$ This  implies for $I'=\{1,\dots,d\}\setminus\{l\}$ by employing Equation~\eqref{eq:WinftyIndex} that 
    \begin{align*}
        |\phi_m^s(x)|= \left|\phi_{m,I'}^s(x_1,\dots,x_{l-1},x_{l+1},\dots, x_d)\right|\cdot \left|\psi^s\left(3N\left(x_l-\nicefrac{m_l}{N}\right)\right)\right| \leq C^{d-1}\cdot Ce^{-Ds}.
    \end{align*}    
     This shows that 
    $|\phi_m^s|_{W^{0,\infty}(\Omega^{c}_m)}\leq C^{d}e^{-Ds}.$ 
    By proceeding in a similar manner and with the same techniques as in the proof of (\ref{item:pou_derivativeGeneralSigmoid}), one can show the remaining Sobolev semi-norm estimates for the higher-order derivatives. 
    The "in-particular" part then follows from~Proposition~\ref{prop:ExpPolEst}.

  \textbf{ad (\ref{item:pou_sumGeneralSigmoid}):}  First of all, assume that $d=1$. Let \underline{$\tau=0$}. It is not hard to see that 
  \begin{align*}
      \sum_{m=0}^N \phi_m^s(x) = \frac{1}{B-A}\left(\varrho(3Nsx+3/2s)-\varrho(3Ns(x-1)-3/2s) \right).
  \end{align*}
  
  We now have for all $x\in (0,1)$ and using the properties of $\varrho$ that 
  \begin{align*}
      \left|1-\sum_{m=0}^N \phi_m^s(x)\right| &= \left|\frac{B-A-\left(\varrho(3Nsx+3/2s)-\varrho(3Ns(x-1)-3/2s)\right)}{B-A}\right|  \\
      &\leq \left|\frac{B-\varrho(3Nsx+3/2s)} {B-A}\right| +\left|\frac{A-\varrho(3Nsx-3Ns-3/2s)} {B-A}\right|  \eqqcolon \mathrm{I} + \mathrm{II}.
  \end{align*}
  We continue by estimating $\mathrm{I}$. Since $3Nsx+3/2s\geq 3/2s>3/2R,$ we obtain that 
  \begin{align*}
      \mathrm{I} \leq \frac{C'e^{-D(3Nsx+3/2s)}}{B-A} \leq \frac{C'e^{-3/2\cdot Ds}}{B-A}
  \end{align*}
  On the other hand, since $3Nsx-3Ns-3/2s\leq -3/2s\leq 0$ we obtain that 
   \begin{align*}
      \mathrm{II} \leq \frac{C'e^{D(3Nsx-3Ns-3/2s)}}{B-A} \leq \frac{C'e^{-3/2\cdot Ds}}{B-A}.
  \end{align*}
  
 For the multidimensional case we have
\begin{align*}
\pabs*{1-\sum_{m\in\MNd}\phi_m^s(x)}&=\pabs*{1-\sum_{m\in\MNd}\prod_{l=1}^d \psi^s\left(3N\left(x_l-\frac{m_l}{N}\right)\right)}\\
&=\pabs*{1-\prod_{l=1}^d \sum_{m=0}^N\psi^s\left(3N\left(x_l-\frac{m}{N}\right)\right)}\\
&=\pabs*{1-\prod_{l=1}^d \left(\underbrace{\frac{1}{B-A}\left(\varrho(3Nsx_l+3/2s)-\varrho(3Ns(x_l-1)-3/2s) \right)}_{\coloneqq \pi_l, \text{ and }\pi_0\coloneqq 1}\right)}\\
&\leq \sum_{l=1}^d\pabs{\pi_0\cdot\ldots\cdot\pi_l (1-\pi_{l+1})}\leq C\cdot e^{-3/2Ds},
\end{align*}
which follows from the one-dimensional case. 
  Now, let $k\in \{1,\dots,j\}$ and we consider only the case $d=1.$ The multi-dimensional case follows in exactly the same manner as the analogous considerations in~(\ref{item:pou_derivativeGeneralSigmoid}) and~(\ref{item:pou_suppGeneralSigmoid}). We have that 
  \begin{align*}
      \left|\left(\sum_{m=0}^N \phi_m^s\right)^{(k)}(x) \right| &\leq (3Ns)^k\cdot \frac{1}{B-A}\left(\left|\varrho^{(k)}(3Nsx+3/2s) \right| +\left|\varrho^{(k)}(3Nsx-3Ns-3/2s) \right|\right) 
  \end{align*}
  Since $x>0,$ we have that $3Nsx+3/2s\geq3/2s>R$. Since $x<1,$ $3Nsx-3Ns-3/2s\leq -3/2R<-R$. Hence, by the assumptions on $\varrho$ we obtain that 
  \begin{align*}
      \left|\left(\sum_{m=0}^N \phi_m^s\right)^{(k)}(x) \right| \leq \frac{C'(3Ns)^k}{B-A}\left( e^{-D(3Nsx+3/2s)} + e^{D(3Nsx-3Ns-3/2s)} \right) \leq \frac{2C'(3Ns)^k}{B-A}e^{-3/2Ds}.
  \end{align*}
  The multidimensional case for $k\in\{0,\ldots,j\}$ follows in a similar manner as above from the tensor structure.
  Now, let \underline{$\tau=1$}. It is not hard to see that for all $x\in\R$ there holds 
  \begin{align*}
      \sum_{m=0}^N \phi_m^s(x) = \frac{1}{s(B-A)} \left(\varrho(3Nsx+2s)-\varrho(3Nsx+s)-\varrho(3Nsx-3Ns-s)+\varrho(3Nsx-3Ns-2s) \right).
  \end{align*}
  Now, let $x\in(0,1).$ We have that $3Nsx+2s,3Nsx+s\geq s>R$ and $3Nsx-3Ns-s,3Nsx-3Ns-2s\leq -s<-R.$ Hence, by the mean value theorem, for every $x\in \R$ there exist $\xi_1\in (3Nxs+s,3Nxs+2s)$ and $\xi_2\in (3Nxs-3Ns-2s,3Nxs-3Ns-s)$ such that
  \begin{align*}
      \sum_{m=0}^N \phi_m^s(x)&= \frac{1}{B-A}\left(\varrho'(\xi_1) - \varrho'(\xi_2)\right).
  \end{align*}
  Now we have that 
  \begin{align*}
      \left|1-  \sum_{m=0}^N \phi_m^s(x)\right| \leq \left|\frac{B-\varrho'(\xi_1)}{B-A} \right| + \left|\frac{A-\varrho'(\xi_2)}{B-A} \right|.
  \end{align*}
 The remainder of the statement can be proven in exactly the same way as the analogous statement for $\tau=0.$ 
   \textbf{ad (\ref{item:pou_networkGeneralSigmoid}):} This immediately follows from the definition of the functions $\phi_m^s$. 
   \end{proof}

 \subsection{Approximation by Localized Polynomials} \label{app:ApproxLocal}
In this section, we demonstrate how to approximate a function $f\in \mathcal{F}_{n,d,p}$ by localized polynomials based on the exponential (respectively polynomial, exact) $(j,\tau)$-PU. We only give the proof for the case of an exponential PU. The other cases follow in essentially the same way with some simplifications.
\begin{lemma}\label{lemma:polynomial_approximation}
We make the following assumption:
\begin{itemize}
    \item Let $d\in \N$, $j,\tau \in\N_0,~k\in\{0,\dots,j\}$, $n\in\N_{\geq k+1}$ and $1\leq p\leq \infty$.
    \item Assume that $(\Psi^{(j,\tau,N,s)})_{N\in\N,s\geq 1}$ is an exponential (respectively polynomial, exact) $(j,\tau)$-PU from Definition~\ref{def:partition_of_unity}. Let $\mu\in(0,1)$. For $N\in\N,$ set 
    \begin{align*}
        s\coloneqq \begin{cases} N^\mu,&\text{if exponential PU,}\\
        N^{\frac{2d/p+d+n}{D}},&\text{if polynomial PU,}\\
		1, &\text{if exact PU},
		\end{cases}
    \end{align*}
    
\end{itemize}
	  Then there is a constant $C=C(d,n,p,k)>0$ and $\widetilde N = \widetilde N(d,p,\mu,k,\tau)\in\N$ such that for every $f\in\Wkp[n][p][\cube^d]$ and every $m\in \{0,\dots,N\}^d$, there exist polynomials $p_{f,m}(x)=\sum_{|\alpha|\leq n-1}c_{f,m,\alpha}x^\alpha$ for $m\in\MNd$ with the following properties:
	
Set $f_N\coloneqq\sum_{m\in\MNd}\phi_m^s p_{f,m}$.
Then, the operator $T_k:\Wkp[n][p][\cube^d]\to\Wkp[k][p][\cube^d]$ with $T_kf=f-f_N$ is linear and bounded with
		\[\displaystyle
			\norm{T_k f}_{\Wkp[k][p][\cube^d]}\leq C\norm{f}_{\Wkp[n][p][\cube^d]}\cdot  \begin{cases}\left(\frac{1}{N}\right)^{n-k-\drate}, \quad & \text{ if exponential PU,} \\
			\left(\frac{1}{N}\right)^{n-k},\quad & \text{ for } k\leq \tau, \text{ if polynomial PU,} \\
			\left(\frac{1}{N}\right)^{n-k}, \quad & \text{ if exact PU,}
			\end{cases} 
			\]
			for all $N\in\N$ with $N\geq \widetilde N$.
\end{lemma}

 Before the proof of this statement, we need some preparation. We start with the following observation.

\begin{remark}\label{remark:polynomial_coefficients}
Since the polynomials utilized in~Lemma~\ref{lemma:polynomial_approximation} are the averaged Taylor polynomials from the Bramble-Hilbert Lemma~\ref{lemma:bramble_hilbert}, we get that there is a constant $C=C(d,n,k)>0$ such that for any $f\in\Wkp[n][p][\cube^d]$ the coefficients of the polynomials $p_{f,m}$ satisfy 
	\[
		\pabs{c_{f,m,\alpha}}\leq C\norm{\widetilde f}_{\Wkp[n][p][\Omega_{m,N}]}N^{d/p},
		\]
		for all $\alpha\in\N^d_0$ with $\pabs{\alpha}\leq n-1,$ and for all $m\in\{0,\ldots,N\}^d$, where $\Omega_{m,N}\coloneqq B_{\frac{1}{N},\norm{\cdot}_{\infty}}\left(\frac{m}{N}\right)$ and $\widetilde f\in\Wkp[n][p][\R^d]$ is an extension of $f$.
\end{remark}

We now state and prove an auxiliary result. The estimation will be very rough and can for sure be improved. This is, however, not necessary for our purpose.
\begin{lemma}\label{lemma:f_p_simple}
    Under the conditions of Lemma~\ref{lemma:polynomial_approximation} and with the notation from Remark~\ref{remark:polynomial_coefficients} we have for all $m,\wtilde m\in\MNd$ the estimate
    \[
    \norm{\widetilde f-p_{f,m}}_{\Wkp[k][p][\Omega_{\wtilde m,N}]}\leq C N^{d/p}\norm{f}_{\Wkp[n][p][\cube^d]},
    \]
    for a constant $C=C(n,d,p,k)$.
\end{lemma}
\begin{proof}
We start with bounding the norm of the polynomial by using the triangle inequality. There holds 
\begin{align*}
    \norm{p_{f,m}}_{\Wkp[k][p][\Omega_{\wtilde m,N}]}= \norm*{\sum_{|\alpha|\leq n-1}c_{f,m,\alpha}x^\alpha}_{\Wkpd[k][p][\Omega_{\wtilde m,N}]{x}}\leq\sum_{|\alpha|\leq n-1}\pabs{c_{f,m,\alpha}}\cdot \norm{x^\alpha}_{\Wkpd[k][p][\Omega_{\wtilde m,N}]{x}}.
\end{align*}
Using that $\Omega_{\wtilde m,N}\subset B_{2,\norm{\cdot}_{\infty}}$ we get
\begin{equation}\label{eq:poly_bound}
    \norm{x^\alpha}_{\Wkpd[k][p][\Omega_{\wtilde m,N}]{x}}\leq (n-1)^k 2^{\pabs{\alpha}}\leq (n-1)^k 2^{n-1}.
\end{equation}
If we now combine Remark~\ref{remark:polynomial_coefficients} with Equation~\eqref{eq:poly_bound}, we get
\begin{align*}
    \sum_{|\alpha|\leq n-1}\pabs{c_{f,m,\alpha}}\norm{x^\alpha}_{\Wkpd[k][p][\Omega_{\wtilde m,N}]{x}}\leq C (n-1)^k 2^{n-1} \sum_{|\alpha|\leq n-1}  N^{d/p} \norm{\widetilde f}_{\Wkp[n][p][\Omega_{m,N}]}\leq C N^{d/p} \norm{f}_{\Wkp[n][p][\cube^d]},
\end{align*}
where we have additionally used Remark~\ref{remark:extension_operator} in the last step.
Finally, we can estimate, by the triangle inequality
\begin{align*}
    \norm{\widetilde f-p_{f,m}}_{\Wkp[k][p][\Omega_{\wtilde m,N}]}\leq C \norm{f}_{\Wkp[k][p][\cube^d]}+CN^{d/p}\norm{f}_{\Wkp[n][p][\cube^d]}\leq CN^{d/p}\norm{f}_{\Wkp[n][p][\cube^d]},
\end{align*}
where we again used the extension property from Equation~\eqref{eq:extension_bound} for the first step.
\end{proof}

Now we are in a position to prove Lemma \ref{lemma:polynomial_approximation}. 
\begin{proof}[Proof of Lemma~\ref{lemma:polynomial_approximation}]
We use approximation properties of the polynomials from the Bramble-Hilbert Lemma~\ref{lemma:bramble_hilbert} to derive local estimates and then combine them using an exponential PU to obtain a global estimate. In order to use this strategy also near the boundary, we make use of an extension operator (see Remark~\ref{remark:extension_operator}).

	\textbf{Step 1 (Local estimates based on Bramble-Hilbert)}: For each $m\in \MNd$ we set 
	\[
		\Omega_{m,N}\coloneqq B_{\frac{1}{N},\norm{\cdot}_{\infty}}\Big(\frac{m}{N}\Big)
	\]and denote by $p_m=p_{f,m}$ the polynomial from Lemma~\ref{lemma:bramble_hilbert} 
	 so that we can directly state the estimate
	\begin{equation}\label{eq:local_Winf_bound}
		\norm[\big]{\widetilde f-p_m}_{\Wkp[k][p][\Omega_{ m,N}]}\leq C\left(\frac{1}{N}\right)^{n-k}\norm{\widetilde f}_{\Wkp[n][p][\Omega_{m,N}]}.
	\end{equation}
 Furthermore, similarly to~\cite[Lemma~C.4]{guhring2019error}, we  obtain the estimate 
	\begin{align*}
		\norm[\big]{\phi_m^s(\widetilde f -p_m)}_{\Wkp[k][p][\Omega_{m,N}]}&\leq C\sum_{\kappa=0}^k \norm{\phi^s_m}_{\Wkp[\kappa][\infty][\Omega_{m,N}]}\norm{\widetilde f -p_m}_{\Wkp[k-\kappa][p][\Omega_{m,N}]}\\
		&\leq C\sum_{\kappa=0}^k N^{\kappa+\dratej{\kappa}}\left(\frac{1}{N}\right)^{n-k+\kappa}\norm{\widetilde f}_{\Wkp[n][p][\Omega_{m,N}]}\\
		&\leq C \left(\frac{1}{N}\right)^{n-k-\drate}\norm{\widetilde f}_{\Wkp[n][p][\Omega_{m,N}]},
	\end{align*}
	where we used the product rule from Lemma~\ref{lemma:product_rule_bound} for the first step and the estimate of the derivative of $\phi_m^s$ from  Lemma~\ref{prop:partition_of_unity}~(\ref{item:pou_derivativeGeneralSigmoid}) together with the Bramble-Hilbert estimate in Equation~\eqref{eq:local_Winf_bound} for the second step.
	
\textbf{Step 2 (Local estimates based on exponential decay)}:
Since our localizing bump functions $\phi_{m}^s$ do not necessarily have compact support on $\Omega_{m,N}$ we also need to bound the influence of $\phi_{ m}^s(\widetilde f -p_{ m})$ on patches $\Omega_{\wtilde m,N}$ with $\wtilde m \neq m$ where we can not use the Bramble-Hilbert lemma. Here, we will make use of the exponential decay of the bump functions $\phi^s_{m}$ outside a certain ball centered at $m/N$ (see Lemma~\ref{prop:partition_of_unity}~(\ref{item:pou_suppGeneralSigmoid})).

This is possible for the case where $\Omega_{\wtilde m,N}$ is not a neighboring patch of $\Omega_{m, N}$, i.e.\ $\norm{\wtilde m-m}_{\infty}>1$. Then $\Omega_{\wtilde m,N}\subset \Omega^{c}_{m} $ and we have (by using Lemma \ref{lemma:product_rule_bound} in the first step), that
	\begin{align*}
		\norm[\big]{\phi_{ m}^s(\widetilde f -p_{ m})}_{\Wkp[k][p][\Omega_{\wtilde m,N}]}&\leq C\norm{\phi^s_{ m}}_{\Wkp[k][\infty][\Omega_{\wtilde m,N}]}\norm{\widetilde f -p_{m}}_{\Wkp[k][p][\Omega_{\wtilde m,N}]}\\
		\expl{Lemma~\ref{prop:partition_of_unity}~(\ref{item:pou_suppGeneralSigmoid})) with $\Omega_{\wtilde m,N}\subset \Omega^{c}_{m}$}&\leq CN^{k+\drate} e^{-DN^\mu} \norm{\widetilde f -p_{ m}}_{\Wkp[k][p][\Omega_{\wtilde m,N}]}\\
		\expl{Lemma~\ref{lemma:f_p_simple}}&\leq C\underbrace{N^{k+\drate} N^{d/p}}_{\coloneqq \gamma(N)}e^{-DN^\mu} \norm{f}_{\Wkp[n][p][\cube^d]}.
	\end{align*}
	Then, by Proposition \ref{prop:ExpPolEst}, there exists $N_1 = N_1(\mu,d,p)\in\N$ such that $e^{-DN^\mu}\leq C\gamma(N)^{-1}\cdot (N+1)^{-d-d/p}\cdot N^{-(n-k-\drate)}$ for all $N\geq N_1$. Consequently, we have 
	\begin{equation*}
		\norm[\big]{\phi_{m}^s(\widetilde f -p_{m})}_{\Wkp[k][p][\Omega_{\wtilde m,N}]}\leq C (N+1)^{-d-d/p}N^{-(n-k-\drate)}\norm{f}_{\Wkp[n][p][\cube^d]},
	\end{equation*}
	for all $N\geq N_1$.
	
\textbf{Step 3 (Mixed local estimates)}:
If $\Omega_{\wtilde m,N}$ is a neighboring patch of $\Omega_{m, N}$, i.e.\ $\norm{\wtilde m-m}_{\infty}=1$, then we have to split the patch in a region $\Omega_{\wtilde m,N}\cap \Omega^{c}_{m}$ where we have exponential decay of the bump function and a region $\Omega_{\wtilde m,N}\setminus \Omega^{c}_{m}\subset \Omega_{m,N}$ where we can make use of the Bramble-Hilbert Lemma. In detail, we have
	\begin{align*}
		\norm[\big]{\phi_{m}^s(\widetilde f -p_{m})}_{\Wkp[k][p][\Omega_{\wtilde m,N}]}
		&\leq \norm[\big]{\phi_{m}^s(\widetilde f -p_{m})}_{\Wkp[k][p][\Omega_{\wtilde m,N}\setminus \Omega^{c}_{m}]} + \norm[\big]{\phi_{m}^s(\widetilde f -p_{m})}_{\Wkp[k][p][\Omega_{\wtilde m,N}\cap \Omega^{c}_{m}]}\\
		&\leq CN^{-(n-k-\drate)}\left(  \norm{\widetilde f}_{\Wkp[n][p][\Omega_{m,N}]} + (N+1)^{-d-d/p}\norm{f}_{\Wkp[n][p][\cube^d]}\right),
		\end{align*}
		for all $N\geq N_1$. Here we used Step 1 to bound the first term of the sum and Step 2 for the second.

\textbf{Step 4 (Global estimate)}: 
		Using that $\widetilde f $ is an extension of $f$ on $\cube^d$ we can write
	\begin{align*}
		&\norm*{f -\sum_{m\in\MNd} \phi_m^s p_m}_{\Wkp[k][p][\cube^d]}\\
		&\leq \norm*{\widetilde f -\sum_{m\in\MNd} \phi_m^s \widetilde f}_{\Wkp[k][p][\cube^d]} +\norm*{\sum_{m\in\MNd} \phi_m^s(\widetilde f-p_m)}_{\Wkp[k][p][\cube^d]}\\
		&\leq \underbrace{\norm*{\widetilde f\Big(\mathbbm{1}_{\cube^d} -\sum_{m\in\MNd} \phi_m^s\Big) }_{\Wkp[k][p][\cube^d]}}_{\textbf{Step 4a}}+ \left(\underbrace{\sum_{\wtilde m\in\MNd} \norm*{\sum_{m\in\MNd} \phi_m^s(\widetilde f-p_m)}_{\Wkp[k][p][\Omega_{\wtilde m,N}]}^p}_{\textbf{Step 4b}}\right)^{1/p},\numberthis\label{eq:global_max}
	\end{align*}
	where the last step follows from $\cube^d\subset\bigcup_{\wtilde m\in\MNd}\Omega_{\wtilde m,N}$. 
	
\textbf{Step 4a (Partition of Unity)}: 
	For the first term in Equation~\eqref{eq:global_max}, we get by the product rule from Lemma~\ref{lemma:product_rule_bound}
	\begin{align*}
	    \norm*{\widetilde f\Big(\mathbbm{1}_{\cube^d} -\sum_{m\in\MNd} \phi_m^s\Big) }_{\Wkp[k][p][\cube^d]}&C\leq \norm{f}_{\Wkp[k][p][\cube^d]}\norm*{\mathbbm{1}_{\cube^d} -\sum_{m\in\MNd} \phi_m^s }_{\Wkp[k][\infty][\cube^d]}\\
	    \expl{Property (\ref{item:pou_sumGeneralSigmoid}) from Lemma~\ref{prop:partition_of_unity}}&\leq C \norm{f}_{\Wkp[k][p][\cube^d]}\cdot N^{-(n-k-\drate)},\numberthis\label{eq:distance_to_one}
	\end{align*}
	for all $N\geq N_2=N_2(\mu,k,\tau)$. For the second inequality we used the same trick as in Step~2 which is based on Proposition~\ref{prop:ExpPolEst}.
	
\textbf{Step 4b (Patches)}: Considering the second term from Equation~\eqref{eq:global_max}, we obtain for each $\wtilde m\in\MNd$ 
	\begin{align*}
		&\norm*{\sum_{m\in\MNd} \phi_m^s(\widetilde f-p_m)}_{\Wkp[k][p][\Omega_{\wtilde m,N}]}\\
		&\leq {\underbrace{\vphantom{\sum_{\substack{m\in\MNd,\vspace{0.2em}\vspace{0.2em}\\ \norm{m-\wtilde m}_{\infty} = 1}}}\norm{\phi_{\wtilde m}^s(\widetilde f-p_{\wtilde m})}_{\Wkp[k][p][\Omega_{\wtilde m,N}]}}_{(\star)}}+ \underbrace{\sum_{\substack{m\in\MNd,\vspace{0.2em}\vspace{0.2em}\\ \norm{m-\wtilde m}_{\infty} = 1}} \norm{\phi_m^s(\widetilde f-p_m)}_{\Wkp[k][p][\Omega_{\wtilde m,N}]}}_{(\star \star)}+ \underbrace{\sum_{\substack{m\in\MNd,\vspace{0.2em}\vspace{0.2em}\\ \norm{m-\wtilde m}_{\infty} > 1}} \norm{\phi_m^s(\widetilde f-p_m)}_{\Wkp[k][p][\Omega_{\wtilde m,N}]}}_{(\star \star \star)}.\numberthis\label{eq:war_of_patches}
	\end{align*}
	The term $(\star)$ can be handled with Step~1, the term $(\star \star)$ with Step~3 and the third one $(\star \star \star)$ with Step~2. Since $(\star \star)$ and $(\star \star \star)$ require a similar strategy we only demonstrate it for the third term. We get from Step~2
\begin{align*}
    \sum_{\substack{m\in\MNd,\vspace{0.2em}\vspace{0.2em}\\ \norm{m-\wtilde m}_{\infty} > 1}}\norm{\phi_m^s(\widetilde f-p_m)}_{\Wkp[k][p][\Omega_{\wtilde m,N}]}
    &\leq CN^{-(n-k-\drate)}(N+1)^{-d-d/p}\sum_{\substack{m\in\MNd,\vspace{0.2em}\vspace{0.2em}\\ \norm{m-\wtilde m}_{\infty} > 1}}\norm{f}_{\Wkp[n][p][\cube^d]}\\
&\leq CN^{-(n-k-\drate)}(N+1)^{-d/p}\norm{f}_{\Wkp[n][p][\cube^d]}.
\end{align*}

We can now bound the sum from Equation~\eqref{eq:war_of_patches} for each $\wtilde m\in\{0,\ldots,N\}^d$ by
	\begin{align*}
		&\norm*{\sum_{m\in\MNd} \phi_m^s(\widetilde f-p_m)}_{\Wkp[k][p][\Omega_{\wtilde m,N}]}\\
		&\leq CN^{-(n-k-\drate)}\left( 2(N+1)^{-d/p}\norm{f}_{\Wkp[n][p][\cube^d]}+\sum_{\substack{m\in\MNd,\vspace{0.2em}\vspace{0.2em}\\ \norm{m-\wtilde m}_{\infty} \leq 1}}\norm{\widetilde f}_{\Wkp[n][p][\Omega_{m,N}]}\right)\numberthis\label{eq:tilde_patch}.
	\end{align*}
Consequently, we get
	\begin{align*}
		&\sum_{\wtilde m\in\MNd} \norm*{\sum_{m\in\MNd} \phi_m^s(\widetilde f-p_m)}_{\Wkp[k][p][\Omega_{\wtilde m,N}]}^p\\
		&\leq CN^{-(n-k-\drate)p}\sum_{\wtilde m\in\MNd} \left( 2(N+1)^{-d/p}\norm{f}_{\Wkp[n][p][\cube^d]}+\sum_{\substack{m\in\MNd,\vspace{0.2em}\vspace{0.2em}\\ \norm{m-\wtilde m}_{\infty} \leq 1}}\norm{\widetilde f}_{\Wkp[n][p][\Omega_{m,N}]}\right)^p\\
		&\leq CN^{-(n-k-\drate)p}(3^d+1)^{p/q} \\ &\qquad\qquad \cdot \left(\sum_{\wtilde m\in\MNd}2^p(N+1)^{-d}\norm{f}_{\Wkp[n][p][\cube^d]}^p+\sum_{\wtilde m\in\MNd}\sum_{\substack{m\in\MNd,\vspace{0.2em}\\\norm[\tiny]{m-\wtilde m}_{\infty} \leq 1}} \norm{\widetilde f}_{\Wkp[n][p][\Omega_{m,N}]}^p\right) \\
		&\leq C N^{-(n-k-\drate)p}\left(\norm{f}^p_{\Wkp[n][p][\cube^d]}+ 3^d \sum_{\wtilde m\in\MNd} \norm{\widetilde f}_{\Wkp[n][p][\Omega_{\wtilde m,N}]}^p\right), \numberthis\label{eq:tilde_patch_2}
	\end{align*}
	where the first step follows from plugging in Equation~\eqref{eq:tilde_patch}, the second step follows from H{\"o}lder's inequality (with $q\coloneqq1-1/p$) and the last step follows from the definition of $\Omega_{\wtilde m,N}$. Moreover, we use in the second and the last step the fact that the number of neighbors of a particular patch is bounded by $3^{d}-1$. 
	To conclude Step 4b we note that from the definition of $\Omega_{\wtilde m, N}$ it follows that there exist $2^d$ disjoint subsets $\mathcal{M}_i\subset\{0,\ldots,N\}^d$ such that $\bigcup_{i=1,\ldots,2^d}\mathcal{M}_i = \{0,\ldots,N\}^d$ and $\Omega_{m_1, N}\cap\Omega_{m_2, N}=\emptyset$ for all $m_1,m_2\in\mathcal{M}_i$ with $m_1\neq m_2$ and all $i=1,\ldots,2^d$. From this we get
	\begin{equation}\label{eq:2dtrick}
	    \sum_{\wtilde m\in\MNd} \norm{\widetilde f}_{\Wkp[n][p][\Omega_{\wtilde m,N}]}^p=\sum_{i=1,\ldots,2^d}\sum_{\wtilde m\in\mathcal{M}_i} \norm{\widetilde f}_{\Wkp[n][p][\Omega_{\wtilde m,N}]}^p\leq 2^d \norm{\widetilde f}_{\Wkp[n][p][\bigcup_{\wtilde m\in\MNd}\Omega_{\wtilde m,N}]}^p
	\end{equation}
	and, finally, together with Remark~\ref{remark:extension_operator}
	\begin{equation}\label{eq:sum_up_the_sum}
	    \sum_{\wtilde m\in\MNd} \norm{\widetilde f}_{\Wkp[n][p][\Omega_{\wtilde m,N}]}^p\leq 2^d \norm{\widetilde f}_{\Wkp[n][p][\bigcup_{\wtilde m\in\MNd}\Omega_{\wtilde m,N}]}^p\leq C \norm{f}^p_{\Wkp[n][p][\cube^d]}.
	\end{equation}
	\textbf{Step 4c (Wrap it all up)}: Combining Equation~\eqref{eq:tilde_patch_2} with~Equation~\eqref{eq:sum_up_the_sum} from Step 4b and inserting it into Equation~\eqref{eq:global_max} together with the estimate in Equation~\eqref{eq:distance_to_one} from Step 4a finally yields
	\[
		\norm{f-f_N}_{\Wkp[k][p][\cube^d]}\leq C N^{-(n-k-\drate)} \norm{f}_{\Wkp[n][p][\cube^d]},
		\]
	for all $N\geq \widetilde N\coloneqq \max\{N_1, N_2\}$ and a constant $C=C(n,d,p,k)>0$. The linearity of $T_k$, $k\in\{0,\dots,j\}$ is a consequence of the linearity of the averaged Taylor polynomial (cf.~\cite[Remark~B.8]{guhring2019error}).
\end{proof}

\subsection{Approximation of Localized Polynomials by Neural Networks}\label{app:ApproxLocPol}

The goal of this subsection is to demonstrate how to approximate sums of  localized polynomials $\sum_{\mathrm{p}} \phi_{\mathrm{p}}\mathrm{poly}_{\mathrm{p}}$ by neural networks.
Corollary \ref{cor:approximate_multiplication} is the foundation for the following results which implements a neural network that approximates the multiplication of multiple inputs:

\begin{lemma}\label{lemma:network_multiplikation}
Let $d,m,K\in\N,j\in\N_0$ and $N\geq 1$, $\mu\geq 0, c>0$ be arbitrary, and let $\varrho\in W^{j,\infty}_{\mathrm{loc}}(\R)$ fulfill the assumptions of Proposition \ref{prop:MonApp} for $n=3,~r=2$. Then there are constants $C(d,m,c,k)>0$ such that the following holds:

	For any $\eps\in \epsin$, and any neural network $\Phi$ with $d$-dimensional input and $m$-dimensional output and with number of layers and nonzero weights all bounded by $K$, such that 
	\begin{equation*}
		\norm{[\act{\Phi}]_l}_{\Wkp[k][\infty][\cube^d]}\leq c N^{k+\drate},
		\end{equation*}
		for $k\in\{0,\dots,j\}$, $l=1,\ldots,m$ and $x\in\cube^d$ there exists a neural network $\Psi_{\eps,\Phi}$ with $d$-dimensional input and one-dimensional output, and with 
		\begin{enumerate}[(i)]
		    \item number of layers and nonzero weights all bounded by $CK $;
		    \item $\norm*{\act{\Psi_{\eps,\Phi}}-\prod_{l=1}^n [\act{\Phi}]_l}_{\Wkp[k][\infty][\cube^d]} \leq C  N^{k+\drate} \eps$;
		\item \label{item:bounding_range_nnmult}$\pabs{\act{\Psi_{\eps,\Phi}}}_{\Wkp[k][\infty][\cube^d]}\leq CN^{k+\drate}$;
		\item $\bweights{\Psi_{\eps,\Phi}}\leq C\max\{\bweights{\Phi},\eps^{-2}\}$.
		\end{enumerate}
\end{lemma}

\begin{proof}
	We show by induction over $m\in\N$ that the statement holds. To make the induction argument easier we will additionally show that the network $\Psi_{\eps,\Phi}$ can be chosen such that the first $L(\Phi)-1$ layers of $\Psi_{\eps,\Phi}$ and $\Phi$ coincide.

	If $m=1$, then we can choose $\Psi_{\eps,\Phi}=\Phi$ for any $\eps\in\epsin$ and the claim holds. 
	
	Now, assume that the claim holds for an arbitrary, but fixed $m\in\N$. We show that it also holds for $m+1$. For this, let $\eps \in\epsin$ and let $\Phi=((A_1,b_1),(A_2,b_2),\dots,(A_L,b_L))$ be a neural network with $d$-dimensional input and $(m+1)$-dimensional output and with number of layers, and nonzero weights all bounded by $K$, where each $A_l$ is an $N_l\times N_{l-1}$ matrix, and $b_l\in\R^{N_l}$ for $l=1,\ldots L$.
	
	\textbf{Step 1 (Invoking induction hypothesis)}: We denote by $\Phi_m$ the neural network with $d$-dimensional input and $m$-dimensional output which results from $\Phi$ by removing the last output neuron and corresponding weights. In detail, we write 
	\[
		A_L=\bmat{c}{
			A_L^{(1,m)}\\[1em]
			a_L^{(m+1)}
		}\quad\text{and}\quad
		b_L=
		\bmat{c}{
			b_L^{(1,m)}\\[1em]
			b_L^{(m+1)}
		},
		\] where $A_L^{(1,m)}$ is a $m\times N_{L-1}$ matrix and $a_L^{(m+1)}$ is a $1\times N_{L-1}$ vector, and $b_L^{(1,m)}\in\R^m$ and $b_L^{(m+1)}\in\R^1$. Now we set 
		\[
			\Phi_m\coloneqq\Big((A_1,b_1),(A_2,b_2),\dots,(A_{L-1},b_{L-1}),\Big(A_L^{(1,m)},b_L^{(1,m)}\Big)\Big).
			\]
		Using the induction hypothesis we get that there is a neural network $$\Psi_{\eps,\Phi_m}=((A'_1,b'_1),(A'_2,b'_2),\dots,(A'_{L'},b'_{L'}))$$ with $d$-dimensional input and one-dimensional output, and at most $KC$ layers and nonzero weights such that 
	\begin{equation*}
		\norm*{\act{\Psi_{\eps,\Phi_m}}-\prod_{l=1}^m [\act{\Phi_m}]_l}_{\Wkp[k][\infty][\cube^d]} \leq C N^{k+\drate} \eps,
	\end{equation*}
	and $\pabs{\act{\Psi_{\eps,\Phi_m}}}_{\Wkp[k][\infty][\cube^d]}\leq C N^{k+\drate}$. Moreover, we have that $\bweights{\Phi_m}\leq \bweights{\Phi}$, so that there we can estimate $\bweights{\Psi_{\eps,\Phi_m}}\leq C\max\{\bweights{\Phi},\eps^{-2}\}$.
	Furthermore, we can assume that the first $L-1$ layers of $\Psi_{\eps,\Phi_m}$ and $\Phi_m$ coincide and, thus, also the first $L-1$ layers of $\Psi_{\eps,\Phi_m}$ and $\Phi$, i.e.\ $A_l=A'_l$ for $l=1,\ldots,L-1$. 
	
	\textbf{Step 2 (Combining $\Psi_{\eps,\Phi_m}$ and $\big[\actbig{\Phi}\big]_{m+1}$)}: Now, we construct a network 
	$\widetilde \Psi_{\eps,\Phi}$ where the first $L-1$ layers of $\widetilde \Psi_{\eps,\Phi}$ and $\Psi_{\eps,\Phi_m}$ (and, thus, also of $\Phi$) coincide (by definition), and $\widetilde \Psi_{\eps,\Phi}$ has two-dimensional output with  $\big[\actbig{\wtilde\Psi_{\eps,\Phi}}\big]_1=\actbig{\Psi_{\eps,\Phi_m}}$ and $\big[\actbig{\wtilde\Psi_{\eps,\Phi}}\big]_2\approx\big[\actbig{\Phi}\big]_{m+1}$.
	For this, we add the formerly removed neuron with corresponding weights back to the $L$-th layer of $\Psi_{\eps,\Phi_m}$ and approximately pass the output through to the last layer. Let $\Phi^{L'-L+1,c,1}_{\eps}=((A^{\mathrm{id}}_1,b^{\mathrm{id}}_1),\ldots,(A^{\mathrm{id}}_{L'-L+1},b^{\mathrm{id}}_{L'-L+1}))$ be the network from~Corollary~\ref{cor:ApproxIdent}.
	We define 
	\begin{align*}
	&\widetilde \Psi_{\eps,\Phi}\coloneqq\\
	&\left((A'_i,b'_i)_{i=1}^{L-1},
	\left(
		\bmat{c}{
			A'_{L}\\[1em]
			A_1^{\mathrm{id}}a^{(m+1)}_L 
		},
	\bmat{c}{
		 b'_{L}\\[1em]
		 A_1^{\mathrm{id}}b_L^{(m+1)} + b_1^{(m+1)}
	} \right),
	\left(
		\bmat{c}{
			A'_{L+1}\\[1em]
			A_2^{\mathrm{id}} 
		},
	\bmat{c}{
		 b'_{L+1}\\[1em]
		 b_2^{\mathrm{id}}
	}\right),\ldots \right.\\
&\qquad \qquad\qquad \qquad	\qquad \qquad\qquad \qquad\qquad\qquad\qquad\qquad\ldots\left.
	\left(
		\bmat{c}{
			A'_{L'}\\[1em]
			A_{L'-L+1}^{\mathrm{id}}
		},
	\bmat{c}{
		 b'_{L'}\\[1em]
		 b_{L'-L+1}^{\mathrm{id}}
	}\right)\right).
	\end{align*}
	Counting the number of nonzero weights of $\wtilde\Psi_{\eps,\Phi}$ we get with Lemma~\ref{lemma:one_layer_concat}~(\ref{item:ConcId}) that
	\begin{align*}
	    M(\wtilde\Psi_{\eps,\Phi})\leq M(\Psi_{\eps,\Phi_m})+\underbrace{M(\Phi)}_{\text{from }a_L^{(m+1)},b_L^{(m+1)}}+\underbrace{4(L'-L+1)}_{\vphantom{a_L^{(m+1)}}\text{from approximative identity} }\leq C K+K+CK\leq CK\numberthis\label{eq:number_weights},
	\end{align*}
	where we used in the second step the induction hypothesis twice together with the assumption on $\Phi$. Similarly, we get the statement for $L(\wtilde\Psi_{\eps,\Phi})$. Furthermore, $\bweights{\wtilde\Psi_{\eps,\Phi}}\leq C\max\{\bweights{\Phi},\eps^{-2}\}$.
	
	Next, we want to apply the approximate multiplication network from Corollary~\ref{cor:approximate_multiplication} to the output of $\wtilde\Psi_{\eps,\Phi}$. For this, we need to find a bounding box for the range of $\actbig{\wtilde\Psi_{\eps,\Phi}}$. We have 
		\begin{equation*}
			\norm{\act{\Psi_{\eps,\Phi_m}}}_{\Linfc} \leq C\quad\text{and}\quad\norm{[\act{\widetilde{\Psi}_{\eps,\Phi}}]_{2}}_{\Linfc}\leq c+\epsilon\leq c+1,
		\end{equation*}
		and get for $B\coloneqq \max\{C,c+1\}$ that $\ran\act{\wtilde\Psi_{\eps,\Phi}}\subset [-B,B]^2$.
			Now, we denote by $\apmult$ the network from Corollary~\ref{cor:approximate_multiplication} with $B=B$ and accuracy $\eps$ and define 
		\[
			\Psi_{\eps,\Phi}\coloneqq\apmult\conc\wtilde \Psi_{\eps,\Phi}.
			\] 
			
	\textbf{Step 3 ($\Psi_{\eps,\Phi}$ fulfills induction hypothesis for $m+1$)}:
			\textbf{ad (i):} Clearly, $\Psi_{\eps,\Phi}$ has $d$-dimensional input, one-dimensional output and, combining Equation~\eqref{eq:number_weights} with (\ref{item:network_complexity_apmult}) of Corollary~\ref{cor:approximate_multiplication} as well as Lemma \ref{lemma:one_layer_concat}~(iii), at most $ CK$
            nonzero weights. 
	
	\textbf{ad (ii):} The first $L-1$ layers of $\Psi_{\eps,\Phi}$ and $\Phi$ coincide and for the approximation properties it holds that 
	\begin{align*}
		&\norm*{\act{\Psi_{\eps,\Phi}}-\prod_{l=1}^{m+1} [\act{\Phi}]_l}_{\Wkp[k][\infty][\cube^d]}\\
		&\alplus = \norm*{\act{\apmult}\circ\actbig{\wtilde\Psi_{\eps,\Phi}}-[\act{\Phi}]_{m+1}\cdot\prod_{l=1}^{m} [\act{\Phi}]_l}_{\Wkp[k][\infty][\cube^d]}\\
		&\alplus \leq \norm[\Big]{\act{\apmult}\circ(\act{\Psi_{\eps,\Phi_m}},[\actbig{\wtilde\Psi_{\eps,\Phi}}]_2) -\act{\Psi_{\eps,\Phi_m}}\cdot[\actbig{\wtilde\Psi_{\eps,\Phi}}]_2}_{\Wkp[k][\infty][\cube^d]}\\
		&\alplus \alplus+\norm[\Big]{\act{\Psi_{\eps,\Phi_m}}\left([\actbig{\wtilde\Psi_{\eps,\Phi}}]_2-[\act{\Phi}]_{m+1}\right)}_{\Wkp[k][\infty][\cube^d]}\\ &\alplus\alplus \alplus+\norm*{[\act{\Phi}]_{m+1}\cdot\big(\act{\Psi_{\eps,\Phi_m}}-\prod_{l=1}^{m} [\act{\Phi}]_l\big)}_{\Wkp[k][\infty][\cube^d]}.\numberthis\label{eq:induction_add_zero}
	\end{align*}
	We continue by considering the first term of the Inequality \eqref{eq:induction_add_zero} and bound the $k$-semi-norm of this term. We apply the chain rule from Corollary~\ref{cor:composition_norm} for $g:\R^2\to\R$ with $g(x,y)=\act{\apmult}(x,y)-x\cdot y$ and $f:\R^d\to\R^2$ with $f=\act{\wtilde\Psi_{\eps,\Phi}}$. We get
	\begin{align*}
		&\pabs[\Big]{\act{\apmult}\circ(\act{\Psi_{\eps,\Phi_m}},[\actbig{\wtilde\Psi_{\eps,\Phi}}]_2) -\act{\Psi_{\eps,\Phi_m}}\cdot[\actbig{\wtilde\Psi_{\eps,\Phi}}]_2}_{\Wkp[k][\infty][\cube^d]}\\
		&\alplus\leq C\sum_{i=1}^k\pabs{\act{\apmult}(x,y)-x\cdot y}_{\Wkp[i][\infty][\intervalo{-B}{B}^2;dxdy]} N^{k+\drate}\\
		&\alplus\leq  Ck\cdot\norm{\act{\apmult}(x,y)-x\cdot y}_{\Wkp[j][\infty][\intervalo{-B}{B}^2;dxdy]} N^{k+\drate}\\
		&\alplus\leq C\eps N^{k+\drate},\numberthis\label{eq:1term} 
	\end{align*}
	where we used the induction hypothesis together with $\pabs{[\actbig{\wtilde\Psi_{\eps,\Phi}}]_2}_{\Wkp[k][\infty][\cube^d]}\leq c N^{k+\drate}$ (which follows from the properties of the approximate identity network from Corollary~\ref{cor:ApproxIdent} together with the chain rule) in the third step and assumed that $c\leq C$. Combining the statements of the semi-norms then yields the required bound for the norm. For the second term we have by the product rule and the chain rule 
	\begin{align*}
	&\norm[\Big]{\act{\Psi_{\eps,\Phi_m}}\left([\actbig{\wtilde\Psi_{\eps,\Phi}}]_2-[\act{\Phi}]_{m+1}\right)}_{\Wkp[k][\infty][\cube^d]}\\
		&\alplus \leq \sum_{i=0}^k\norm{\act{\Psi_{\eps,\Phi_m}}}_{\Wkp[i][\infty][\cube^d]}\cdot\norm*{[\actbig{\wtilde\Psi_{\eps,\Phi}}]_2-[\act{\Phi}]_{m+1}}_{\Wkp[k-i][\infty][\cube^d]}\\
		&\alplus\leq \sum_{i=0}^k c N^{i+\dratej{i}}\cdot C\eps N^{k-i+\dratej{k-i}}\leq kcCN^{k+\drate}\eps.\numberthis\label{eq:2term}
	\end{align*}
	
	To estimate the last term of~\eqref{eq:induction_add_zero} we apply the product rule from Lemma~\ref{lemma:product_rule_bound} and get
	\begin{align*}
		&\norm*{[\act{\Phi}]_{m+1}\cdot\Big(\act{\Psi_{\eps,\Phi_m}}-\prod_{l=1}^{m} [\act{\Phi}]_l\Big)}_{\Wkp[k][\infty][\cube^d]}\\
		&\alplus \leq \sum_{i=0}^k\norm{[\act{\Phi}]_{m+1}}_{\Wkp[i][\infty][\cube^d]}\cdot\norm*{\act{\Psi_{\eps,\Phi_m}}-\prod_{l=1}^{m} [\act{\Phi}]_l}_{\Wkp[k-i][\infty][\cube^d]}\\
		&\alplus\leq \sum_{i=0}^k c N^{i+\dratej{i}}\cdot C N^{k-i+\dratej{k-i}}\eps\leq kcCN^{k+\drate}\eps.\numberthis\label{eq:3term}
	\end{align*}
	For the second step, we used again the induction hypothesis together with $$\pabs{[\act{\Phi}]_{m+1}}_{\Wkp[k][\infty][\cube^d]}\leq c N^{k+\drate}.$$ Combining~\eqref{eq:induction_add_zero} with~\eqref{eq:1term}, \eqref{eq:2term} and~\eqref{eq:3term} yields 
	\begin{equation*}
		\norm*{\act{\Psi_{\eps,\Phi}}-\prod_{l=1}^{m+1} [\act{\Phi}]_l}_{\Wkp[k][\infty][\cube^d]}\leq C N^{k+\drate}\eps.
	\end{equation*}

	\textbf{ad (iii):} The estimate 
	\[
		\pabs{\act{\Psi_{\eps,\Phi}}}_{\Wkp[k][\infty][\cube^d]}\leq CN^{k+\drate}.
		\]
can be shown similarly as above.	

	\textbf{ad (iv):} Finally, we need to derive a bound for the absolute values of the weights. From the definition of $\Psi_{\eps,\Phi}$ together with~Lemma~\ref{lemma:one_layer_concat}~(iii) we get
	\[\bweights{\Psi_{\eps,\Phi}}=\bweights{\apmult\conc\wtilde \Psi_{\eps,\Phi}}\leq C\cdot\max\{\eps^{-2}, \bweights{\wtilde\Psi_{\eps,\Phi}}\}.
	\]
	From $\bweights{\wtilde\Psi_{\eps,\Phi}}\leq C\max\{\bweights{\Phi},\eps^{-2}\}$ (see Step 2) it follows that $\bweights{\Psi_{\eps,\Phi}}\leq C\max\{\bweights{\Phi},\eps^{-2}\}$. This concludes the proof.
\end{proof}

In the last part of this subsection, we are finally in a position to construct neural networks which approximate sums of localized polynomials.
\begin{lemma}\label{lemma:network_polynomial_approximation}
	Let $j,\tau\in\N_{0},$, $d,N\in \N$,  $k\in\{0,\ldots,j\}$, Additionally, let $\varrho$ be such that it fulfills the assumptions of Proposition \ref{prop:MonApp} (for $n=3,$ $r=2$). Let $n\in\N_{\geq k+1}$, $1\leq p\leq\infty$, and $\mu>0$. 
	Assume that $\left(\Psi^{(j,\tau,N,s)}\right)_{N\in\N,s\geq 1}$ be the exponential (respectively polynomial, exact) $(j,\tau)$-PU from Definition \ref{def:partition_of_unity}. For $N\in\N,$ set 
    \begin{align*}
        s\coloneqq \begin{cases} N^\mu,&\text{if exponential PU,}\\
        N^{\frac{2d/p+d+n}{D}},&\text{if polynomial PU,}\\
		1, &\text{if exact PU},
		\end{cases}
    \end{align*}
	  Then, there is a constant $C=C(n,d,p,k)>0$ with the following properties:
	  
	Let $\eps\in \epsin$, $f\in\Wkp[n][p][\cube^d]$ and $p_m(x)\coloneqq p_{f,m}(x)=\sum_{\pabs{\alpha}\leq n-1} c_{f,m,\alpha}x^\alpha$ for $m\in\MNd$ be the polynomials from Lemma~\ref{lemma:polynomial_approximation}. Then there is a neural network $\Phi_{P,\eps}=\Phi_{P,\eps}(f,d,n,N,\mu,\eps)$ with $d$-dimensional input and one-dimensional output, with at most $C$ layers and $C (N+1)^d$ nonzero weights, such that  
		\[
		\norm*{\sum_{m\in\MNd}\phi^s_m p_m -\act{\Phi_{P,\eps}}}_{\Wkp[k][p][\cube^d]}\leq C \norm{f}_{\Wkp[n][p][\cube^d]} \eps,
		\]
		and $\bweights{\Phi_{P,\eps}}\leq C\norm{f}_{\Wkp[n][p][\cube^d]} \eps^{-2}s^2 N^{2(d/p+d+ k)+d/p + d}$.
		
\end{lemma}

\begin{proof} As before, we only provide the proof only for the case of an exponential $(j,\tau)$-PU.

	\textbf{Step 1 (Approximating localized monomials $\phi_m^s(x) x^\alpha$)}: Let $|\alpha|\leq n-1,$ $m \in \{0, \dots, N \}^d$ and set $\widetilde \eps\coloneqq \eps N^{-(d/p +d+ k+\drate)}$. By Corollary \ref{cor:ApproxIdent} and inductively using the trick that $|xy-uz|\leq |x(y-z)|+|z(x-u)|$, there is a neural network $\Phi_\alpha$ with $d$-dimensional input and $\pabs{\alpha}$-dimensional output, with two layers, at most $4(n-1)$ nonzero weights bounded in absolute value by $C\widetilde{\eps}^{-1}$ 
	such that
	\begin{equation}\label{eq:inductive_product}
		\norm[\big]{x^\alpha-\prod_{l=1}^{\pabs{\alpha}}[\act{\Phi_\alpha}]_l(x)}_{W^{k,\infty}((0,1)^d;dx)}\leq C\widetilde{\eps}
	\end{equation}
		and
	\begin{equation}\label{eq:phi_poly_bound}
		\norm{[\act{\Phi_\alpha}]_l}_{\Wkp[k][\infty][\cube^d]}\leq \widetilde{\eps}+ 1\leq 2, \quad \text{for all }l=1,\ldots,\pabs{\alpha}. 
	\end{equation}
	Let now $\Phi_m$ be the neural network from Lemma~\ref{prop:partition_of_unity} (\ref{item:pou_networkGeneralSigmoid}) (for $s=N^\mu$) and define the network
	\[
		\Phi_{m,\alpha}\coloneqq \mathrm{P}(\Phi_m,\Phi_\alpha,\Phi_{n-1-|\alpha|,2}),
		\]
		where the parallelization is provided by Lemma \ref{lem:ParalSame} and $\Phi_{n-1-|\alpha|,2}=\left((0_{d,d},0_d),(0_{n-1-|\alpha|,d},1_{n-1-|\alpha|}) \right)$. Consequently, $\Phi_{m,\alpha}$ has $2\leq K_0$ layers and $C+4(n-1)\leq K_0$ nonzero weights for a suitable constant $K_0=K_0(n,d)\in\N$, $\bweights{\Phi_{m,\alpha}}\leq C \max\{\widetilde{\eps}^{-1},N^{1+\mu}\}$ and $\norm{\prod_{l=1}^{n-1+d}[\act{\Phi_{m,\alpha}}(x)]_l-\phi_m^s(x) x^\alpha}_{W^{k,\infty}((0,1)^d);dx}\leq C\widetilde{\eps}$. Moreover, as a consequence of Lemma~\ref{prop:partition_of_unity} (\ref{item:pou_networkGeneralSigmoid}) together with Equation~\eqref{eq:phi_poly_bound} we have 
	\begin{equation*}
		\norm{[\act{\Phi_{m,\alpha}}]_l}_{\Wkp[k][\infty][\cube^d]}\leq C N^{k+\drate},\quad\text{for all }l=1,\ldots,n-1+d.
	\end{equation*}
	To construct an approximation of the localized monomials $\phi_m^s(x) x^\alpha$, let $\Psi_{\widetilde \eps,(m,\alpha)}$ be the neural network provided by Lemma~\ref{lemma:network_multiplikation} (with $\Phi_{m,\alpha}$ instead of $\Phi$, $m=\pabs{\alpha}+d\in\N$, $K=K_0\in\N$) for $m\in\{0,\ldots,N\}^d$ and $\alpha\in\N_0^d,\pabs{\alpha}\leq n-1$. Then  $\Psi_{\widetilde \eps,(m,\alpha)} $ has at most $C$ layers (independently of $m,\alpha$), number of nonzero weights and $\bweights{\Psi_{\widetilde \eps,(m,\alpha)}}\leq C\max\{N^{1+\mu},\eps^{-2}N^{2(d/p+d+ k+\drate)}\}$. Moreover, 
	\begin{align*}
		&\norm*{\phi_m^s(x) x^\alpha -\actbig{\Psi_{\widetilde \eps,(m,\alpha)}}(x)}_{\Wkpd[k][\infty][\cube^d]{x}}\\
		&\leq\norm[\big]{\phi_m^s(x) x^\alpha-\prod_{l=1}^{n-1+d}[\act{\Phi_{m,\alpha}}(x)]_l}_{\Wkpd[k][\infty][\cube^d]{x}} + \norm[\big]{\prod_{l=1}^{n-1+d}[\act{\Phi_{m,\alpha}}]_l- \actbig{\Psi_{\widetilde \eps,(m,\alpha)}}}_{\Wkp[k][\infty][\cube^d]}\\
		&\leq C N^{k+\drate} \widetilde \eps\leq C\eps N^{-d/p-d},
	\end{align*} 
	where we used Equation~\eqref{eq:inductive_product} together with the product rule for the last step.

\textbf{Step 2 (Constructing $\Phi_{P,\eps}$)}:
We set 
	\[
		T\coloneqq\pabs{\{(m,\alpha):m\in\MNd, \alpha\in\N_0^d,\pabs{\alpha}\leq n-1\}}.
		\]
		We note that every network $\Psi_{\widetilde{\epsilon},(m,\alpha)}$ has the same number of layers and, by using  Lemma~\ref{lem:ParalSame}, we parallelize the localized polynomial approximations 
	\[
		\mathrm{P}\big(\Psi_{\widetilde \eps,(m,\alpha)}:m\in \{0,\dots,N\}^d, \alpha \in \N_0^d,~|\alpha|\leq n-1 \big)
		\]
		and note that the resulting network has at most $C$ layers and $C T$ nonzero weights bounded in absolute value by $C\max\{N^{1+\mu},\eps^{-2}N^{2(d/p+d+k+\drate)}\}\leq C \eps^{-2}N^{2(d/p+d+ k+\drate)}$. Next, we define the matrix $\As\in\R^{1,T}$ by $\As\coloneqq[c_{f,m,\alpha} :m\in\MNd, \alpha\in\N_0^d,\pabs{\alpha}\leq n-1]$ and the neural network $\Phis\coloneqq((\As,0))$. Finally, we set
	\begin{equation}\label{eq:final_network}
		\PhiPs\coloneqq\Phis\conc\mathrm{P}\big(\Psi_{\widetilde \eps,(m,\alpha)}:m\in\MNd, \alpha\in\N_0^d,\pabs{\alpha}\leq n-1\big).
		\end{equation}
		 From Lemma~\ref{lemma:one_layer_concat}(\ref{item:ConcVector}) we get  $\PhiPs$ is a neural network with $d$-dimensional input and one-dimensional output, with at most $C$ layers and, by Lemma~\ref{lemma:one_layer_concat}, $C T \leq C (N+1)^d$ nonzero weights. For the absolute values of the weights it holds that 
		 \begin{align*}
		 \bweights{\PhiPs}&\leq (N+1)^dC\norm{f}_{\Wkp[n][p][\cube^d]}N^{d/p} \eps^{-2}N^{2(d/p+d+ k+\drate)}\\
		 &\leq C\norm{f}_{\Wkp[n][p][\cube^d]} \eps^{-2}N^{2(d/p+d+ k+\drate)+d/p + d}
		 \end{align*}
		 where we used the bound for the coefficients $c_{f,m,\alpha}$ from~Remark~\ref{remark:polynomial_coefficients}. Moreover, we have
		\[
			\act{\PhiPs}=\sum_{m\in\MNd}\sum_{|\alpha|\leq n-1}c_{f,m,\alpha}\Realization(\Psi_{\widetilde \eps,(m,\alpha)}).
		\]
	Note that the network $\Phi_{P,\eps}$ only depends on $p_{f,m}$ (and thus on $f$) via the coefficients $c_{f,m,\alpha}$. 
	
	\textbf{Step 3 (Estimating the approximation error in $\norm{\cdot}_{W^{k,p}}$)}:
		We get 
		\begin{align*}
		&\norm*{\sum_{m\in\MNd}\phi_m^s(x) p_m(x) -\act{\Phi_{P,\eps}}(x)}_{\Wkpd[k][p][\cube^d]{x}}\\
		&\alplus=\norm*{\sum_{m\in\MNd}\sum_{|\alpha|\leq n-1}c_{f,m,\alpha}\Big(\phi_m^s(x) x^\alpha -\actbig{\Psi_{\widetilde \eps,(m,\alpha)}}(x)\Big)}_{\Wkpd[k][p][\cube^d]{x}}\\
			&\alplus \leq\sum_{m\in\MNd}\sum_{|\alpha|\leq n-1}\pabs{c_{f,m,\alpha}}\norm*{\phi_m^s(x) x^\alpha -\actbig{\Psi_{\widetilde \eps,(m,\alpha)}}(x)}_{\Wkpd[k][p][\cube^d]{x}}\\
			&\alplus \leq\sum_{m\in\MNd}\sum_{|\alpha|\leq n-1}\norm{\widetilde f}_{\Wkp[n-1][p][\Omega_{m,N}]}N^{d/p} C \eps N^{-d/p-d},
	\end{align*}
	where we used again the bound for the coefficients $c_{f,m,\alpha}$ together with $\norm{\cdot}_{\Wkp[k][p][\cube^d]}\leq C \norm{\cdot}_{\Wkp[k][\infty][\cube^d]}$ in the last step. Similar as in~Equation~\eqref{eq:2dtrick} we finally have
		\begin{align*}
		\norm*{\sum_{m\in\MNd}\phi_m^s(x) p_m(x) -\act{\Phi_{P,\eps}}(x)}_{\Wkpd[k][p][\cube^d]{x}}
		&\leq C \eps N^{-d} \sum_{m\in\MNd}\norm{\widetilde f}_{\Wkp[n][p][\cube^d]}\\
		&\leq C \eps \norm{f}_{\Wkp[n][p][\cube^d]}.
	\end{align*}
	This concludes the proof.
\end{proof}

\subsection{Putting Everything Together}\label{app:PutTogether}
Now we conclude the proof of Proposition \ref{prop:main}. Again, we only provide the proof for exponential $(j,\tau)$-PUs. The rest follows in a similar manner by adapting the calculations to come accordingly.
\begin{proof}[Proof of Proposition~\ref{prop:main}] We divide the proof into two steps: First, we approximate the function $f$ by a sum of localized polynomials. Afterwards, we proceed by approximating this sum by a neural network.

	For the first step, we set
	\begin{equation}\label{eq:definition_N}
		N\coloneqq\ceil*{\left(\frac{\eps}{2\widetilde{C}}\right)^{-1/(n-k-\drate)}}\quad\text{and}\quad s\coloneqq N^\mu,
	\end{equation}
	where $\widetilde{C}=\widetilde{C}(n,d,p,k)>0$ is the constant from Lemma~\ref{lemma:polynomial_approximation}. Without loss of generality we may assume that $\widetilde{C}\geq 1$. The same lemma yields that if $\Psi^{(j,\tau)}=\Psi^{(j,\tau)}(d,N,\mu)=\left\{\phi_m^s:m\in\MNd\right\}$ is the PU from Lemma~\ref{prop:partition_of_unity} and $\widetilde N =\widetilde N(d,p,\mu,k)$ is the constant from Lemma~\ref{lemma:polynomial_approximation}, then there exist polynomials $p_m(x)=\sum_{\pabs{\alpha}\leq n-1}c_{f,m,\alpha}x^\alpha$ for $m\in\{0,\ldots,N\}^d$ such that
	\begin{align*}
		\norm*{f-\sum_{m\in\{0,\ldots,N\}^d}\phi_m^s p_m}_{\Wkp[k][p][\cube^d]}&\leq \widetilde{C}\left(\frac{1}{N}\right)^{n-k-\drate}
		\leq \widetilde{C} \frac{\eps}{2\widetilde{C}}=\frac{\eps}{2},\numberthis\label{eq:final_triangle_1}
	\end{align*}
	
	for all $\eps\in (0,\widetilde \eps$), where $\widetilde \eps=\widetilde \eps(d,p,\mu,k)>0$ is chosen such that $N\geq \widetilde N$.

	For the second step, let $\widetilde{C}'=\widetilde{C}'(n,d,p,k)$ be the constant from Lemma~\ref{lemma:network_polynomial_approximation} and $\Phi_{P,\eps}$ be the neural network provided by Lemma~\ref{lemma:network_polynomial_approximation} with $\eps/(2\widetilde{C}')$ instead of $\eps$. Then $\Phi_{P,\eps}$ has at most $\widetilde{C}'$ layers and at most
	\begin{align*}
		&\widetilde{C}'\left(\left(\frac{\eps}{2\widetilde{C}'}\right)^{-1/(n-k-\drate)}+2\right)^d
		\leq \widetilde{C}' 3^d \left(\frac{\eps}{2\widetilde{C}'}\right)^{-d/(n-k-\drate)}
		\leq C\eps^{-d/(n-k-\drate)}
	\end{align*}
		nonzero weights. In the first step we have used $(2\widetilde{C}')/\eps\geq 1$. The weights are bounded in absolute value by
		\begin{align*}
		    \bweights{\Phi_{P,\eps}}&\leq \widetilde{C}'\eps^{-2}N^{2(d/p+d k+\drate)+d/p + d}\\ &\leq C\eps^{-2-(2(d/p+d+ k+\drate)+d/p+d)/(n-k-\drate)}=C\eps^{-\theta},
		\end{align*}
				
		for a suitable $\theta=\theta(d,p,k,n,\mu)>0$.
		Additionally, there holds
	\begin{equation}\label{eq:final_triangle_2}
		\norm*{\sum_{m\in\MNd}\phi_m^s p_m -\act{\Phi_{P,\eps}}}_{\Wkp[k][p][\cube^d]}\leq \widetilde{C}' \frac{\eps}{2 \widetilde{C}'}\leq \frac{\eps}{2}.
		\end{equation}
 By applying the triangle inequality as well as Equations~\eqref{eq:final_triangle_1} and \eqref{eq:final_triangle_2} we arrive at
		\begin{align*}
			&\norm*{f-\act{\Phi_{P,\eps}}}_{\Wkp[k][p][\cube^d]}\\
			&\alplus\leq \norm*{f-\sum_{m\in\MNd}\phi_m^s p_m}_{\Wkp[k][p][\cube^d]}+\norm*{\sum_{m\in\MNd}\phi_m^s p_m-\act{\Phi_{P,\eps}}}_{\Wkp[k][p][\cube^d]}\\
			&\alplus \leq\frac{\eps}{2}+\frac{\eps}{2}=\eps,
		\end{align*}
		thereby concluding the proof.
\end{proof}

\section{Proof of Theorem~\ref{thm:main} (Encodability of the Weights)}\label{app:Encod}

We now proceed with the proof of  Theorem {\rm\ref{thm:main}}. 
\begin{proof}[Proof of Theorem {\rm\ref{thm:main}}]
 Let $C=C(d,n,p, \mu,k)>0$, $\theta=\theta(d,n,p,k, \mu)>0$ and $\widetilde \eps=\widetilde \eps(d,p,\mu,k)>0$ be the constants from Proposition~\ref{prop:main} and let $\epsilon \in (0,\min\{1/3,\widetilde{\epsilon}\})$. Moreover, for $f\in \Fndp$, let $$\Phi_{\eps,f}\coloneqq ((A_{\mathrm{sum}},0))\conc\mathrm{P}\big(\Psi_i:i=1,\ldots,T\big)$$
be the neural network from Proposition~\ref{prop:main} (defined in Equation~\eqref{eq:final_network}) with at most $L$ layers and $M(\Phi_{\eps,f})\leq C\cdot\eps^{-d/(n-k-\drate)}$ nonzero weights bounded in absolute value by $C\eps^{-\theta}$, such that 
	\[
		\norm{\act{\Phi_{\eps,f}} - f}_{\Wkp[k][p][\cube^d]}\leq \frac{\eps}{3}.
	\]
We will make use of the following additional properties of $\Phi_{\eps,f}$:
\begin{enumerate}[(i)]
    \item Only the entries of $A_{\mathrm{sum}}$ depend on the function $f$. In other words, the entries of $\Psi_1,\ldots, \Psi_T$ are independent from $f$. They only depend on $\eps,n,d,p,k,\mu$.
    \item There exists $s=s(k,n,d,p)>0$ (we assume w.l.o.g.\ that the same $s$ can be used) such that
    \begin{enumerate}[(a)]
        \item $\norm*{\act{\Psi_i}}_{\Wkp[k][\infty][\cube^d]}\leq \eps^{- s}$ for $i=1,\ldots, T$. This follows from Lemma~\ref{lemma:network_multiplikation}~(\ref{item:bounding_range_nnmult}) in combination with Step~1 and~2 of the proof of Lemma~\ref{lemma:network_polynomial_approximation} and choice of $N$ in~Equation~\eqref{eq:definition_N}.
        \item $T\leq \eps^{-s}$. This follows from the definition of $T$ (see Step~2 of the proof of~Lemma~\ref{lemma:network_polynomial_approximation});
        \item $M(\Phi_{\eps,f})\leq \eps^{-s}.$
    \end{enumerate}
    \item $A_{\mathrm{sum}}=(a_m)_{m=1}^T\in\R^{1,T}$.
    \item The last layer $(A_\mathrm{last},b_\mathrm{last})$ of $\mathrm{P}\big(\Psi_i:i=1,\ldots,T\big)$ has a block diagonal structure, where each block is a vector (see also~Lemma~\ref{lem:ParalSame}). Thus, in every column of $A_\mathrm{last}$ there is at most one nonzero entry.
\end{enumerate}	

We replace the weights in the last layer of $\Phi_{\eps,f}$ by elements from an appropriate set of weights with cardinality bounded polynomially in $\eps^{-1}$ and show that the resulting network is still close enough to $f$. Afterwards, we construct a coding scheme for the entire set of weights.

\textbf{Step 1 (Rounding the weights in $A_{\mathrm{sum}}$)}: 
  We now show that with rounding precision~$\nu\coloneqq 2s+2$ we have for the neural network $$\widetilde{\Phi}_{\eps,f}^{(1)}\coloneqq ((\widetilde A_{\mathrm{sum}},0))\conc\mathrm{P}\big(\Psi_i:i=1,\ldots,T\big)$$ where $\widetilde{A}_{\mathrm{sum}}\in ([-\eps^{-\theta},\eps^{-\theta}]\cap\eps^\nu\Z)^{1, T}$ is the rounded weight matrix $A_{\mathrm{sum}}\in\R^{1, T}$ that 
 \begin{align*}
     \norm{\Realization(\Phi_{\eps, f})-\Realization(\widetilde{\Phi}_{\eps, f}^{(1)})}_{W^{k,p}((0,1)^d)}\leq \eps/3.
 \end{align*}
 Clearly,
 \begin{align*}
     &\norm*{\actbig{(A_{\mathrm{sum}},0)\conc\mathrm{P}\big(\Psi_i:i=1,\ldots,T\big)}-\actbig{(\widetilde A_{\mathrm{sum}},0)\conc\mathrm{P}\big(\Psi_i:i=1,\ldots,T\big)}}_{W^{k,p}((0,1)^d)}\\
     &\alplus\leq \left\|{\sum_{i=1}^T a_i \act{\Psi_i} - \sum_{i=1}^T \widetilde a_i \act{\Psi_i} }\right\|_{\Wkp[k][\infty][\cube^d]}\\
&\alplus\leq \sum_{i=1}^T \pabs{a_i-\widetilde a_i} \norm*{\act{\Psi_i}}_{W^{k,\infty}((0,1)^d)} \\
\expl{rounding precision is $\eps^\nu$}&\alplus\leq \sum_{i=1}^T\eps^{\nu}\norm*{ \act{\Psi_i}}_{W^{k,\infty}((0,1)^d)}  \\
\expl{(ii) and (iii) above}&\alplus\leq\eps^{\nu}\eps^{-s}  \eps^{-s}\leq \eps^2\leq \eps/3.
\end{align*}
To get our final network, we replace the bias term $\wtilde A_{\mathrm{sum}} b_\mathrm{last}$ (which is also bounded in absolute value by~$\eps^{-\theta}$) in the last layer of $\wtilde \Phi_{\eps,f}^{(1)}$ by the nearest element in $[-\eps^{-\theta},\eps^{-\theta}]\cap\eps^\nu\Z$ and denote the resulting network by $\wtilde \Phi_{\eps,f}$. It now easily follows that $\norm{\act{\wtilde \Phi_{\eps,f}^{(1)}}-\act{\wtilde \Phi_{\eps,f}}}_{\Wkp[k][\infty][\cube^d]}\leq \eps/3$
which implies by the triangle inequality that $\norm{f-\Realization(\widetilde{\Phi}_{\eps,f})}_{W^{k,p}((0,1)^d)} \leq \eps$.

\textbf{Step 2 (Construction of coding scheme)}: We will now show that there is a constant $C_2=C_2(d,n,p,k,\mu)>0$ and a coding scheme $\Bcal = (B_\ell)_{\ell\in\N}$ such that for each $\eps>0$ and each $f\in\Fndp$ the nonzero weights of $\widetilde{\Phi}_{\eps,f}$ are in $\ran B_{\ceil{C_2\log(1/\eps)}}$. 
 
     If we denote by $W_\eps$ the collection of  nonzero weights of $(\Psi)_{m=1}^T$ (which are independent of $f$), then we have $\pabs{W_\eps}\leq M(\Phi_{\eps,f})\leq\eps^{-s}$. Furthermore, we have $\pabs{\mesh}=2\floor{\eps^{-\theta-\nu}}+1\leq\eps^{-s_2}$ with $s_2\coloneqq \theta+\nu+2$. 
     \begin{itemize}
         \item The matrix weights in the last layer of $\wtilde \Phi_{\eps,f}$ are in the set $G_{\mathrm{mult}}\coloneqq\{x_1x_2: x_1\in W_\eps, x_2\in \mesh\}$ with cardinality bounded by $\eps^{-(s+s_2)}$.
         \item The bias in the last layer is an element of $\mesh$.
         \item The weights of $\wtilde \Phi_{\eps,f}$ in the layers $1,\ldots, L-1$ are in the set $W_\eps$.
     \end{itemize}
 Setting $C_2\coloneqq 2(s+s_2)$ it follows that there exists a surjective mapping 
 $$B_{\ceil{C_2 \log_2(1/\eps)}}: \{0,1\}^{\ceil{C_2 \log_2(1/\eps)}} \to G_{\mathrm{mult}}\cup W_\eps\cup\left(\mesh\right),$$
 which shows the claim.
\end{proof}
 
 \section{PU-properties of the Activation Functions from Table \ref{tab:ActFunctions}}\label{sec:activation_admissibility}
 
 In this section, we examine the PU-properties of the activation functions listed in Table \ref{tab:ActFunctions}.
 
 The  smoothness properties of all functions in Table \ref{tab:ActFunctions} are clear. In particular, all functions are in $C^{\infty}(\R\setminus\{0\})$.
 
 In order to show that the activation functions to follow allow for exponential (respectively polynomial) PUs, we consider the exponential (respectively polynomial)  $(j,\tau)$ admissibility conditions of Definition \ref{def:Admissible}.
  \noindent \underline{\textbf{Exact PUs.}}
      
      \begin{itemize}
      \item[\textbf{(leaky) ReLU and RePUs:}] These functions admit exact PUs. For the ReLU case, see for instance \cite{yarotsky2017error,guhring2019error}. For RePUs, this follows from the properties of B-splines (see \cite[Chapter IX]{DeBoor}).
      \end{itemize}
      
    \noindent \underline{\textbf{Exponential PUs.}}

            \begin{itemize}
      \item[\textbf{$\mathrm{ELU}_a$ for $a>0,a\neq 1$:}] Here, $j=1,$ $\tau=1,$ $A=0$ and $B=1.$ Moreover, $R>0$ can be chosen arbitrarily. Then, for $D=1,$ we have, for all $x>R,$ that $|1-\varrho'(x)|=|1-1|=0$ and, for all $x<-\R$ that $|\varrho'(x)|=|ae^x|=ae^{Dx}.$ 
       \item[\textbf{$\mathrm{ELU}_1$:}] Here, $j=2,$ $\tau=1,$ $A=0$ and $B=1.$ Moreover, $R>0$ can be chosen arbitrarily. Then, for $D=1,$ we have, for all $x>R,$ that $|1-\varrho'(x)|=|1-1|=0$ and, for all $x<-\R$ that $|\varrho'(x)|=|e^x|=e^{Dx}.$
       Moreover, we have for all $|x|>R$ that $|\varrho''(x)|\leq e^{-|x|}= e^{-D|x|}.$
       \item[\textbf{sigmoid:}] Here, $j\in\N_0$ is arbitrary, $\tau=0,$ $A=0$ and $B=1.$ Moreover, $R>0$ can be chosen arbitrarily. Then we have, for all $x>R,$ that $|1-\varrho(x)|\leq e^{-x}$ and, for all $x<-\R$ that $|\varrho(x)|\leq e^x.$
       The other statements follow from the fact that, for the sigmoid activation function, the $k$-th derivative is a finite linear combination of the powers $\varrho,\dots,\varrho^k$ of $\varrho$ (see, e.g., \cite{SigmoidDer}). Choosing $D$ suitably then shows the claim.
       \item[\textbf{tanh:}] Since $\tanh(x)=2\cdot \mathrm{sigmoid}(2x)-1,$ the proof of this statement follows from the proof of the statement for the sigmoid activation function for $A=-1,~B=1.$
       \item[\textbf{softplus:}] Here, $j\in\N_0$ is arbitrary, $\tau=1,$ $A=0$ and $B=1.$ Moreover, $R>0$ can be chosen arbitrarily. Then, for all $x>R,$ there holds $|1-\varrho'(x)|= |1-\mathrm{sigmoid}(x)|\leq e^{-x}$ and, for all $x<-\R$ that $|\varrho'(x)|=|\mathrm{sigmoid}(x)|\leq  e^x.$ The proof of (d.3) for the higher-order derivatives follows from the properties of the higher derivatives of the sigmoid function. 
         \item[\textbf{swish:}] Here, $j\in\N_0$ is arbitrary, $\tau=1,$ $A=0,$ and $B=1.$ It is not hard to see that for all $k\in\N$ there holds
         $$ \mathrm{swish}^{(k)}(x)= x\cdot \mathrm{sigmoid}^{(k)}(x)+ k\cdot \mathrm{sigmoid}^{(k-1)}(x). 
         $$
         
        Now, the statement follows from the analogous observations for the sigmoid function combined with the fact that for $r,u>0$ with $r>u$ there holds
        \begin{align*}
            \lim_{x\to \infty} \frac{xe^{-rx}}{e^{-ux}} =0,\qquad\qquad \lim_{x\to -\infty} \frac{xe^{rx}}{e^{ux}} =0 .
        \end{align*}
  \end{itemize}
  
 \noindent \underline{\textbf{Polynomial PUs.}}

  \begin{itemize}
      \item[\textbf{softsign:}] Here, $j\in\N_0$ is arbitrary, $\tau=0,$ $A= -1,$ $B=1.$  The polynomial converence properties (d.1)-(d.3) follow immediately from the definition of the softsign function. 
      \item[\textbf{inverse square root linear unit:}] Here, $j=3$, $\tau=1$, $A=0$ and $B=1$. The polynomial converence properties (d.1)-(d.3) follow immediately from the definition of the inverse square root linear unit.
      \item[\textbf{inverse square root unit:}]  Here, $j\in\N_0$ is arbitrary, $\tau=0$, $A=-1$ and $B=1$. The polynomial converence properties (d.1)-(d.3) follow immediately from the definition of the inverse square root  unit.
      \item[\textbf{arctan: }] Here, $j\in\N_0$ is arbitrary, $\tau=0$, $A=-\pi/2$ and $B=\pi/2$. The polynomial converence properties (d.1)-(d.3) follow immediately from the fact that $\varrho'(x)=1/(1+x^2)$ which in particular implies polynomial convergence behavior for $\arctan$ itself.
  \end{itemize}

\end{document}